\documentclass{article}

\usepackage[english]{babel}

\usepackage[letterpaper,top=2cm,bottom=2cm,left=3cm,right=3cm,marginparwidth=1.75cm]{geometry}

\usepackage{bm}
\usepackage{amsmath}
\usepackage{amssymb}
\usepackage{amsthm}
\usepackage{graphicx}
\usepackage{parallel}
\usepackage{ stmaryrd }
\usepackage{tikz}
\usepackage{multicol}
\usepackage{ stmaryrd }
\usepackage{ latexsym }
\usepackage[shortlabels]{enumitem}
\usepackage{hyperref}
\hypersetup{colorlinks,allcolors=blue,breaklinks=true}

\newcommand{\mf}{{\mathfrak{m}}{}}
\newcommand{\rf}{{\rho}{}}
\newcommand{\iterCond}{{$(\triangle)$}}

\newcommand{\forkindep}[1][]{
  \mathrel{
    \mathop{
      \vcenter{
        \hbox{\oalign{\noalign{\kern-.3ex}\hfil$\vert$\hfil\cr
              \noalign{\kern-.7ex}
              $\smile$\cr\noalign{\kern-.3ex}}}
      }
    }\displaylimits_{#1}
  }
}

\newcommand{\mipo}{\operatorname{MiPo}}
\newcommand{\Char}{\operatorname{Char}}

\newcommand{\NN}{\mathbb{N}}

\newcommand{\QQ}{\mathbb{Q}}
\newcommand{\RR}{\mathbb{R}}
\newcommand{\FF}{\mathbb{F}}
\newcommand{\ZZ}{\mathbb{Z}}
\newcommand{\VV}{\mathbb{V}}

\newcommand{\QQp}[1]{\mathbb{Q}[X]_{\operatorname{irr}}^{#1}}
\newcommand{\Kp}[1]{K[X]_{\operatorname{irr}}^{#1}}
\newcommand{\Cc}{{\mathcal{C}}}
\newcommand{\Ccalg}{{\mathcal{C}^{\operatorname{alg}}}}
\newcommand{\Cctrans}{{\mathcal{C}^{\operatorname{tr}}}}

\newcommand{\mm}{\mathcal{M}}

\newcommand{\bb}{\mathcal{B}}
\newcommand{\Gg}{\mathcal{G}}
\newcommand{\ii}{\mathcal{I}}
\newcommand{\jj}{\mathcal{J}}
\newcommand{\kk}{\mathcal{K}}
\newcommand{\rr}{\mathcal{R}}

\newcommand{\set}[1]{{\{#1\}}}
\newcommand{\Set}[1]{{\left\{#1\right\}}}
\newcommand{\spanA}[2]{{{\langle#1\rangle_{#2}}}}

\newcommand{\Fac}{{\operatorname{Fac}}}
\newcommand{\Ker}{{\operatorname{Ker}}}
\newcommand{\ACF}{{\operatorname{ACF}}}
\newcommand{\Image}{{\operatorname{Im}}}

\newcommand{\dcl}{{\operatorname{dcl}}}
\newcommand{\rk}{{\operatorname{rk}}}
\newcommand{\cl}{{\operatorname{cl}}}
\newcommand{\Diag}{{\operatorname{Diag}}}
\newcommand{\Id}{{\operatorname{Id}}}
\newcommand{\Th}{{\operatorname{Th}}}
\newcommand{\lcm}{{\operatorname{lcm}}}
\newcommand{\Lex}{{\operatorname{Lex}}}
\newcommand{\acl}{{\operatorname{acl}}}

\newcommand{\ud}{{\underline{d}}{}}

\newcommand{\uu}{{\underline{u}}{}}
\newcommand{\uv}{{\underline{v}}{}}
\newcommand{\uw}{{\underline{w}}{}}
\newcommand{\ux}{{\underline{x}}{}}
\newcommand{\uy}{{\underline{y}}{}}

\newcommand{\uzero}{{\underline{0}}}
\newcommand{\ulambda}{{\underline{\lambda}}}

\newcommand{\li}{{\scalebox{0.5}{{$\operatorname{li}$}}}}
\newcommand{\ld}{{\scalebox{0.5}{{$\operatorname{ld}$}}}}

\newcommand{\lii}{{\scalebox{0.5}{$\operatorname{li}$}}}
\newcommand{\ldd}{{\scalebox{0.5}{$\operatorname{ld}$}}}

\newcommand{\uvvec}{{\underline{\Vec{v}}}}

\newcommand{\uxvec}{{\underline{\Vec{x}}}}

\newcommand{\xvec}{{\Vec{x}}}

\newcommand{\vvec}{{\Vec{v}}}

\newcommand{\Hfour}{{$(\operatorname{H4})$}}
\newcommand{\aclCond}{{$(V)$}}

\newcommand{\TKvs}{{T_{K\operatorname{-vs}}}}
\newcommand{\TKvsThe}{{T_{K\operatorname{-vs},\theta}}}
\newcommand{\TKvsTheC}{{T^C_{K\operatorname{-vs},\theta}}}

\newcommand{\RCF}{{\operatorname{RCF}}}

\newcommand{\tp}{{\operatorname{tp}}}

\newcommand{\LK}{{L_K}}
\newcommand{\NIP}{$\operatorname{NIP}$}

\newcommand{\NATP}{$\operatorname{NATP}$}

\newcommand{\LKThe}{{L_{K,\theta}}}
\newcommand{\LRC}{{L_{R_C}}}

\newcommand{\Lr}{L_{\operatorname{r}}}

\newcommand{\standartConstruction}{\hyperref[lemma_standart_construction]{Standard Construction}}

\newcommand{\placeholderNotation}{\hyperref[def_placeholder_notation]{Placeholder-Notation}}

\newcommand{\drawTextHelper}[5]{
\node[anchor=center, scale = 1] at (#1 * #4 + 0.5 * #4, -#2 * #5 - 0.5 * #5) {$#3$};
}
\newcommand{\drawText}[3]{\drawTextHelper{#1}{#2}{#3}{1.0}{0.6}}

\newcommand{\drawBorderHelper}[6]{\draw [black, line width=0.75]
(#1 * #5 + 0.105, -#2 * #6) --
(#1 * #5, -#2 * #6) --
(#1 * #5, -#4 * #6) --
(#1 * #5 + 0.105, -#4 * #6)
(#3 * #5 - 0.105, -#2 * #6) --
(#3 * #5, -#2 * #6) --
(#3 * #5, -#4 * #6) --
(#3 * #5 - 0.105, -#4 * #6);
\node[anchor=east, scale = 1.0] (Frame) at (0.15, -0.5 * #2 * #6 + -0.5 * #4 * #6 -0.08) {};}
\newcommand{\drawBorder}[4]{\drawBorderHelper{#1}{#2}{#3}{#4}{1.0}{0.6}}

\newcommand{\drawHDotsHelper}[5]{
\draw [line width=1.0, line cap=round, dash pattern=on 0 off 9.5 * #4]
        (#1 * #4, -#2 * #5 - 0.5 * #5) --
        (#1 * #4 + #3 * #4 + 0.02 * #4, -#2 * #5 - 0.5 * #5);    
}
\newcommand{\drawHDots}[3]{\drawHDotsHelper{#1}{#2}{#3}{1.0}{0.6}}

\newcommand{\drawVDotsHelper}[5]{
\draw [line width=1.0, line cap=round, dash pattern=on 0pt off 9.5 * #5]
        (#1 * #4 + 0.5 * #4, -#2 * #5) --
        (#1 * #4 + 0.5 * #4, -#2 * #5 -#3 * #5 - 0.02 * #5);    
}
\newcommand{\drawVDots}[3]{\drawVDotsHelper{#1}{#2}{#3}{1.0}{0.6}}

\newtheorem*{notation}{Notation}
\newtheorem*{remark*}{Remark}

\newtheorem{theorem}{Theorem}[section]
\newtheorem*{theorem*}{Theorem}
\newtheorem{theoremi}{Theorem}

\newtheorem{definition}[theorem]{Definition}
\newtheorem{fact}[theorem]{Fact}
\newtheorem{remark}[theorem]{Remark}
\newtheorem{lemma}[theorem]{Lemma}
\newtheorem{corollary}[theorem]{Corollary}
\newtheorem{example}[theorem]{Example}
\newtheorem{observation}[theorem]{Observation}
\newtheorem{question}[theorem]{Question}

\newtheorem{subdefinition}{Definition}[theorem]
\newtheorem{subclaim}{Claim}[theorem]

\newtheorem*{subdefinition*}{Definition}
\newtheorem*{subclaim*}{Claim}

\newenvironment{innerproof}[1][Proof]
{%
  \begin{proof}[#1]%
}
{%
  \end{proof}%
}

\usepackage{titlesec}
\titleformat{\section}
  {\normalfont\Large\bfseries\boldmath}{\thesection}{1em}{}
\titleformat{\subsection}
  {\normalfont\large\bfseries\boldmath}{\thesubsection}{1em}{}
\titleformat{\subsubsection}
  {\normalfont\normalsize\bfseries\boldmath}{\thesubsubsection}{1em}{}

\title{Model theory of generic vector space endomorphisms III: Reducts}
\author{Leon Chini}

\newcommand{\Addresses}{{
  \bigskip
  \footnotesize

  \textsc{Mathematisches Institut, Universität Bonn, Endenicher Allee 60, D-53115 Bonn, Germany}\par\nopagebreak
  \textit{E-mail address}: 
  \href{mailto:lchini@uni-bonn.de}{\tt lchini@uni-bonn.de}
}}

\begin{document}
\maketitle

\begin{abstract}
\noindent This paper further studies the model companion of an endomorphism acting on a vector space, possibly with extra structure. Let $T$ be a model-complete theory that $\varnothing$-defines an infinite $K$-vector space $\mathbb{V}$. In previous work, we introduced a family $\{T^C_\theta : C \in \mathcal{C}\}$ of extensions of the theory $T_\theta := T \cup \{\text{``$\theta$ is an endomorphism of $\mathbb{V}$''}\}$ that parameterizes all consistent extensions of the form  
$$  
    T_\theta \cup \left\{\sum\nolimits_{k}\bigcap\nolimits_{l}\operatorname{Ker}(\rho_{j, k, l}[\theta]) = \sum\nolimits_{k}\bigcap\nolimits_{l} \operatorname{Ker}(\eta_{j, k, l}[\theta]) : j \in \mathcal{J}\right\},  
$$  
where all sums and intersections are finite, all the $\rho[\theta]$'s and $\eta[\theta]$'s are polynomials over $K$ with $\theta$ plugged in, and $\mathcal{J}$ is some possibly infinite index set. We also presented a sufficient condition that implies that every $T^C_\theta$ has a model companion $T\theta^C$. We simplify our axiomatization of $T\theta^C$ and the criterion for its existence for theories ``close to the theory of $K$-vector spaces''.  
We apply this to the explicit case where $T$ is the pure theory of $K$-vector spaces and characterize all $\varnothing$-definable endomorphisms of $\mathbb{V}$ in this case. Given an existentially closed model $(\mathcal{M}, \theta) \models T^C_\theta$ and a polynomial $\rho\in K[X]$, we show that $(\mathcal{M},\operatorname{Ker}(\rho[\theta]))$ is, unless $\operatorname{Ker}(\rho[\theta]) = \{0\}$ or $\operatorname{Ker}(\rho[\theta]) = \mathbb{V}$, an existentially closed model of $T_V := T \cup \{\text{``$V$ is a vector subspace of $\mathbb{V}$''}\}$. In the same vein, we present a criterion for when $(\mathcal{M}, \rho[\theta])$ is again an existentially closed model of $T^{C'}_\theta$ for some $C' \in \mathcal{C}$.
\end{abstract}

\tableofcontents
\section{Introduction}
This paper is a continuation of \cite{Chi25} and \cite{Chi25b}, which deal with the model companion of an endomorphism acting on a vector space, possibly with extra structure. For the relevance of this line of inquiry and a description of earlier work, see the introduction of \cite{Chi25}. The goal of this paper is to study certain reducts of these model companions.
One often studies reducts of a theory, i.e., the theories obtained by omitting some parts of its language.
In this way, one can, for example, deduce that the original theory is ``wild'' in a certain sense, such as not being stable, \NIP{}, $\operatorname{NTP}_2$, etc., by showing that the reduct is ``wild''.
Conversely, one can often deduce that a reduct is ``tame'' by showing that the original theory is ``tame''.

We now discuss the results from the previous two papers. Let $L$ be a language and $T$ a model-complete $L$-theory with an infinite $\varnothing$-definable $K$-vector space $\VV$ in every model.  
Given a consistent set $C$ of constraints on an endomorphism, which encodes conditions of the form  
$$
    \sum\nolimits_{k}\bigcap\nolimits_{l}\Ker(\rho_{k, l}[\theta]) = \sum\nolimits_{k}\bigcap\nolimits_{l} \Ker(\eta_{k, l}[\theta])
$$  
(where all sums and intersections are finite,  
and all the $\rho[\theta]$'s and $\eta[\theta]$'s are polynomials over $K$ with the endomorphism $\theta$ plugged in), we define the following theory in the language $L_\theta := L \cup \set{\theta}$:  
$$
T^C_\theta := T \cup \set{\text{``$\theta$ is an endomorphism of $\VV$''}} \cup \set{\text{``$\theta$ satisfies the constraints in $C$''}}.
$$
Note that these sets of constraints are a simplification for the sake of this introduction and that we will actually use so-called kernel configurations instead (see Definition \ref{def_kernel_conf}).
In \cite{Chi25}, we showed that $T^C_\theta$ has a model companion $T\theta^C$ if $T$ satisfies a certain condition \Hfour{}, which corresponds to Definition 1.10 in \cite{dEl21b} and basically states that ``$\psi(\ux; \uy)$ implies no finite disjunction of non-trivial linear dependencies in $\ux$ over $\VV$'' can be expressed as an $L$-formula $\sigma_\psi(\uy)$ for every $L$-formula $\psi(\ux; \uy)$. These results, as well as all others relevant to this paper, will be recalled in Section \ref{sec_prelim_res}. In \cite{Chi25b}, we studied definable sets in $T\theta^C$, the completions of this theory, as well as the algebraic closure in its models. Only a few consequences of those results for $T =\TKvs$, the theory of $K$-vector spaces, are relevant to this paper.

“Our first new result simplifies the characterization of existentially closed models of $T^C_\theta$ for theories $T$ that are ``close to the theory of $K$-vector spaces''; see Theorem \ref{theorem_easier_char_acl_cond}. Here, by $T$ being ``close to the theory of $K$-vector spaces'', we mean that:
\begin{enumerate}[(i)]
    \item $T$ eliminates $\exists^\infty x \in \VV$, in the sense of Definition \ref{def_ex_inf_v}.
    \item Given $\mm \prec \mm' \models T$, a tuple $\uv$ of elements of $\VV'$ (the vector space of $\mm'$) is $K$-linearly independent over $\VV$ (the vector space of $\mm$) if and only if $\uv$ is $\acl_L$-independent over $M$, the universe of $\mm$.
\end{enumerate}
We then apply this simplified characterization to the concrete case in which $T$ is $\TKvs$, the theory of $K$-vector spaces; see Theorem \ref{corollary_c_comp}. When $C$ is the empty set of constraints, we obtain that a model of $\TKvsTheC$ is existentially closed if and only if the preimage of $v$ under $\rho[\theta]$ is infinite for every $v \in \VV$ and every $\rho \in K[X]$ with positive degree.

Furthermore, we investigate the endomorphisms of $\VV$ definable in $\TKvs\theta^C$. In \cite{Chi25}, we presented a ring $R_C$ of endomorphisms of $\VV$ that are definable in $T\theta^C$ using only the language of the $K$-vector space and $\theta$; see Fact \ref{theorem_r_c_def} for more details.
This ring induces an $R_C$-module structure on $\VV$ that extends its $K$-vector space structure. We also claimed that this ring has three ``optimality'' properties, but only proved two of these. We prove the last remaining ``optimality'' property in this paper:
\begin{theoremi}[Theorem \ref{theorem_homomorphism_LKthe}]
    Every endomorphism of the vector space $\VV$ that is $\varnothing$-definable in $\TKvs\theta^C$ is already an element of the ring $R_C$.
\end{theoremi}
\noindent Given $\rho \in K[X]$, we also investigate what happens if one replaces $\theta$ with a predicate for $\Ker(\rho[\theta])$ or $\Image(\rho[\theta])$:
\begin{theoremi}[Theorem \ref{theorem_kernel_generic}, Theorem \ref{theorem_image_gen}] \label{theorem_reduct_intro_1}
    Suppose that $(\mm, \theta) \models T^C_\theta$ is existentially closed, and let $\rho \in K[X]$ be given.
    Then the structure $(\mm, \Ker(\rho[\theta]))$ is either interdefinable with $\mm$ or an existentially closed model of the $L \cup \set{V}$-theory
    $$
    T_V := T \cup \set{\text{``$V$ is a vector subspace of $\VV$"}},
    $$
    where $V$ is a predicate of the same arity as $\VV$.
    The same is true for the structure $(\mm, \Image(\rho[\theta]))$.
\end{theoremi}
\noindent If $T \supseteq \TKvs$ (and the vector space $\VV$ is just the reduct in the language of $K$-vector spaces), then the model companion of $T_V$, which exists if $T$ satisfies \Hfour{}, is a special case of d'Elbée's ``generic expansion by a reduct'' \cite{dEl21b}.
In a future paper, we will show that our construction preserves the neostability property \NATP{} (not the antichain tree property), which was quite recently introduced by Ahn and Kim in \cite{AK24}.
Theorem \ref{theorem_reduct_intro_1} will allow us in another future paper to transfer this preservation result to the construction presented in \cite{Blo23}. In \cite{Blo23}, Block Gorman takes o-minimal theories $T$ in which a torsion-free divisible abelian group $D$ is definable and studies the companionability of $T_\Gg$, the theory $T$ expanded by a predicate for a dense and codense divisible subgroup of $D$. For the precise assumptions on $T$ and $D$, see (ii) of Example \ref{examples_hfour}. It is not difficult to see that the model companion of $T_\Gg$ is, if it exists, precisely the model companion of $T_V$ if $\VV$ is chosen as $D$ viewed as a $\QQ$-vector space.

Another natural question is what happens if one replaces $\theta$ with $\zeta[\theta]$ for some $\zeta \in K[X]$.
More concretely, if $(\mm, \theta) \models T\theta^C$, one can ask whether $(\mm, \zeta[\theta]) \models T\theta^{C'}$ for some potentially different set of constraints $C'$.
In Theorem \ref{theorem_iterations}, we show that there can be at most one such $C'$ and give a precise condition, in terms of $C$ and $\zeta$, under which $(\mm, \zeta[\theta])$ is an existentially closed model of $T_\theta^{C'}$.
This question was asked by d'Elbée in the setting of $\operatorname{ACFH}$, the model companion of the theory of algebraically closed fields expanded by a multiplicative endomorphism; see Question 5.16 in \cite{dEl25}.
Since $\operatorname{ACFH}$ corresponds to the empty set of constraints $C = \varnothing$, we can answer this question positively in our setting:
\begin{theoremi}[Corollary \ref{corollary_iteration_crit}] \label{theoremi_iter}
    Suppose $(\mm, \theta) \models T^C_\theta$ is existentially closed, $C = \varnothing$, and let $\zeta \in K[X]$ be given with $\deg(\zeta) \geq 1$.
    Then $(\mm, \zeta[\theta])$ is an existentially closed model of $T^C_\theta$.
\end{theoremi}
\noindent Theorem \ref{theoremi_iter} above also holds for the set of constraints stating that $\rho[\theta]$ is injective for every $\rho \in K[X] \setminus \set{0}$.\\

\noindent \textbf{Acknowledgment.} The author would like to thank Christian d'Elbée for reading parts of an earlier draft of this paper and for providing valuable feedback and suggestions.
The author would also like to thank Philipp Hieronymi for general support.

\section{Preliminary Results} \label{sec_prelim_res}

In this chapter, we provide a recap of all relevant results from \cite{Chi25} and \cite{Chi25b}.

\subsection{The setting}  
Let $L$ be a first-order language, let $T$ be a model-complete $L$-theory, and let $K$ be a field.
Furthermore, assume that the theory of $K$-vector spaces is definable in $T$.
By this we mean that there are $L$-formulas $\Omega_{\VV}(\ux)$, $\Omega_{0}(\ux)$, $\Omega_{+}(\ux_1, \ux_2, \uy)$, and $\Omega_{\lambda\cdot}(\ux, \uy)$ for each $\lambda \in K$ such that, in every model $\mm \models T$, they define an infinite $K$-vector space $(\VV, 0, +, (\lambda \cdot)_{\lambda \in K})^{\mm}$.
In the case $K = \QQ$, one may also view $(\VV, 0, +)^{\mm}$ as a torsion-free divisible abelian group, since these are precisely the $\QQ$-vector spaces.
If a model $\mm \models T$ is given, then $\VV$ denotes $\VV^\mm$.
If the model is denoted with $\mm'$ instead of $\mm$, then we will write $\VV'$ instead of $\VV$.
If no model of $T$ is clear from the context, then we will still use the letter $\VV$ to denote a $K$-vector space.
We may say ``vector space'' instead of ``$K$-vector space'', ``polynomial'' instead of ``$K$-polynomial'', ``linearly independent'' instead of ``$K$-linearly independent'', and so on. Definable will always mean $\varnothing$-definable.

\begin{example}
    \label{example_main}
    Two of our main examples are as follows:
    \begin{enumerate}[(i)]
        \item Let $\LK = \set{0, +, (\lambda \cdot)_{\lambda\in K}}$ and let $\TKvs$ be the theory of $K$-vector spaces, with the obvious formulas.
        Then, for any $\mm \models \TKvs$, we choose $(\VV, 0, +, (\lambda \cdot)_{\lambda\in K}) = \mm$.
        Similarly, we can also work with ordered divisible abelian groups, since these are precisely ordered $\QQ$-vector spaces.
        \item Let $K = \QQ$, $\Lr = \set{0, 1, +, \cdot, <}$, and $T = \RCF$, the theory of real closed fields.
        Then, for any $\rr \models \RCF{}$, we choose $(\VV, 0, +, (q \cdot)_{q\in \QQ})^{\rr}$ to be $(\rr_{>0}, 1, \cdot, (x \mapsto x^q)_{q\in \QQ})$.
    \end{enumerate}
\end{example}

\noindent Notationally, we will treat $\VV$ as a unary set, or even as a separate sort.
Any $\LK$-term or formula can then be viewed as an $L$-definable function or an $L$-formula, respectively.
Note that $0$, $+$, and $(\lambda\cdot)_{\lambda\in K}$ need not belong to $L$, so they are not necessarily $L$-terms.
For a given theory, there can be multiple definable vector spaces, as in the case of $\RCF{}$.
Thus, whenever a theory $T$ is given, we actually mean the tuple $(T, \Omega_{\VV}, \Omega_{0}, \Omega_{+}, (\Omega_{\lambda\cdot})_{\lambda \in K})$.

We now define $L_\theta := L(\theta) := L \cup \set{\theta}$, where $\theta$ is a function symbol not contained in $L$.
This has the drawback that, for example, if $\rr \models \RCF$ and the positive elements are viewed as a $\QQ$-vector space, then $\theta$ must also be defined outside $\VV = \rr_{>0}$.
In this case, one might try to set $\theta(0) = 0$ and $\theta(-1) \in \set{1, -1}$ so as to extend $\theta$ to an endomorphism of $(\rr, 1, \cdot)$, but then the ambient structure is no longer a vector space.
For the sake of uniformity, we instead define $\theta(x)$ to be the neutral element of $\VV$ for all $x \not\in \VV$ and set
$$
T_\theta := T \cup \set{\text{``$\theta_{\restriction \VV}$ is a $(\VV, 0, +, (\lambda \cdot)_{\lambda\in K})$-endomorphism''}} \cup \set{\forall x \not\in \VV: \theta(x) = 0}.
$$
In particular, $\theta^n(x) = 0$ for all $x \not\in \VV$ and all $n > 0$.
For consistency, we also define $\theta^0(x) = 0$ whenever $x \not\in \VV$, and $x + y = 0$ whenever either $x \not\in \VV$ or $y \not\in \VV$.
Practically, we will ignore the behavior of $\theta$ outside of $\VV$ and simply treat $\theta$ as a function defined only on $\VV$.
For example, we set $\Ker(\theta) := \set{v \in \VV : \theta(v) = 0}$, and similarly define the kernel for any function that is an endomorphism of $\VV$.

Note that if $\VV$ is an $n$-ary set, then we actually need $n$ different $n$-ary function symbols $\theta_1, \dots, \theta_n$ rather than a single unary function symbol $\theta$.
However, as mentioned above, we will treat elements of $\VV$ as singletons in order to simplify notation and will therefore pretend that $\theta$ is unary.

\begin{definition}
    Given a polynomial $\rho \in K[X]$ and any $d \geq \deg(\rho)$, we let $\rho[\theta]$ denote the endomorphism of $\VV$ defined by
    $
    \rho[\theta](v) := \sum\nolimits_{i=0}^{d} (\rho)_i \cdot \theta^i(v)
    $
    for all $v \in \VV$. Here each $(\rho)_i$ is the respective coefficient of $X^i$ in $\rho$.
\end{definition}

\begin{notation}
    If $\theta$ is clear from the context we may write $\Ker(\rho)$ and $\Image(\rho)$ instead of $\Ker(\rho[\theta])$ and $\Image(\rho[\theta])$.
\end{notation}

\noindent It is easy to check that $\rho[\theta] + \eta[\theta] = (\rho + \eta)[\theta]$ and $\rho[\theta] \circ \eta[\theta] = (\rho \cdot \eta)[\theta]$. Many classical facts for polynomials, such as Bézout's identity or Euclidean division, can be translated to this setting. The following results follow directly from Bézout's identity:
\begin{fact}
    \label{corollary_polynomials_kernel_facts}
    Given polynomials $\rho_1, \rho_2 \in K[X]$, we obtain
    $\Ker(\rho_1) \cap \Ker(\rho_2) = \Ker(\gcd(\rho_1, \rho_2))$. If $\rho_1, \rho_2$ are relatively prime, then we obtain $\Ker(\rho_1 \cdot \rho_2) = \Ker(\rho_1) \oplus \Ker(\rho_2)$.
\end{fact}

\begin{fact}
    \label{corollary_polynomials_kernel_facts2}
    Given polynomials $\rho_1, \rho_2 \in K[X]$, we obtain
    $\Image(\rho_1) \cap \Image(\rho_2) = \Image(\lcm(\rho_1, \rho_2))$. If $\rho_1, \rho_2$ are relatively prime, then we obtain $\Image(\rho_1 \cdot \rho_2) = \Image(\rho_1) \cap \Image(\rho_2)$.
\end{fact}

\noindent All conventions for $\gcd$ and $\lcm$ are chosen in this paper such that the two facts above generalize to any number of polynomials, even if some of them are $0$; for details, see the paragraph above Corollary 2.5 in \cite{Chi25}.

\begin{fact}
    \label{fact_inv_poly_on_ker}
    Given polynomials $\rho_1, \rho_2 \in K[X]$ relatively prime, the map $\rho_1[\theta]_{\restriction \Ker(\rho_2)} \colon \Ker(\rho_2) \to \Ker(\rho_2)$ is an isomorphism, and there is some $\chi \in K[X]$ such that $\rho_1[\theta]_{\restriction \Ker(\rho_2)}^{-1} = \chi[\theta]_{\restriction \Ker(\rho_2)}$. In particular, $\Ker(\rho_2) \subseteq \Image(\rho_1)$ holds.
\end{fact}

\subsection{Kernel configurations and extensions of $T_\theta$}

\noindent As stated in the introduction, we consider a family $\set{T^C_\theta : C \in \Cc}$ of extensions of $T_\theta$.

\begin{notation}
    We let $\Kp{}$ denote the set of all monic irreducible polynomials over $K$.
\end{notation}

\noindent We start by defining our index set $\Cc$:

\begin{definition} \label{def_kernel_conf}
    We call a pair $(c, d)$ a \textbf{kernel configuration} if
    \begin{enumerate}[(i)]
        \item $c \colon \Kp{} \to \NN \cup \set{\infty}$ is a function; and
        \item $d \in \NN_{> 0} \cup \set{\infty}$ is either $\infty$ or satisfies $d = \sum_{f \in \Kp{}} \deg(f) \cdot c(f)$.
    \end{enumerate}
    We let $\Cc$ denote the set of all kernel configurations.
    Given such a kernel configuration $C = (c, d) \in \Cc$, we set $C(f) := c(f)$ for all $f \in \Kp{}$.
    We define the \textbf{degree} of $C$ by $\deg(C) := d$.
    We say that $C$ is \textbf{algebraic} if $\deg(C) < \infty$.
    In this case, we define the \textbf{minimal polynomial} of $C$ by $\mipo(C) := \prod_{f \in \Kp{}} f^{C(f)}$
    (since $\deg(C) < \infty$, only finitely many factors are different from $1$, so this product is well defined).
    We say that $C$ is \textbf{transcendental} if $\deg(C) = \infty$.
    We let $\Ccalg$ and $\Cctrans$ denote the sets of all algebraic and all transcendental kernel configurations, respectively.
\end{definition}

\noindent Note that the set of kernel configurations depends on the field $K$.
We are now ready to define the family $\set{T^C_\theta : C \in \Cc}$:

\begin{definition} \label{def_T_C_theta}
    Given $C \in \Cc$ and an endomorphism $\theta \colon \VV \to \VV$, we say that $\theta$ is a \textbf{$\mathbf{C}$-endomorphism} if one of the following holds:
    \begin{enumerate}[(i)]
        \item $C$ is algebraic and $\Ker(\mipo(C)) = \VV$, that is, $\mipo(C)[\theta] = 0$.
        \item $C$ is transcendental and $\Ker(f^{C(f)}) = \Ker(f^{C(f)+1})$ for all $f \in \Kp{}$ with $C(f) < \infty$.
    \end{enumerate}
    We define $T^C_\theta := T_\theta \cup \set{\text{``$\theta$ is a $C$-endomorphism''}}$.
\end{definition}

\noindent The family $\set{T^C_\theta : C \in \Cc}$ might seem a bit arbitrary at first; however, notice that any consistent extension of the form
$$
    T_\theta \cup \Set{\sum\nolimits_{k}\bigcap\nolimits_{l}\Ker(\rho_{j, k, l}[\theta]) = \sum\nolimits_{k}\bigcap\nolimits_{l} \Ker(\eta_{j, k, l}[\theta]) : j \in \jj},
$$
where all sums and intersections are finite,  all the $\rho_{j, k, l}$'s and $\eta_{j, k, l}$'s are polynomials over $K$, and $\jj$ is a potentially infinite index set, is equivalent to some $T^C_\theta$ (see Corollary 2.13 in \cite{Chi25} - the proof heavily uses consequences of Bézout's identity). Also note that every $T^C_\theta$ is consistent, and that $T^C_\theta \not\equiv T^{C'}_\theta$ whenever $C \neq C'$ (see Lemma 2.19 in \cite{Chi25}). So the set $\Cc$ parametrizes all consistent extensions as described above, in some sense. Some concrete examples:
\begin{enumerate}[(i)]
    \item Let $C_\infty$ be transcendental with $C_\infty(f) = \infty$ for all $f \in \Kp{}$. One can check that $T_\theta^{C_\infty}$ is $T_\theta$. 
    \item Let $C_0$ be transcendental with $C_0(f) = 0$ for all $f \in \Kp{}$. One can check that $T_\theta^{C_0}$ is $T_\theta \cup \set{\text{``$\rho[\theta]$ is injective''} : \rho \in K[X] \setminus \set{0}}$. In an existentially closed model of $T_\theta^{C_0}$, the maps $\rho[\theta]$ are isomorphisms, so one can solve systems of equations of the form $\bigwedge_{k=1}^m \sum_{l=1}^n \rho_{k, l}[\theta](x_l) = y_k$ just like in a $K(X)$-vector space.
    \item Let $C_f$ be algebraic with $\mipo(C_f) = f \in \Kp{}$. One can, similarly to (ii), solve such systems of equations as in a $K[X]/(f)$-vector space. Here, however, one has the potential advantage that the sequence $\set{\theta^i(v) : i \in \omega}$ is already determined by $\set{\theta^i(v) : 0 \leq i < \deg(f)}$.
\end{enumerate}
These are, in some cases, the ``easiest'' examples to work with, and it can often be helpful to first work with one of these kernel configurations and then turn to the general case. By contrast, the ``hardest'' kernel configurations to work with are those for which $\set{f \in \Kp{} : 0 < C(f) < \infty}$ is infinite. These notions of ``easy'' and ``hard'' mostly describe the notational complexity introduced when working with these kernel configurations, but one can also show that they correspond to neostability-theoretic complexity of the model companion of $\TKvsTheC$ (see Remark \ref{rem_neo_stab_properties}).

In \cite{Chi25}, we also showed that every $T^C_\theta$ is inductive (see Lemma 3.2 there). Hence, the model companion of each $T^C_\theta$ exists if and only if the class of existentially closed models of $T^C_\theta$ is elementary. In this case, the model companion is exactly the axiomatization.

One can easily check that if $C$ and $C'$ are algebraic kernel configurations, then $C = C'$ if and only if $\mipo(C) = \mipo(C')$. Furthermore, we have $\deg(C) = \deg(\mipo(C))$. Finally, we would like to note that the equation in (ii) of Definition \ref{def_kernel_conf} also holds for $C$-endomorphisms if $C$ is algebraic:

\begin{fact} \label{remark_alg_kc}
    Let $C$ be algebraic and let $\theta$ be a $C$-endomorphism. Then
    $\Ker(f^{C(f)}) = \Ker(f^{C(f)+1})$ holds for all $f \in \Kp{}$.
\end{fact}

\begin{fact}[Standard Construction, Lemma 2.18 in \cite{Chi25}]
    \label{lemma_standart_construction}
    Given a $C$-endomorphism $\theta \colon \VV \to \VV$ and a vector space $\VV' \supset \VV$ with $\dim(\VV'/\VV) \geq \aleph_0$, there exists a $C$-endomorphism $\theta' \colon \VV' \to \VV'$ extending $\theta$.
\end{fact}

\begin{fact} \label{fact_trivvivi}
    If $C$ is \textbf{trivial}, that is, if $\deg(C) = 1$, then the models of $T_\theta^C$ and $T$ are interdefinable.
    Hence, the model companion of $T_\theta^C$ is $T_\theta^C$ itself.
\end{fact}

\noindent We will often implicitly assume that $C$ is non-trivial.

\begin{notation}
    We introduce a few more notations for working with a kernel configuration $C \in \Cc$:
    \begin{enumerate}[(i)]
        \item Given $f \in \Kp{}$ with $C(f) < \infty$, we write $f^C$ instead of $f^{C(f)}$, $f^{C+k}$ instead of $f^{C(f) + k}$, and so on.
        \item Given a finite set $F \subseteq \Kp{}$ with $C(f) < \infty$ for all $f \in F$, we set \hbox{$F^C := \prod_{f \in F} f^C$}.
        \item We define the following subsets of $\Kp{}$:
        $$
        \Kp{C<\infty} := \set{f \in \Kp{} : C(f) < \infty}, \quad \Kp{0<C<\infty} := \set{f \in \Kp{} : 0 < C(f) < \infty},
        $$
        $$
        \Kp{C=0} := \set{f \in \Kp{} : C(f) = 0}, \quad \text{and} \quad \Kp{C=\infty} := \set{f \in \Kp{} : C(f) = \infty}.
        $$
    \end{enumerate}
\end{notation}

\noindent With the above, for any algebraic kernel configuration $C \in \Ccalg$, we have $\deg(C) = \deg(\mipo(C))$ and
    $$
    \mipo(C) = \prod\nolimits_{f \in \Kp{0<C<\infty}} f^C = (\Kp{0<C<\infty})^C.
    $$
\noindent One should note that any $\LKThe$-sentence that holds in $\TKvsTheC$ also holds in $T^C_\theta$ (here $\LKThe$ is the language of $K$-vector spaces with an endomorphism).
To be more precise, one has to modify the sentence accordingly, e.g., quantifiers of the form $\exists x$ must be replaced with $\exists x \in \VV$ and the formulas that define addition/scalar multiplication must be used instead of the function symbols in $\LKThe$.
In general, it also turns out that whenever a model $(\mm, \theta)$ of $T^C_\theta$ is existentially closed, $(\VV, \theta)$ is an existentially closed model of $\TKvsTheC$ (see Remark 3.4 in \cite{Chi25}).

\subsection{$C$-image-completeness}

\noindent Note that the condition of $\theta$ being a $C$-endomorphism does not (at least in the transcendental case) imply any kind of equations that involve the image of $\rho[\theta]$ for some $\rho \in K[X]$.
In Remark 2.21 in \cite{Chi25}, we discussed that it is very unlikely that considering expansions of $T_\theta$ that also impose equations on the images (or mixed equations with sums and intersections of both kernels and images) will lead to new model companions.
However, the following condition holds in any existentially closed model of $T^C_\theta$ and is fundamental to understanding the structure of these models:

\begin{definition} \label{def_C_image_comple}
    We say that an endomorphism $\theta \colon \VV \to \VV$ is \textbf{$\mathbf{C}$-image-complete} if it is a $C$-endomorphism and $\Image(f^{C+1}) = \Image(f^C)$  holds for all $f \in \Kp{C<\infty}$ (recall $\Image(f^C) := \Image(f^{C(f)}[\theta])$).
    We may call a model $(\mm, \theta) \models T_\theta$ \textbf{$\mathbf{C}$-image-complete} if the endomorphism $\theta$ is $C$-image-complete.
\end{definition}

\begin{fact} \label{fact_c_image_comple}
    The following statements hold:
    \begin{enumerate}[(i)]
        \item If $C$ is algebraic, then every $C$-endomorphism is $C$-image-complete (see Lemma 3.11 in \cite{Chi25}).
        \item If $C$ is any kernel configuration and $(\mm, \theta) \models T_\theta^C$ is existentially closed, then $(\mm, \theta)$, or equivalently $\theta$, is also $C$-image-complete (see Corollary 3.13 in \cite{Chi25}).
    \end{enumerate}
\end{fact}

\noindent Note that the converse of (ii) in Fact \ref{fact_c_image_comple} above does not hold.
The most important consequence of $C$-image-completeness is that we can decompose $\VV$ into definable direct summands as follows.

\begin{fact}[Lemma 3.14 in \cite{Chi25}] \label{lemma_decomposition}
    If $\theta \colon \VV \to \VV$ is $C$-image-complete and $F \subseteq \Kp{0<C<\infty}$ is a finite set, then we have
    $$
    \VV = \Image(F^C) \oplus \Ker(F^C) = \Image(F^C) \oplus \bigoplus\nolimits_{f\in F} \Ker(f^C).
    $$
    If $C$ is algebraic and $F = \Kp{0<C<\infty}$, then the summand $\Image(F^C)$ is $\set{0}$ and can therefore be omitted.
\end{fact}

\noindent In \cite{Chi25}, we defined additional endomorphisms of $\VV$ using these decompositions.

\begin{notation}
    For any $\rho \in K[X] \setminus \set{0}$, we let $\Fac(\rho) := \set{f \in \Kp{} : f \mid \rho}$ denote the set of all irreducible factors of $\rho$.
    As a convention, we set $\Fac(0) = \varnothing$.
\end{notation}

\begin{fact}[Lemma 3.15 in \cite{Chi25}] \label{fact_endo_gen}
    In the theory \hbox{$\TKvsThe \cup \set{\text{``$\theta$ is $C$-image-complete''}}$}, the following endomorphisms are definable:
    \begin{enumerate}[(i)]
        \item For $F \subseteq \Kp{0<C<\infty}$ finite, we define the \textbf{projection to the image of $\bm{F^C[\theta]}$} by
        $$
        \pi_{\Image(F^C)}(x) := \text{``the unique $u \!\in\! \Image(F^C)$ for which there is $v \in \Ker(F^C)$ with $x \!=\! u \!+\! v$''.}
        $$
        \item For $F \subseteq \Kp{0<C<\infty}$ finite, we define the \textbf{projection to the kernel of $\bm{F^C[\theta]}$} by
        $$
        \pi_{\Ker(F^C)}(x) := \text{``the unique $v \!\in\! \Ker(F^C)$ for which there is $u \in \Image(F^C)$ with $x \!=\! u \!+\! v$''.}
        $$
        We clearly have $\pi_{\Ker(F^C)} = 1[\theta] - \pi_{\Image(F^C)}$.
        \item For every monic polynomial $\eta \in K[X]$ with $\Fac(\eta) \subseteq \Kp{C<\infty}$, we define the \textbf{pseudo-inverse of $\bm{\eta[\theta]}$} by
        $$
        \eta[\theta]^{-1}(x) := \text{``the unique $u \in \Image(\Fac(\eta)^C)$ with $\eta[\theta](u) = \pi_{\Image(\Fac(\eta)^C)}(x)$''.}
        $$
        Notice that $\Fac(\eta)^C = (\Fac(\eta) \cap \Kp{0<C<\infty})^C$.
        In practice, we will also use $\eta[\theta]^{-1}$ for (non-zero) non-monic polynomials by setting $\eta[\theta]^{-1} := \lambda^{-1} \cdot (\eta/\lambda)[\theta]^{-1}$ for the leading coefficient $\lambda$ of $\eta$.
    \end{enumerate}
\end{fact}

\noindent Notice that whenever $\Fac(\eta) \cap \Kp{0<C<\infty} \neq \varnothing$, we obtain $\eta[\theta] \circ \eta[\theta]^{-1} = \pi_{\Image(\Fac(\eta)^C)} \neq \Id = 1[\theta]$, so the notation might be a bit misleading. Here $\Id$ is the identity on $\VV$.
It is also easy to see that $\Id = 1[\theta] = \pi_{\Image(f^C)} + \pi_{\Ker(f^C)}$. We consider the ring generated by all $\LKThe$-definable endomorphisms we have collected so far:

\begin{fact}[Theorem 3.18 in \cite{Chi25}] \label{theorem_r_c_def}
    We let $R_C$ be the set of all endomorphisms definable in the theory $\TKvsThe \cup \set{\text{``$\theta$ is $C$-image-complete''}}$ that are $\set{+, \circ}$-generated by
    $$
    \set{\rho[\theta] : \rho \in K[X]} \cup \set{\pi_{\Image(F^C)} : F \subseteq \Kp{0<C<\infty} \text{ finite}} \cup \set{\eta[\theta]^{-1} : \eta \text{ monic with }\Fac(\eta) \subseteq \Kp{C<\infty}}.
    $$
    The structure $(R_C, 0[\theta], 1[\theta], +, \circ)$ is a unitary commutative ring with $\Char(R_C) = \Char(K)$. In fact, it can also be seen as a $K$-algebra, as $K \subseteq R_C$ (identifying $\lambda \in K$ with the endomorphism $x \mapsto \lambda \cdot x$ which is $\lambda[\theta]$).
\end{fact}

\noindent We may sometimes write $0$ instead of $0[\theta]$ and $\Id$ or $1$ instead of $1[\theta]$. In particular, if we regard $R_C$ purely as a ring, we may write $(R_C, 0, 1, +, \cdot)$. The ring $R_C$ may again look complicated at first, but for the ``easiest to work with'' kernel configurations, we obtain the following:
\begin{enumerate}[(i)]
    \item $(R_{C_\infty}, 0, 1, +, \cdot) \simeq (K[X], 0, 1, +, \cdot)$ for the unique kernel configuration $C_\infty \in \Cctrans$, which satisfies \hbox{$C_\infty(f) = \infty$} for all $f \in \Kp{}$.
    \item $(R_{C_0}, 0, 1, +, \cdot) \simeq (K(X), 0, 1, + , \cdot)$ for the unique kernel configuration $C_0 \in \Cctrans$ with $C_0(f) = 0$ for all $f \in \Kp{}$. 
    Notice that this is a field.
    \item $(R_{C}, 0, 1, +, \cdot) \simeq (K[X]/(\mipo(C)), 0, 1, +, \cdot)$ for all $C \in \Ccalg$.
    This also implies that our ring $(R_{C}, 0, 1, +, \cdot)$ is a field for all algebraic kernel configurations $C$ with $\mipo(C)$ being irreducible.
\end{enumerate}
Other examples of $R_C$ can be found in Corollary 3.25 in \cite{Chi25}. 
There, the case where $C$ is transcendental and $\Kp{0<C<\infty}$ is infinite again turns out to be the most complicated. Multiplication rules, such as $\rho[\theta] \circ \pi_{\Image(F^C)} = \rho[\theta]$ if $F^C \mid \rho$, can be found in Lemma 3.21 in \cite{Chi25}.  

\begin{fact}[see Remark 3.26 in \cite{Chi25}] \label{fact_when_field}
    $R_C$ is a field if and only if $C = C_0$, as in (ii) above, or if $C$ is algebraic with $\mipo(C)$ irreducible. 
\end{fact}

\noindent Also note that the elements of $R_C$ are, as defined in Fact \ref{theorem_r_c_def}, definable functions in the theory $\TKvsThe \cup \set{\text{``$\theta$ is $C$-image-complete''}}$.
By this, we mean that the elements of $R_C$ are equivalence classes of $\LKThe$-formulas modulo the theory \hbox{$\TKvsThe \cup \set{\text{``$\theta$ is $C$-image-complete''}}$} that define an endomorphism in every model of $\TKvsThe \cup \set{\text{``$\theta$ is $C$-image-complete''}}$.
So, in order to prove $r = r'$, we need to show
$$
    r^{(\VV, \theta)} = r'^{(\VV, \theta)} \quad \text{for all $(\VV, \theta) \models \TKvsThe \cup \set{\text{``$\theta$ is $C$-image-complete''}}$},
$$
and, in order to prove $r \neq r'$, we need to find $(\VV, \theta) \models \TKvsThe \cup \set{\text{``$\theta$ is $C$-image-complete''}}$ with
$
    r^{(\VV, \theta)} \neq r'^{(\VV, \theta)}.
$
In Remark 3.20 in \cite{Chi25}, we showed that $r \neq r'$ implies $r^{(\VV, \theta)} \neq r'^{(\VV, \theta)}$ if $(\VV, \theta)$ is an existentially closed model of $\TKvsTheC$.
In Theorem \ref{theorem_homomorphism_LKthe} (with $G = H = (\VV, 0, +)$), we will see that in an existentially closed model of $T_\theta^C$, any $\LKThe$-definable endomorphism of $\VV$ is, in fact, in $R_C$.

\subsection{Existentially closed models and first-order axiomatization}

We now state the characterization of the existentially closed models of $T^C_\theta$ from \cite{Chi25}. We start with the remaining ingredients:

\begin{definition}
    \label{def_param_c_sequence_system} A \textbf{parametrized $\mathbf{C}$-sequence-system} is an $\LKThe$-formula of the form
    $$
    S(\ux; \uy) = \bigwedge\nolimits_{k=1}^n f_k^{q_k}[\theta](x_{\ldd, k}) = y_k
    $$
    with $\ux := \ux_\li\ux_\ld := (x_{\lii, k} : 1 \leq k \leq m)(x_{\ldd, k} : 1 \leq k \leq n)$ and $\uy = (y_1, \dots, y_n)$ that satisfies the following conditions:
    \begin{enumerate}[(i)]
        \item If $C$ is algebraic, then $m = 0$.
        \item $f_k \in \Kp{0<C}$ and $q_k \in \set{q \in \NN : 0 < q \leq C(f_k)}$ hold for all $k \in \set{1, \dots, n}$.
    \end{enumerate}
    
\end{definition}
\noindent We will always denote parametrized $C$-sequence-systems by the letter $S$.
Given such a parametrized $C$-sequence-system $S(\ux; \uy)$, we assume that everything is as above, that is, $m$, $n$, and the $f_k$'s and $q_k$'s are defined implicitly, and we set $\ux = \ux_\li\ux_\ld$ as above.
When we partition $\ux = \ux_\li\ux_\ld$ as above, we think of:
    \begin{enumerate}[(i)]
        \item $\ux_\li$ as the linearly independent part of $\ux$, since $S(\ux; \uy)$ does not imply any linear dependencies for the sequence $(\theta^i(x_{\lii, k}) : 1 \leq k \leq m, i \in \omega)$;
        \item $\ux_\ld$ as the linearly dependent part of $\ux$, since the formula $S(\ux; \uy)$ implies that the sequence $(\theta^i(x_{\ldd, k}) : i \in \omega)$ is linearly dependent over $\spanA{y_k}{K}$ for each $k \in \set{1, \dots, n}$.
    \end{enumerate}
The names $\ux_\li$ and $\ux_\ld$ for these tuples are abbreviations for linearly independent and linearly dependent. Notice that in the algebraic case, we require $\ux_\li$ to be empty, which makes sense, as we have $\sum_{i=0}^{\deg(\mipo(C))} (\mipo(C))_i \cdot \theta^i(v) = 0$ for any $v \in \VV$ in that case.

\begin{definition}
    \label{def_c_sequence_system} \label{def_compatible} Let $S(\ux; \uy) = \bigwedge_{k=1}^n f_k^{q_k}[\theta](x_{\ldd, k}) = y_k$ be a parametrized $C$-sequence-system as in Definition \ref{def_param_c_sequence_system}, and let $(\VV, \theta) \models \TKvsThe$ be given.
\begin{enumerate}[(i)]
    \item We say that a tuple $\uu = (u_1, \dots, u_n) \in \VV$ is \textbf{compatible} with $S$ if $u_k \in \Ker(f_k^{C-q_k})$ for every $k \in \set{1, \dots, n}$ with $f_k \in \Kp{0<C<\infty}$.
    \item A \textbf{$\mathbf{C}$-sequence-system} over $(\VV, \theta)$ is an $\LKThe(\VV)$-formula of the form
    \hbox{$
    S(\ux) = S'(\ux; \uu)
    $}
    where $S'$ is a parametrized $C$-sequence-system and $\uu \in \VV$ is compatible with $S'$.
\end{enumerate}
\end{definition}

\noindent We will also denote $C$-sequence-systems over some $(\VV, \theta) \models \TKvsTheC$ by the letter $S$.
Notice that a $C$-sequence-system over $(\VV, \theta)$ is also a $C$-sequence-system over any extension $(\VV', \theta')$ that is also a model of $\TKvsThe$.

\begin{definition}[Placeholder notation] \label{def_placeholder_notation}
    Let $\ux = (x_k : k \in \kk)$ be a tuple of variables.
    We define the \textbf{placeholder sequence} \hbox{$\uxvec := (x^i_k : k \in \kk, i \in \omega)$} to be a new tuple of variables.
    We call each $x^i_k$ a \textbf{placeholder variable} or a \textbf{placeholder} for $\theta^i(x_k)$.
    We furthermore define:
    \begin{enumerate}[(i)]
        \item $\ux^i := (x_k^i : k \in \kk)$ for each $i \in \omega$, and
        \item $\xvec_k := (x^i_k : i \in \omega)$ for each $k \in \kk$.
    \end{enumerate}
    We may sometimes write $(\ux^i : i \in \omega)$ or $(\xvec_k : k \in \kk)$ instead of $\uxvec$.
    For a singleton $x$, we similarly define $\xvec := (x^i : i \in \omega)$.
    If a formula $\psi(\uxvec; \uw)$ is given, we define:
    \begin{enumerate}[(i)]
        \setcounter{enumi}{2}
        \item $\psi_\theta(\ux; \uw) := \psi((\theta^i(x_k) : k \in \kk, i \in \omega); \uw)$.
    \end{enumerate}
\end{definition}

\begin{definition} \label{def_formual_bounded}
    Let $S(\ux; \uy)= \bigwedge\nolimits_{k=1}^n f_k^{q_k}[\theta](x_{\ldd, k}) = y_k$ be a parametrized $C$-sequence-system as in Definition \ref{def_param_c_sequence_system}.
    If $\psi(\uxvec; \uw)$ is a formula, then we say $\psi(\uxvec; \uw)$ is \textbf{bounded} by $S$ if one of the following equivalent conditions holds:
        \begin{enumerate}[(i)]
            \item For all $k \in \set{1, \dots, n}$, the variable $x^i_{\ldd, k}$ does not appear in $\psi(\uxvec; \uw)$ for \hbox{$i \geq \deg(f_k^{q_k})$}.
            \item For all $k \in \set{1, \dots, n}$, the term $\theta^i(x_{\ldd, k})$ does not appear in $\psi_\theta(\ux; \uw)$ for \hbox{$i \geq \deg(f_k^{q_k})$}.
        \end{enumerate}
    We also say that $\psi_\theta(\ux; \uw)$ is \textbf{bounded} by $S$, if $\psi(\uxvec; \uw)$ is bounded by $S$. Furthermore, we say that a formula is \textbf{bounded} by a $C$-sequence-system $S$ (i.e., a parametrized $C$-sequence-system with some compatible parameters plugged in) if it is bounded by the underlying parametrized $C$-sequence-system.
\end{definition}

\noindent In practice, for a formula $\psi(\uxvec)$ to be bounded by $S$ means that no subterm of the form $\theta^i(x_{\ldd, k})$ appearing in $\psi_\theta(\ux)$ can be replaced by applying a Euclidean division with the equation $f_k^{q_k}[\theta](x_{\ldd, k}) = y_k$ from $S(\ux; \uy)$. Indeed, if $i \geq \deg(f^{q_k}_k)$, then we could replace $\theta^i(x_{\ldd, k})$ with $\chi[\theta](y_k) + r[\theta](x_{\ldd, k})$, where $\chi, r \in K[X]$ are the unique polynomials  satisfying $\chi \cdot f^{q_k}_k + r = X^i$ and $\deg(r) < \deg(f^{q_k}_k) \leq i$.

Now that we have all ingredients, we can state the characterization of existentially closed models of $T^C_\theta$:

\begin{theorem}[Theorem 3.32 in \cite{Chi25}] \label{theorem_big_characterization}
    $(\mm, \theta) \models T^C_\theta$ is existentially closed if and only if it is $C$-image-complete and 
    $$
        (\mm, \theta) \models \exists \ux \in \VV : \psi_\theta(\ux) \wedge S(\ux)
    $$
    holds for any $C$-sequence-system $S(\ux)$ over $(\VV, \theta)$ and $L(M)$-formula $\psi(\uxvec)$ that is bounded by $S$ and does not imply any finite disjunction of non-trivial linear dependencies in $\uxvec$ over $\VV$.
\end{theorem}

\noindent We give two examples where the characterization simplifies quite a lot:
\begin{enumerate}[(i)]
    \item Fix $f \in \Kp{}$ and let $C_f$ be the unique algebraic kernel configuration with $\mipo(C_f) = f$.
    A model $(\mm, \theta) \models T^{C_f}_\theta$ is existentially closed if and only if
    $$
        (\mm, \theta) \models \exists \ux \in \VV : \psi(\theta^0(\ux), \dots, \theta^{\deg(f)-1}(\ux))
    $$
    holds for every $L(M)$-formula $\psi(\ux^0, \dots, \ux^{\deg(f)-1})$ that does not imply any finite disjunction of non-trivial linear dependencies in $\ux^0, \dots, \ux^{\deg(f)-1}$ over $\VV$.
    This is Theorem 3.33 in \cite{Chi25}.
    
    \item Let $C_0$ be the unique transcendental kernel configuration with $C_0(f) = 0$ for all $f \in \Kp{}$.
    A model $(\mm, \theta) \models T^{C_0}_\theta$ is existentially closed if and only if $\rho[\theta]$ is invertible for every $\rho \in K[X] \setminus \set{0}$ and
    $$
        (\mm, \theta) \models \exists \ux \in \VV : \psi_\theta(\ux)
    $$
    holds for every $L(M)$-formula $\psi(\uxvec)$ that does not imply any finite disjunction of non-trivial linear dependencies in $\uxvec$ over $\VV$.
    This is Theorem 3.34 in \cite{Chi25}.
\end{enumerate}

\noindent The next step is to first-order axiomatize this characterization when possible.
For this, we need the following two families of formulas:

\begin{fact}[Lemma 3.36 in \cite{Chi25}]
    Given a parametrized $C$-sequence-system $S(\ux; \uy)$, there is a $\LKThe$-formula $\delta_S(\uy)$ such that $(\mm, \theta) \models \delta_S(\uu)$ holds if and only if $\uu$ is compatible with $S$.
\end{fact}

\begin{definition}[see Definition 1.11 in \cite{dEl21b}] \label{def_hfour}
    We say that $T$ (with the specific choice of $\VV$) satisfies $(\operatorname{H4})$ if, for every $L$-formula $\psi(\ux; \uw)$, there is an $L$-formula $\sigma_\psi(\uw)$ such that, for all $\mm \models T$ and $\ud \in M$, we have $\mm \models \sigma_\psi(\ud)$ if and only if one of the following two equivalent conditions holds:
    \begin{enumerate}[(i)]
        \item The formula $\psi(\ux; \ud)$ implies no finite disjunction of non-trivial linear dependencies in $\ux$ over $\VV$.
        \item There are an elementary extension $\mm' \succ \mm$ and a tuple $\uv' \in \VV'$ linearly independent over $\VV$ such that $\mm' \models \psi(\uv'; \ud)$.
    \end{enumerate}
\end{definition}

\begin{theorem}[Theorem 3.39 in \cite{Chi25}] \label{theorem_first_oder}
If $T$ satisfies \Hfour{}, then $T_\theta^C$ has a model companion $T\theta^C$, i.e., a first-order axiomatization of the class of existentially closed models.
It is axiomatized by the theory \hbox{$T_\theta \cup \set{\text{``$\theta$ is $C$-image-complete''}}$} together with the sentence
$$
\forall \uw:\forall \uy \in \VV : (\sigma_\psi(\uw) \wedge \delta_S(\uy)) \rightarrow \exists \ux \in \VV : \psi_\theta(\ux; \uw) \wedge S(\ux; \uy)
$$
for every parametrized $C$-sequence-system $S(\ux; \uy)$ and every $L$-formula $\psi(\uxvec; \uw)$ that is boun\-ded by $S$.
\end{theorem}

\begin{example} \label{examples_hfour}
    \Hfour{} holds in the following settings:
    \begin{enumerate}[(i)]
        \item The theory $\TKvs$ with the vector space $(\VV, +, 0, (\lambda \cdot)_{\lambda \in K})$ being the entire structure satisfies \Hfour{}.
        This follows easily from quantifier elimination.
        \item Any complete and model-complete o-minimal theory $T$ extending the theory of divisible ordered abelian groups, with $(\VV, +, 0, (q \cdot)_{q \in \QQ})$ being a subinterval of the line (but not necessarily a subgroup) and continuous operations, satisfies \Hfour{} if and only if there is no infinite definable family of germs of $(\VV, +, 0, (q \cdot)_{q \in \QQ})$-endomorphisms at $0_\VV$; combine Theorem 2.4 and Lemma 2.9 in \cite{Blo23}.
        \item Any complete and model-complete o-minimal expansion of $\RCF{}$ with $(\VV, +, 0, (q \cdot)_{q \in \QQ})$ given by $(R_{>0}, \cdot, 1, (x \mapsto x^q)_{q\in \QQ})$ satisfies \Hfour{} if and only if no partial exponential function is definable.
        This is a special case of (ii); see the proof of Theorem A in \cite{Blo23}.
        \item Let $\FF_q$ be a finite field with $q = p^r$.
        If $T$ expands the theory of an $\FF_q$-vector space and $\VV$ is that vector space, then $T$ satisfies \Hfour{} if and only if it eliminates $\exists^\infty$; see the proof of Theorem 5.2 in \cite{dEl21b}.

        For expansions of the theories $\ACF{}_p$, $\operatorname{SCF}_{p, e}$ ($e$ either finite or infinite), $\operatorname{Psf}_p$, $\operatorname{ACFA}_p$, and $\operatorname{DCF}_p$, this implies that whenever $\FF_q$ is contained in every model as constants, \Hfour{} holds with the $\FF_q$-vector space given by addition.
        For more details, see Example 5.10 in \cite{dEl21b}.
    \end{enumerate}
\end{example}

\noindent In order to apply our characterization of existentially closed models, we often need to check that a formula $\psi(\uxvec{})$ implies no finite disjunction of non-trivial linear dependencies in $\uxvec{}$ over $\VV$. Lemma \ref{lemma_iter_lin_indep} and Corollary \ref{corollary_iter_lin_indp} below help us to do so.

\begin{lemma} \label{lemma_iter_lin_indep}
    Let $t_1(\ux), \dots, t_m(\ux)$ be $\LK(\VV)$-terms in $\ux = (x_1, \dots, x_n)$ that are linearly independent over $\VV$, and let $\uu = (u_1, \dots, u_m) \in \VV'$ be linearly independent over $\VV$, where $\mm' \succ \mm$.
    Then there exists $\uv \in \VV''$, for some elementary extension $\mm'' \succ \mm'$, such that $\uv$ is linearly independent over $\VV$ and $\mm'' \models t_k(\uv) = u_k$ for $k = 1, \dots, m$.
\begin{proof}
    Write $t_k(\ux) = \sum\nolimits_{l=1}^n \lambda_{k, l} \cdot x_l + u_{0,k}$.
    By replacing each $u_k$ with $u_k - u_{0,k}$, we can assume that $u_{0,k} = 0$.
    Now we can write $\bigwedge_{k=1}^m t_k(\ux) = u_k$ as $A \cdot \ux = \uu$, where we treat the tuples $\ux$ and $\uu$ as column matrices of sizes $n \times 1$ and $m \times 1$, respectively, and define
    $$
    A := \begin{tikzpicture}[baseline=(Frame.base)]
    \drawText{0}{0}{\lambda_{1,1}}
    \drawText{0}{2}{\lambda_{m,1}}
    \drawText{2}{0}{\lambda_{1,n}}
    \drawText{2}{2}{\lambda_{m,n}}
    \drawHDots{1}{0}{1}
    \drawHDots{1}{2}{1}
    \drawVDots{0}{1}{1}
    \drawVDots{2}{1}{1}
    \drawBorder{0}{0}{3}{3}
    \end{tikzpicture}\in K^{m \times n}.
    $$
    By the linear independence of $t_1(\ux), \dots, t_m(\ux)$, the matrix $A$ has full row rank, so $m \leq n$.
    Hence, there are invertible matrices $P \in K^{m \times m}$ and $Q \in K^{n \times n}$ such that
    $
    P \cdot A \cdot Q = \Diag_{m \times n}(1, \dots, 1).
    $
    Write $\uv' := P \cdot \uu$, and note that this tuple is still linearly independent over $\VV$, since $P$ is invertible.
    Now choose $\mm'' \succ \mm'$ and $\uv'' := (v''_{m+1}, \dots, v''_n)$ such that $\uv'\uv''$ is linearly independent over $\VV$.
    The tuple $\uv := Q \cdot \uv'\uv''$ is again linearly independent.
    Let $D := \Diag_{m \times n}(1, \dots, 1)$, and note that this is the matrix which, when multiplying from the left, projects an $n$-vector down to its first $m$ entries.
    By construction, we obtain
    $$
    P \cdot A \cdot \uv = P \cdot A \cdot Q \cdot Q^{-1} \cdot \uv = D \cdot Q^{-1} \cdot \uv = D \cdot \uv'\uv'' = \uv' = P \cdot \uu.
    $$
    Since $P$ is invertible, we also have $A \cdot \uv = \uu$, i.e., $\mm'' \models \bigwedge_{k=1}^m t_k(\uv) = u_k$.
    Since $\uv$ is linearly independent over $\VV$, we can conclude.
\end{proof}
\end{lemma}

\begin{corollary}
    \label{corollary_iter_lin_indp}
    Let $\ux = (x_i : i \in \ii)$.
    If $\varphi(\ux)$ implies no finite disjunction of non-trivial linear dependencies in $\ux$ over $\VV$, and $(t_{i}(\ux') : i \in \ii)$ is a sequence of $\LK(\VV)$-terms linearly independent over $\VV$, then $\varphi((t_{i}(\ux') : i \in \ii))$ implies no finite disjunction of non-trivial linear dependencies in $\ux'$ over $\VV$.
\begin{proof}
    Let $\ii_0$ be the finite subset of all $i \in \ii$ for which $x_i$ actually appears in $\varphi(\ux)$, and let $\ux'_0$ be the finite subtuple of $\ux'$ consisting of all variables that actually appear in $(t_{i}(\ux') : i \in \ii_0)$.
    For convenience, we write $\ux' = \ux'_0\ux'_1$, and, for $i \in \ii_0$, we may also treat $t_i(\ux')$ as a term in $\ux'_0$.
    Now let $\uu = (u_i : i \in \ii) \in \VV'$ in some $\mm' \succ \mm$ be a realization of $\varphi(\ux)$ linearly independent over $\VV$.
    By Lemma \ref{lemma_iter_lin_indep}, we can find $\uv_0 \in \VV''$ for some $\mm'' \succ \mm'$ such that $\uv_0$ is linearly independent over $\VV$ and $(t_{i}(\uv_0) : i \in \ii_0) = (u_i : i \in \ii_0)$.
    Now, for any choice of $\uv_1$ of the right length in the $\VV$ of some elementary extension of $\mm''$, the tuple $\uv := \uv_0\uv_1$ is a realization of $\varphi((t_{i}(\ux') : i \in \ii))$.
    Since $\uv_0$ is linearly independent over $\VV$, we can choose $\uv_1$ so that $\uv$ is linearly independent over $\VV$.
    Hence $\varphi((t_{i}(\ux') : i \in \ii))$ implies no finite disjunction of non-trivial linear dependencies in $\ux'$ over $\VV$.
\end{proof}
\end{corollary}

\noindent We give one final definition for (parametrized) $C$-sequence-system:

\begin{definition} \label{def_rk_deg_def}
    Let $S(\ux; \uy) = \bigwedge_{k=1}^n f_k^{q_k}[\theta](x_{\ldd, k}) = y_k$ be a parametrized $C$-sequence-system, as in Definition \ref{def_param_c_sequence_system} (i.e., $\ux = \ux_\li\ux_\ld$ and so on).
    We define:
    \begin{enumerate}[(i)]
        \item The \textbf{rank} of $S(\ux; \uy)$ as $\rk(S) := m = |\ux_\li|$.
        \item The \textbf{degree} of $S(\ux; \uy)$ as $\deg(S) := \sum\nolimits_{k=1}^n \deg(f_k^{q_k})$.
    \end{enumerate}
    We define the rank and degree of a $C$-sequence-system (i.e., a parametrized $C$-sequence-system with compatible parameters plugged in) to be the rank and degree of the underlying parametrized $C$-sequence-system.
\end{definition}

\noindent So, for every (parametrized) $C$-sequence-system, $(\rk(S), \deg(S))$ lies in $\NN^2$.  
Note that $(\NN^2, <_\Lex)$ is a well-ordered set, where $<_\Lex$ is the lexicographical order, which defines the first entry to be more significant than the second entry.
This allows us to do inductions on $C$-sequence-systems.

Assuming \Hfour{}, we obtained some results in \cite{Chi25b} for $T\theta^C$ regarding completions, the algebraic closure, and quantifier elimination. We will state their consequences for $\TKvs\theta^C$, the model companion of $\TKvsTheC$.

\begin{definition}(Definition 4.2 in \cite{Chi25b})
    Given an arbitrary set $A \subseteq M$, where $(\mm, \theta) \models T_\theta \cup \set{\text{``$\theta$ is $C$-image-complete''}}$, we define $\cl_\theta(A)$ as the smallest set that contains $A$ and is closed under both $\acl_L$ and multiplication with any element in $R_C$.
\end{definition}

\begin{example}[Example 4.11, Observation 4.12, and Example 4.17 in \cite{Chi25b}]
    \label{example_kvs} 
    The theory $\TKvs\theta^C$ is complete and has quantifier elimination in the language of $R_C$-modules, i.e., $\LRC = (0, +, (r \cdot)_{r\in R_C})$. Furthermore, if $R_C$ is a field, then $\TKvs\theta^C$ is the theory of $R_C$-vector spaces and hence strongly minimal. Also $\acl_{\LKThe} = \dcl_{\LKThe} = \cl_\theta$ holds in $\TKvs\theta^C$. 
\end{example}

\section{A better characterization for theories close to $\TKvs$}
The characterization of existentially closed models of $T^C_\theta$ given in Theorem \ref{theorem_big_characterization} imposes few assumptions on $T$.
Consequently, there can be many definable interactions between $\VV$ and the structure outside of $\VV$.
For some theories, this leads to an unnecessarily complicated characterization.
In this section, we simplify our characterization for theories $T$ that do not have much definable interaction between $\VV$ and the structure outside of $\VV$, in the following sense:

\begin{definition} \label{def_acl_cond}
    We say $T$ satisfies $(V)$ if, for any $\mm' \succ \mm \models T$ and any tuple $\uv'$ in $\VV'$, $\uv'$ is $\acl_L$-independent over $M$ if and only if it is linearly independent over $\VV$.
    Here $\uv' = (v'_1, \dots, v'_m)$ being $\acl_L$-independent over $M$ means that $v'_k \not\in \acl_L(M \uv' \setminus \set{v'_k})$ for $k =1, \dots, m$.
\end{definition}

\begin{example} \label{example_acl_cond}
    The following holds:
    \begin{enumerate}[(i)]
        \item The theory of $K$-vector spaces $\TKvs$ satisfies \aclCond{}.
        \item If $K$ is an ordered field, then the theory of ordered $K$-vector spaces $T_{K\!\operatorname{-ovs}}$ satisfies \aclCond{}.
        In particular, the theory of ordered $\QQ$-vector spaces $T_{\QQ\operatorname{-ovs}} = \Th(\RR, 0, +, (q\cdot)_{q\in \QQ}, <)$ satisfies \aclCond{}.
        \item Set $K = \QQ$.
        The theory $T := \Th(\RR, 0, +, \pi \cdot,  (q\cdot)_{q\in \QQ })$, with $\VV$ being the $\QQ$-vector space given by addition, does not satisfy \aclCond{}.
        To see this, choose $\mm' \succ \mm \models T$ and some $v \in M' \setminus M$.
        Then the tuple $(v, \pi \cdot v)$ is $\QQ$-linearly independent over $M$, but clearly not $\acl_L$-independent over $M$.
    \end{enumerate}
\end{example}

\subsection{Interactions between \aclCond{} and \Hfour{}}

Since $\TKvs$ satisfies \Hfour{} and condition \aclCond{} states that, for tuples in $\VV$, algebraic independence is essentially the same as algebraic independence in $\TKvs$ (i.e., linear independence), one might be tempted to think that \aclCond{} implies \Hfour{}.
While this is not true, it is almost true; we just need to additionally assume that $T$ eliminates $\exists^\infty x \!\in\!\VV$ as follows:
\begin{definition} \label{def_ex_inf_v}
    We say that $T$ \textbf{\boldmath eliminates ${\exists^\infty x \!\in\!{\VV}}$} if, for any $L$-formula $\psi(\ux; \uw)$, there is an $L$-formula $\phi(\uw)$ such that, for any $\mm \models T$ and $\ud \in M$, the following holds:
    $$
    \mm \models \phi(\ud) \quad \Leftrightarrow \quad |\psi(\mm; \ud) \cap \VV| \geq \aleph_0.
    $$
\end{definition}
\noindent As with the usual elimination of $\exists^\infty$, one can easily show that $\phi(\uw)$ must be equivalent to the formula $\exists^{>q} \ux\!\in\! \VV : \psi(\ux; \uw)$ for some $q \in \NN$ that depends on $\psi(\ux; \uw)$, and that it is enough to show elimination in a single variable $x$.

\begin{lemma} \label{lemma_aclCond_Hfour}
    If $T$ satisfies \aclCond{} and eliminates $\exists^\infty x \!\in\! \VV$, then $T$ also satisfies \Hfour{}.
\begin{proof}
    Given an $L$-formula $\psi(\ux; \uw)$ with $\ux = (x_1, \dots, x_m)$, we claim that the following formula is the one from \Hfour{} (see Definition \ref{def_hfour}):
    $$
    \sigma_\psi(\uw) := \exists^\infty x_1 \!\in\! \VV : \dots \exists^\infty x_m \!\in\! \VV : \psi(\ux; \uw).
    $$
    Assume that $\mm \models T$ and $\ud \in M$ are given, and that there are an elementary extension $\mm' \succ \mm$ and a tuple $\uv' \in \VV'$ linearly independent over $\VV$ with $\mm' \models \psi(\uv'; \ud)$.
    Without loss $\mm'$ is sufficiently saturated.
    We need to show that $\sigma_\psi(\ud)$ holds.
    By \aclCond{}, the tuple $\uv'$ must also be $\acl_L$-independent over $M$.
    Now assume we have already shown
    $$
    \exists^\infty x_{k+1} \!\in\! \VV : \dots \exists^\infty x_m \!\in\! \VV : \psi(y_1,\dots,y_k, x_{k+1},\dots,x_m; \ud) \in \tp(v'_1, \dots, v'_k/ \ud).
    $$
    Since $v'_k \not\in \acl_L(v'_1\dots v'_{k-1}\ud)$, we have $\tp(v'/v'_1, \dots, v'_{k-1}\ud) = \tp(v'_k/v'_1, \dots, v'_{k-1}\ud)$ for infinitely many $v' \in \VV'$.
    We immediately see that
    $$
    \exists^\infty x_{k} \!\in\! \VV : \dots \exists^\infty x_m \!\in\! \VV : \psi(y_1,\dots,y_{k-1}, x_k,\dots,x_m; \ud) \in \tp(v'_1, \dots, v'_{k-1}/ \ud)
    $$
    also holds.
    Repeating the above yields $\mm \models \sigma_\psi(\ud)$.

    Now assume that $\sigma_\psi(\ud)$ holds for some $\mm \models T$ and $\ud \in M$.
    By compactness, we can construct a chain $\mm \prec \mm_1 \prec \dots \prec \mm_m := \mm'$ of elementary extensions with elements $v'_k \in \VV_k \setminus \VV_{k-1}$ such that
    $
    \mm' \models \psi(v'_1, \dots, v'_m; \ud)
    $
    holds.
    Notice that, by construction, the tuple $\uv' := (v'_1, \dots, v'_m) \in \VV'$ is linearly independent over $\VV$.
\end{proof}
\end{lemma}

\noindent Obviously, the converse of Lemma \ref{lemma_aclCond_Hfour} does not hold:

\begin{example}
    $\RCF{}$ with $(\VV, 0, +, (q\cdot)_{q \in \QQ})^\rr := (\rr_{>0}, 1, \cdot, (x \mapsto x^q)_{q \in \QQ})$ satisfies \Hfour{}, but not \aclCond{}.
\end{example}



\noindent It is easy to check that if $T$ satisfies \Hfour{}, then $T$ must also eliminate $\exists^\infty x \!\in\! \VV$.
In fact, we can show something even stronger, which is similar to Remark 1.10 in \cite{dEl21b}:

\begin{theorem} \label{obser_no_elim_exist_inf_no_model_companion}
    If $T$ does not eliminate $\exists^\infty x \in \VV$, then the model companion of $T^C_\theta$ does not exist for any non-trivial $C \in \Cc$.
\begin{proof}
    Suppose that $T$ does not eliminate $\exists^\infty x \!\in\! \VV :\psi(x; \uw)$ and that $T^C_\theta$ has a model companion $T\theta^C_*$ for some fixed non-trivial $C \in \Cc$.
    Without loss of generality, assume that $\psi(x; \ud)$ defines a subset of $\VV \setminus\set{0}$ for any $\ud \in \mm \models T$.
    Fix $i \in \omega$. Choose a sufficiently saturated model $\mm_i \models T$ and a tuple $\ud_i \in M_i$ such that $\psi(x; \ud_i)$ defines a finite set with at least $i$ elements. Let $\bb^0 := \set{v^0_\alpha : \alpha \in I}$ be a basis of $\spanA{\psi(\VV_i; \ud_i)}{K}$.
    We define a partial $C$-endomorphism on a larger linearly independent set as follows.
    If $C$ is algebraic, extend $\bb^0$ to a linearly independent set $\bb := \set{v^j_\alpha : \alpha \in I, 0 \leq j < \deg(\mipo(C))}$ and define $\theta_i$ on $\spanA{\bb}{K}$ by setting
    $$
    \theta_i(v^j_\alpha) := \begin{cases}
        v^{j+1}_\alpha & \text{if $j < \deg(\mipo(C)) - 1$}, \\
        \sum\nolimits_{k= 0}^{\deg(\mipo(C)) - 1} - (\mipo(C))_k \cdot v^{k}_\alpha & \text{if $j = \deg(\mipo(C))-1$}.
    \end{cases}
    $$
    If $C$ is transcendental, extend $\bb^0$ to a linearly independent set $\bb := \set{v^j_\alpha : \alpha \in I,\ j \in \omega}$ and define $\theta_i$ on $\spanA{\bb}{K}$ by setting $\theta_i(v^j_\alpha) := v^{j+1}_\alpha$ for all $j \in \omega$.
    One can easily check that this defines a $C$-endomorphism on $\spanA{\bb}{K} \subseteq \VV_i$ in any case.
    Now we can extend $\theta_i$ to all of $\VV_i$ using our \standartConstruction{}.
    Since $(\mm_i, \theta_i)$ is contained in an existentially closed model of $T^C_\theta$, we can also assume that $(\mm_i, \theta_i)$ is existentially closed, i.e., a model of $T\theta^C_*$.
    Fix any $v \in \psi(\VV_i; \ud_i)$.
    By construction of $\theta_i$, we have $v \in \spanA{v^0_\alpha : \alpha \in I}{K}$ and $\theta_i(v) \in \spanA{v^1_\alpha : \alpha \in I}{K}$.
    Since $C$ is non-trivial, the spaces $\spanA{v^0_\alpha : \alpha \in I}{K}$ and $\spanA{v^1_\alpha : \alpha \in I}{K}$ have trivial intersection.
    Since $0 \notin \psi(\VV; \ud_i)$, we obtain $\theta_i(v) \not\in \psi(\VV; \ud_i)$. Together with compactness, this shows that the type
    $$
    \set{|\psi(\VV; \uw)| \geq i \wedge \forall x \in \VV : \psi(x; \uw) \rightarrow \neg\psi(\theta(x); \uw) : i \in \omega}
    $$
    has a realization $\ud$ in some model $(\mm, \theta) \models T\theta^C_*$.

    Since $\psi(x; \ud)$ defines an infinite subset of $\VV$, it implies no finite disjunction of non-trivial linear dependencies in $x$ over $\VV$.
    It is now easy to check that the formula $\psi(x^0; \ud) \wedge \psi(x^1; \ud)$ implies no finite disjunction of non-trivial linear dependencies in $\xvec$ over $\VV$. If $C$ is transcendental, we can see that $S(x) = \top$ is a $C$-sequence-system over $(\VV, \theta)$ that bounds $\psi(x^0; \ud) \wedge \psi(x^1; \ud)$, so we can use Theorem \ref{theorem_big_characterization} to show that there is some $v \in \VV$ with $(\mm, \theta) \models \psi(v; \ud) \wedge \psi(\theta(v); \ud)$. If $C$ is algebraic, we can also find such a $v$; for details, see the proof of Claim 4.20.1 in \cite{Chi25b}.
    In any case, the existence of such a $v$ contradicts our definition of $\ud$.
\end{proof}
\end{theorem}

\begin{corollary} \label{corollary_conclusion_hfour_aclcond}
    Suppose $T$ satisfies \aclCond{}.
    Then $T$ satisfies \Hfour{} if and only if $T$ eliminates $\exists^\infty x \in \VV$.
\end{corollary}

\subsection{The better characterization}

We now improve our characterization from Theorem \ref{theorem_big_characterization} for theories $T$ that satisfy \aclCond{} and eliminate $\exists^\infty \ux \!\in\!\VV$.
We do so by restricting the $C$-sequence-systems one has to consider:

\begin{definition} \label{def_atomic_c_ss}
    Fix $C \in \Cc$ and let $(\VV, \theta) \models \TKvsThe$ be given.
    We say that an $\LKThe(\VV)$-formula $S(x)$ is an \textbf{atomic} $C$-sequence-system over $(\VV, \theta)$ if one of the following holds:
    \begin{enumerate}[(i)]
        \item $S(x) = \top$ and $C \in \Cctrans$.
        \item $S(x) = f[\theta](x) = u$ for some $f \in \Kp{C=\infty}$ and $u \in \VV$.
        \item $S(x) = f[\theta](x) = u$ for some $f \in \Kp{0<C<\infty}$ and $u \in \Ker(f^{C-1})$.
    \end{enumerate}
    Every atomic $C$-sequence-system is indeed a $C$-sequence-system as in Definition \ref{def_c_sequence_system}.
\end{definition}

\begin{theorem} \label{theorem_easier_char_acl_cond}
    Suppose $T$ satisfies \aclCond{} and eliminates $\exists^\infty x \!\in\!\VV$.
    Then a model $(\mm, \theta) \models T^C_\theta$ is existentially closed if and only if $\theta$ is $C$-image-complete and the condition
    $$
    (\mm, \theta) \models \exists x \in \VV : \psi_\theta(x) \wedge S(x)
    $$
    holds for every atomic $C$-sequence-system $S(x)$ and every $L(M)$-formula $\psi(\xvec)$ that is bounded by $S$ and implies no finite disjunction of non-trivial linear dependencies in $\xvec$ over $\VV$.
    Furthermore, the model companion of $T^C_\theta$ exists for every kernel configuration $C \in \Cc$.
\begin{proof}
    The existence of the model companion follows from Theorem \ref{theorem_first_oder}, since $T$ satisfies \Hfour{} by Lemma \ref{lemma_aclCond_Hfour}.
    By Theorem \ref{theorem_big_characterization}, the left-to-right direction is immediate.
    It remains to prove the right-to-left direction.

    Fix a model $(\mm, \theta) \models T^C_\theta$ satisfying the right-hand side of Theorem \ref{theorem_easier_char_acl_cond}.
    Since $(\mm, \theta)$ is already $C$-image-complete, Theorem \ref{theorem_big_characterization} reduces existential closedness to the following condition:
    $$
    (\mm, \theta) \models \exists \ux \in \VV : \psi_\theta(\ux) \wedge S(\ux)
    $$
    holds for every $C$-sequence-system $S(\ux)$ and every $L(M)$-formula $\psi(\uxvec)$ that is bounded by $S(\ux)$ and implies no finite disjunction of non-trivial linear dependencies in $\uxvec$ over $\VV$.
    We prove this condition by induction on $(\rk(S), \deg(S))$, as defined in Definition \ref{def_rk_deg_def}.
    Let the given $C$-sequence-system be
    $$
        S(\ux) = \bigwedge\nolimits_{k=1}^n f_k^{q_k}[\theta](x_{\ldd, k}) = u_k
    $$
    with all notation as in Definition \ref{def_c_sequence_system}:
    We have $\ux = \ux_\li\ux_\ld$, where $\ux_\li = (x_{\li, 1}, \dots, x_{\li, m})$ and $\ux_\ld = (x_{\ld, 1}, \dots, x_{\ld, n})$; if $C$ is algebraic, then $m = 0$; every $f_k$ lies in $\Kp{0<C}$; every $q_k$ is a positive integer; and, for all $k$ with $C(f_k) < \infty$, we have $0 < q_k \leq C(f_k)$ and $u_k \in \Ker(f_k^{C-q_k})$.

    For the following arguments, we remind the reader that, given a formula $\psi(\ux; \uw)$, we let $\sigma_\psi(\uw)$ be the formula from \Hfour{} (see Definition \ref{def_hfour}) expressing that $\psi(\ux; \uw)$ implies no finite disjunction of non-trivial linear dependencies in $\ux$ over $\VV$.
    We will also use \placeholderNotation{} (Definition \ref{def_placeholder_notation}) heavily in the following arguments:
    Recall that, given a tuple $\ux = (x_1, \dots, x_m)$, we let $\uxvec := (x^i_k : 1 \leq k \leq m, i \in \omega)$ be the tuple of all placeholders for $\theta^i(x_k)$.
    Given a formula $\psi(\uxvec; \uw)$, we let $\psi_\theta(\ux; \uw)$ be the formula obtained from $\psi(\uxvec; \uw)$ by replacing every occurrence of $x^i_k$ with $\theta^i(x_k)$.
\begin{subclaim} \label{case_base_case}
    If $(\rk(S), \deg(S)) = (0,0)$, then $(\mm, \theta) \models \exists \ux \in \VV : \psi_\theta(\ux) \wedge S(\ux)$.
\begin{innerproof}
    The assumption $(\rk(S), \deg(S)) = (0,0)$ implies that $\ux$ is empty.
    Therefore, $\psi_\theta(\ux)$ is just an $L(M)$-sentence $\psi$.
    The formula $S(\ux)$ is also just an empty conjunction, i.e., $\top$.
    By assumption on $\psi(\uxvec)$, it does not imply any finite disjunction of non-trivial linear dependencies in $\uxvec$ over $\VV$.
    Since $\ux$ is empty, so is $\uxvec$.
    Thus, the only finite disjunction of non-trivial linear dependencies in $\uxvec$ over $\VV$ is the empty disjunction, namely $\bot$.
    Hence $\psi$ holds.
    Therefore, $(\mm, \theta) \models \exists \ux \in \VV : \psi_\theta(\ux) \wedge S(\ux)$.
\end{innerproof}
\end{subclaim}

\noindent From now on, assume that
\begin{align}
    (\mm, \theta) \models \exists \ux_* \in \VV : \psi_{*,\theta}(\ux_*) \wedge S_*(\ux_*) \label{tag_induction_hypo}
\end{align}
holds for every $C$-sequence-system $S_*(\ux_*)$ with $(\rk(S_*), \deg(S_*)) <_\Lex (\rk(S), \deg(S))$, and every $L(M)$-formula $\psi_*(\uxvec{}_*)$ that is bounded by $S_*$ and implies no finite disjunction of non-trivial linear dependencies in $\uxvec{}_*$ over $\VV$.
We also recall that (\ref{tag_induction_hypo}) holds if $S_*(\ux_*)$ is an atomic $C$-sequence-system, by our assumptions on $(\mm, \theta)$.
For the induction step, recall that \aclCond{} implies the following:
Given $\mm' \succ \mm \models T$ and $\uv' \in \VV'$, the tuple $\uv'$ is $K$-linearly independent over $\VV$ if and only if it is $\acl_L$-independent over $M$.

\begin{subclaim} \label{case_rk_greater_0}
    If $\rk(S) > 0$, then $(\mm, \theta) \models \exists \ux \in \VV : \psi_\theta(\ux) \wedge S(\ux)$.
\begin{innerproof}
    Recall that $\rk(S) = |\ux_\li| = m$ by definition.
    In particular, $C$ is transcendental in this case.
    Let $S(\ux) = \bigwedge\nolimits_{k=1}^n f_k^{q_k}[\theta](x_{\ldd, k}) = u_k$ be as above Claim \ref{case_base_case}.
    Write
    $$
    \ux = x_{\lii, 1}(x_{\lii, 2}, \dots, x_{\lii, m})\ux_\ld =: x_{\lii, 1}\ux_*.
    $$
    Since $\psi(\uxvec)$ implies no finite disjunction of non-trivial linear dependencies in $\uxvec$ over $\VV$, we can find $\mm' \succ \mm$ and a realization $\vvec{}_{\lii, 1}'\uvvec{}'_* \in \VV'$ of $\psi(\uxvec)$ that is $K$-linearly independent over $\VV$.
    By \aclCond{}, the tuple $\uvvec{}'_*$ is $\acl_L$-independent over $\vvec{}_{\lii, 1}'M$.
    Hence we can find $\mm'' \succ \mm'$ and a realization of $\psi(\vvec{}_{\lii, 1}'\uxvec_*)$ that is linearly independent over $\VV'$.
    Therefore, $\mm' \models \sigma_{\varphi}(\vvec{}'_{\lii, 1})$, where $\sigma_\varphi(\xvec_{\lii, 1})$ is the formula from \Hfour{} for
    $
    \varphi(\uxvec{}_*; \xvec{}_{\lii, 1}) := \psi(\xvec{}_{\lii, 1}\uxvec{}_*).
    $
    The formula $\sigma_\varphi(\xvec{}_{\lii, 1})$ is bounded by the trivial atomic $C$-sequence-system $\top$; see (i) of Definition \ref{def_atomic_c_ss}.
    Since $\vvec{}_{\lii, 1}'$ realizes $\sigma_\varphi(\xvec{}_{\lii, 1})$, the formula $\sigma_\varphi(\xvec{}_{\lii, 1})$ implies no finite disjunction of non-trivial linear dependencies in $\xvec{}_{\lii, 1}$ over $\VV$.
    By applying (\ref{tag_induction_hypo}) with this trivial atomic $C$-sequence-system, there is some $v_{\lii, 1} \in \VV$ with
    $$
    (\mm, \theta) \models (\sigma_{\varphi})_\theta(v_{\lii, 1}).
    $$
    This means that $\varphi(\uxvec{}_*; (\theta^i(v_{\lii, 1}) : i \in \omega))$ implies no finite disjunction of non-trivial linear dependencies in $\uxvec{}_*$ over $\VV$.
    Moreover, $\varphi(\uxvec{}_*; (\theta^i(v_{\lii, 1}) : i \in \omega))$ is bounded by the $C$-sequence-system $S_*(\ux_*) := S(x_{\lii, 1}\ux_*)$ (i.e., we just remove the variable $x_{\li, 1}$ from $S(\ux)$; recall that the subtuple $\ux_\li \subseteq \ux$ does not actually appear in a $C$-sequence-system).
    Since $\rk(S_*) < \rk(S)$, we can apply (\ref{tag_induction_hypo}).
    This gives a tuple $\uv_* \in \VV$ with
    $$
    (\mm, \theta) \models \varphi_\theta(\uv_*; (\theta^i(v_{\lii, 1}) : i \in \omega)) \wedge S_*(\uv_*).
    $$
    Unfolding the definitions of $\varphi(\uxvec{}_*; \xvec{}_{\lii, 1})$ and $S_*(\ux_*)$, we see that $\uv := v_{\lii, 1}\uv_* \in \VV$ realizes the formula $\psi_\theta(\ux) \wedge S(\ux)$ in $(\mm, \theta)$.
\end{innerproof}
\end{subclaim}

\begin{subclaim} \label{claim_dsfdsfdsfsf}
    If $\rk(S) = 0$, then $(\mm, \theta) \models \exists \ux \in \VV : \psi_\theta(\ux) \wedge S(\ux)$.
\begin{innerproof}
    We may assume $\deg(S) > 0$, since otherwise Claim \ref{case_base_case} applies.
    The assumption $\rk(S) = 0$ means that $\ux_\li$ is empty, so we simplify notation by writing $\ux = (x_1, \dots, x_n)$ and
    $$
    S(\ux) = \bigwedge\nolimits_{k=1}^n f_k^{q_k}[\theta](x_{k}) = u_k,
    $$
    where, for each $k$, one of the following holds:
    \begin{enumerate}[(i)]
        \item $f_k \in \Kp{C=\infty}$, $q_k > 0$, and $u_k \in \VV$.
        \item $f_k \in \Kp{0<C<\infty}$, $0 < q_k \leq C(f_k)$, and $u_k \in \Ker(f_k^{C-q_k})$.
    \end{enumerate}
    Set $\ux_* := (x_2, \dots, x_n)$, so $\ux = x_1\ux_*$, and set $f := f_1$, $q := q_1$, and $u := u_1$.
    Since $\psi(\xvec{}_1\uxvec{}_*)$ implies no finite disjunction of non-trivial linear dependencies in $\xvec{}_1\uxvec{}_*$ over $\VV$, we can find $\mm' \succ \mm$ and a realization $\uvvec{}' \in \VV'$ of $\psi(\uxvec)$ that is linearly independent over $\VV$.
    We define an $L(M)$-formula $\psi_*(\xvec{}_0\xvec{}_1\uxvec{}_*)$ so that $\psi_{*,\theta}(x_0x_1\ux_*)$ is $\psi_\theta(x_1\ux_*)$ but with every occurrence of $\theta^i(x_1)$ replaced with $\chi_i[\theta](x_0) + r_i[\theta](x_1)$, where $\chi_i, r_i$ are the unique polynomials satisfying $\chi_i \cdot f^{q-1} + r_i = X^i$ and $\deg(r_i) < \deg(f^{q-1})$.
    This means that we replace every occurrence of $x_1^i$ with $\chi_i[x_0] + r_i[x_1]$, where $\chi_i[x_0] := \sum_{j=0}^{\deg(\chi_i)} (\chi_i)_j \cdot x_0^j$, and similarly for $r_i[x_1]$.
    With this definition, the tuple
    $$
    \left(\sum\nolimits_{j=0}^{\deg(f^{q-1}) + i} (f^{q-1} \cdot X^i)_j \cdot v_1^{\prime}{\!}^j : i \in \omega\right){}^\frown \vvec{}'_1{}^\frown\uvvec{}'_*
    $$
    realizes $\psi_*(\xvec{}_0\xvec{}_1\uxvec{}_*)$.
    Since $\psi(\uxvec)$ is bounded by the $C$-sequence-system $S$, the variable $x_1^i$ does not appear in $\psi(\uxvec)$ for any $i \geq \deg(f^q)$.
    It follows that $x_0^i$ does not appear in $\psi_*(\xvec{}_0\xvec{}_1\uxvec{}_*)$ for $i \geq \deg(f)$, and $x_1^i$ does not appear in $\psi_*(\xvec{}_0\xvec{}_1\uxvec{}_*)$ for $i \geq \deg(f^{q-1})$.
    Moreover, the tuple
    $$
    \left(\sum\nolimits_{j=0}^{\deg(f^{q-1}) + i} (f^{q-1} \cdot X^i)_j \cdot v_1^{\prime}{\!}^j  : i \in \omega\right){}^\frown\big(v_1^{\prime}{\!}^0, \dots, v_1^{\prime}{\!}^{\deg(f^{q-1})-1}\big){}^\frown\uvvec{}'_*
    $$
    is linearly independent over $\VV$.
    Thus $\psi_*(\xvec{}_0\xvec{}_1\uxvec{}_*)$ implies no finite disjunction of non-trivial linear dependencies in $\xvec{}_0\xvec{}_1\uxvec{}_*$ over $\VV$.
    Using \aclCond{} as in Claim \ref{case_rk_greater_0}, we find some $\vvec{}'_0 \in \VV'$ that is linearly independent over $\VV$ and satisfies $\sigma_\varphi(\xvec{}_0)$ for the $L(M)$-formula
    $$
    \varphi(\xvec{}_1\uxvec{}_*; \xvec{}_0) := \psi_*(\xvec{}_0\xvec{}_1\uxvec{}_*).
    $$
    Since $x_0^i$ does not appear in $\psi_*(\xvec{}_0\xvec{}_1\uxvec{}_*)$ for $i \geq \deg(f)$, it also does not appear in $\sigma_\varphi(\xvec{}_0)$ for $i \geq \deg(f)$.
    Therefore, $\sigma_\varphi(\xvec{}_0)$ is bounded by the atomic $C$-sequence-system $f[\theta](x_0) = u$.
    We quickly verify that this is indeed an atomic $C$-sequence-system:
    In case (i) for $k = 1$, we have $f \in \Kp{C=\infty}$, and in case (ii), we have $f \in \Kp{0<C<\infty}$ and $u \in \Ker(f^{C-1})$.
    Since (\ref{tag_induction_hypo}) also holds for atomic $C$-sequence-systems, we find some $v_0 \in \VV$ with
    $$
    (\mm, \theta) \models \sigma_{\varphi}((\theta^i(v_0) : i \in \omega)) \wedge f[\theta](v_0) = u.
    $$
    By the definition of $\sigma_\varphi(\xvec{}_0)$, the formula $\varphi(\xvec{}_1\uxvec{}_*; (\theta^i(v_0) : i \in \omega))$ implies no finite disjunction of non-trivial linear dependencies in $\xvec{}_1\uxvec{}_*$ over $\VV$.
    It is bounded by the $C$-sequence-system
    $$
    S_*([x_1]\ux_*) := \big[ f^{q-1}[\theta](x_1) = v_0\; \! \wedge\!\big] \bigwedge\nolimits_{k=2}^{n} f^{q_k}_k[\theta](x_k) = u_k,
    $$
    where the bracketed part is omitted if $q = 1$.
    Note that when $q = 1$, $\varphi(\xvec{}_1\uxvec{}_*;(\theta^i(v_0) : i \in \omega))$ is a formula only in $\uxvec{}_*$, since $x^i_1$ appears only for $i < \deg(f^{q-1})$.
    For readability, we write the rest of the argument as if $q > 1$.
    The formula $S_*(x_1\ux_*)$ is indeed a $C$-sequence-system.
    In case (i) for $k = 1$, this is clear.
    In case (ii), we have $v_0 \in \Ker(f^{C-q+1})$, since $u \in \Ker(f^{C-q})$ and $f[\theta](v_0) = u$.
    By construction, $\deg(S_*) < \deg(S)$, since $\deg(S) := \sum_{k=1}^n \deg(f_k^{q_k})$ by definition.
    Hence we can apply (\ref{tag_induction_hypo}) and find some $v_1\uv_* \in \VV$ with
    $$
    (\mm, \theta) \models \varphi_\theta(v_1\uv_*; (\theta^i(v_0) : i \in \omega)) \wedge S_*(v_1\uv_*).
    $$
    By definition, $\varphi_\theta(v_1\uv_*; (\theta^i(v_0) : i \in \omega))$ is just $\psi_{*,\theta}(v_0v_1\uv_*)$, which is obtained by replacing every occurrence of $\theta^i(x_1)$ in $\psi_\theta(x_1\uv_*)$ with
    $$
    \chi_i[\theta](v_0) + r_i[\theta](v_1) = \chi_i[\theta](f^{q-1}[\theta](v_1)) + r_i[\theta](v_1) = \theta^i(v_1).
    $$
    Therefore, $\varphi_\theta(v_1\uv_*; (\theta^i(v_0) : i \in \omega))$ is just $\psi_\theta(v_1\uv_*)$.
    Together with $f[\theta](v_0) = u$, the equation $f^{q-1}[\theta](v_1) = v_0$ implies $f^q[\theta](v_1) = u$.
    We conclude that $(\mm, \theta) \models \psi_\theta(v_1\uv_*) \wedge S(v_1\uv_*)$.
\end{innerproof}
\end{subclaim}
\noindent This completes the proof of Theorem \ref{theorem_easier_char_acl_cond}.
In the non-trivial cases, we used \aclCond{} to split $S(\ux)$ into an atomic $C$-sequence-system and a $C$-sequence-system with smaller $(\rk, \deg)$.
\end{proof}
\end{theorem}

\section{The case $T = \TKvs$}

We now study the simplest case of our construction, namely the case in which $T$ is the theory $\TKvs$ of $K$-vector spaces.
By Remark 3.4 in \cite{Chi25}, if $(\mm, \theta)$ is an existentially closed model of $T^C_\theta$, then $(\VV, 0, +, (\lambda \cdot)_{\lambda\in K}, \theta)$ is an existentially closed model of $\TKvsTheC$.
Thus, studying this case gives a ``lower bound'' for how complicated the theory $T\theta^C$ can be for general theories $T$.

\subsection{Characterization using $C$-completeness} \label{sec_c_completeness}
We now give a final characterization of existentially closed models of $\TKvsTheC$.
We show that it is equivalent to the following stronger version of $C$-image-completeness:

\begin{definition} \label{def_c_complete_new}
    Let $\theta$ be an endomorphism of $\VV$.
    We call $\theta$ \textbf{$\mathbf{C}$-complete} if all of the following hold:
    \begin{enumerate}[(i)]
        \item $\theta$ is $C$-image-complete, that is, $\theta$ is a $C$-endomorphism, as in Definition \ref{def_T_C_theta}, that additionally satisfies $\Image(f^C) = \Image(f^{C+1})$ for all $f \in \Kp{C<\infty}$.
        \item If $f \in \Kp{0<C<\infty}$, then, for every $u \in \Ker(f^{C-1})$, the set $\set{x \in \VV : f[\theta](x) = u}$ is infinite.
        \item If $f \in \Kp{C=\infty}$, then, for every $u \in \VV$, the set $\set{x \in \VV : f[\theta](x) = u}$ is infinite.
        \item If $C \in \Cctrans$, then the type $\set{\text{``$(\theta^i(x) : i \in \omega)$ is linearly independent over $\VV$"}}$ is finitely satisfiable in $(\VV, \theta)$.
    \end{enumerate}
    We call a model $(\mm, \theta) \models T^C_\theta$ \textbf{$\mathbf{C}$-complete} if its endomorphism $\theta$ is $C$-complete.
\end{definition}

\begin{remark}
    If $C$ is algebraic, it is enough to verify $\Ker(\mipo(C)) = \VV$ and (ii).
    \begin{proof}
        For $C \in \Ccalg$, any $C$-endomorphism is automatically $C$-image-complete (see (i) of Fact \ref{fact_c_image_comple}), the set $\Kp{C=\infty}$ in (iii) is empty, and (iv) holds trivially since $C \not\in \Cctrans$.
    \end{proof}
\end{remark}

\begin{theorem} \label{corollary_c_comp}
    A model $(\VV, \theta) \models \TKvsTheC$ is existentially closed if and only if it is $C$-complete.
\begin{proof}
    We first prove ``$\Rightarrow$''.
    Let $(\VV, \theta)$ be an existentially closed model of $\TKvsTheC$.
    We verify all conditions of $C$-completeness:
    \begin{enumerate}[(i)]
        \item The $C$-image-completeness of $\theta$ follows directly from (ii) of Fact \ref{fact_c_image_comple}.
        \item Let $f \in \Kp{}$ with $0 < C(f) < \infty$ and $u \in \Ker(f^{C-1})$ be given.
        For every finite $U \subseteq \VV$, the $L(\VV)$-formula $\psi_U(\xvec) := x^0 \not\in U$ does not imply any finite disjunction of non-trivial linear dependencies in $\xvec$ over $\VV$.
        Also note that $\psi_U(\xvec)$ is bounded by the $C$-sequence-system $f[\theta](x) = u$.
        By Theorem \ref{theorem_big_characterization}, we find $v \in \VV \setminus U$ with $f[\theta](v) = u$.
        If the set $\set{x \in \VV : f[\theta](x) = u}$ were finite, then taking $U$ to be this set would give a contradiction.
        \item The same argument shows that the set $\set{x \in \VV : f[\theta](x) = u}$ is infinite for every $u \in \VV$ and $f \in \Kp{C=\infty}$.
        \item Assume $C \in \Cctrans$ and define
        $$
        p(x) = \set{\text{``$(\theta^i(x) : i \in \omega)$ is linearly independent over $\VV$"}}.
        $$
        We also define $q(\xvec) := \set{\text{``$\xvec$ is linearly independent over $\VV$"}}$.
        The type $q(\xvec)$ is closed under finite conjunctions, and it is easy to check that $p(x) = \set{\psi_\theta(x) : \psi(\xvec) \in q(\xvec)}$.
        Notice that any $\psi(\xvec) \in q(\xvec)$ implies no finite disjunction of non-trivial linear dependencies in $\xvec$ over $\VV$ and is bounded by the trivial $C$-sequence-system $S(x) = \top$.
        By Theorem \ref{theorem_big_characterization}, we can find a realization of $\psi_\theta(x)$ in $\VV$.
        Hence $p(x)$ is finitely satisfiable in $(\VV, \theta)$.
    \end{enumerate}
    We now prove ``$\Leftarrow$'', using Theorem \ref{theorem_easier_char_acl_cond}, our improved characterization for theories close to $\TKvs$.
    As noted in (i) of Example \ref{example_acl_cond}, $\TKvs$ satisfies \aclCond{}.
    Since $\TKvs$ is strongly minimal, it also eliminates $\exists^\infty x\!\in\!\VV$.
    Let $S(x)$ be an atomic $C$-sequence-system over $\VV$, and let $\psi(\xvec)$ be an $\LK(\VV)$-formula that is bounded by $S$ and does not imply any finite disjunction of non-trivial linear dependencies in $\xvec$ over $\VV$.
    We need to show that
    $$
    (\VV, \theta) \models \exists x \in \VV : \psi_\theta(x) \wedge S(x)
    $$
    holds.
    By quantifier elimination, $\psi(\xvec)$ is equivalent to a finite disjunction $\bigvee_k \psi_k(\xvec)$, where each $\psi_k(\xvec)$ is a finite conjunction of atomic and negated atomic $\LK(\VV)$-formulas in $\xvec$.
    Since $\psi(\xvec)$ does not imply any finite disjunction of non-trivial linear dependencies in $\xvec$ over $\VV$, there must be some $k$ such that $\psi_k(\xvec)$ is of the form
    $$
    \bigwedge\nolimits_{l=1}^n \sum\nolimits_{i=0}^{\deg(\rho_l)} (\rho_l)_i \cdot x^i \not \in U_l
    $$
    for some nonzero polynomials $\rho_l$ and finite sets $U_l \subseteq \VV$.
    Without loss of generality, we may assume $\psi(\xvec) = \psi_k(\xvec)$.
    It is easy to check that $\psi_\theta(x) = \bigwedge\nolimits_{l=1}^n \rho_l[\theta](x) \not \in U_l$.
    We now deal with each kind of atomic $C$-sequence-system separately; see Definition \ref{def_atomic_c_ss}.
    \begin{enumerate}[(i)]
        \item Assume $S(x) = \top$.
        In this case, $C$ must be transcendental, and $\psi_\theta(x)$ is a finite conjunction of formulas from the type in (iv) of Definition \ref{def_c_complete_new}.
        By the $C$-completeness of $\theta$, there is a realization $v \in \VV$ of $\psi_\theta(x)$.
        Since $S(x) = \top$, this $v$ realizes $\psi_\theta(x) \wedge S(x)$.
        \item Assume $S(x)$ is of the form $f[\theta](x) = u$ for some $f \in \Kp{C=\infty}$.
        Fix $l$ with $1 \leq l \leq n$.
        Since $\psi(\xvec)$ is bounded by $S$, we have $\deg(\rho_l) < \deg(f)$.
        Because $f \in \Kp{}$ is irreducible and $\rho_l$ is nonzero, we obtain $\gcd(\rho_l, f) = 1$.
        Bézout's identity for polynomials over $K$ gives $\chi_1, \chi_2 \in K[X]$ with $1 = \gcd(\rho_l, f) = \chi_1 \cdot \rho_l + \chi_2 \cdot f$.
        Let $u' \in U_l$ be given, and assume that $v \in \VV$ satisfies $\rho_l[\theta](v) = u'$ and $f[\theta](v) = u$.
        Then
        $$
        v = 1[\theta](v) = \gcd(\rho_l, f)[\theta](v) = \chi_1[\theta](\rho_l[\theta](v)) + \chi_2[\theta](f[\theta](v)) = \chi_1[\theta](u') + \chi_2[\theta](u),
        $$
        so
        $
        \set{x \in \VV : \rho_l[\theta](x) \in U_l \wedge f[\theta](x) = u} \subseteq \chi_1[\theta](U_l) + \chi_2[\theta](u)
        $
        is finite.
        Since this holds for every $l$ and the set $\set{x \in \VV : f[\theta](x) = u}$ is infinite by $C$-completeness, there is some
        $
        v \in \set{x \in \VV : \bigwedge\nolimits_{l=1}^n \rho_l[\theta](x) \not\in U_l \wedge f[\theta](x) = u}.
        $
        Such a $v$ realizes $\psi_\theta(x) \wedge S(x)$.
        \item Assume $S(x)$ is of the form $f[\theta](x) = u$ with $f \in \Kp{0<C<\infty}$ and $u \in \Ker(f^{C-1})$.
        The same argument as in (ii), using condition (ii) of Definition \ref{def_c_complete_new} in place of condition (iii), gives a realization of $\psi_\theta(x) \wedge S(x)$. \qedhere
    \end{enumerate}
\end{proof}
\end{theorem}

\noindent As recalled in the beginning of this section, for any existentially closed model $(\mm, \theta) \models T^C_\theta$, the definable structure $(\VV, \theta)$ is an existentially closed model of $\TKvsTheC$ and is therefore $C$-complete by Theorem \ref{corollary_c_comp}.
Also, the theory $\TKvsThe \cup \set{\text{``$\theta$ is $C$-complete"}} = \TKvs\theta^C$ is complete (see Example \ref{example_kvs}), which motivates the name $C$-completeness.

The $C$-completeness of existentially closed models of $T^C_\theta$ can be used to describe the $\LKThe$-definable sets, such as $\Ker(f^n)$, $\Image(F^C)$, and so on.
Here is an example:

\begin{lemma} \label{corollary_kernels_growing}
    If $(\mm, \theta) \models T^C_\theta$ is an existentially closed model, then $\Ker(f^m) \subsetneq \Ker(f^n)$ for any $f \in \Kp{0<C}$ and $m, n \in \NN$ with $0 \leq m < n \leq C(f)$.
\begin{proof}
    By Theorem \ref{corollary_c_comp}, we know that the endomorphism $\theta$ is $C$-complete.
    Using (ii) and (iii) of Definition \ref{def_c_complete_new}, we can find an element $v \in \VV$ with $f[\theta](v) = u$ for any $u \in \Ker(f^m)$ and any $m < C(f)$.
    This implies $\Ker(f^{m+1}) \supsetneq \Ker(f^m)$ for every $m < C(f)$.
    Therefore $\Ker(f^m) \subsetneq \Ker(f^n)$ whenever $0 \leq m < n \leq C(f)$.
\end{proof}
\end{lemma}

\subsection{Definable subgroups and homomorphisms}

In this section, we study the $\varnothing$-definable subgroups of $(\VV^n, 0, +)$ in $\TKvs\theta^C$ and the $\varnothing$-definable homomorphisms between them.
The author believes that most of the arguments presented in this section are well known, despite not having found them in the literature.
Let $(\VV, \theta) \models \TKvs\theta^C$ be given, and let $G \subseteq \VV^n$ be a $\varnothing$-definable subgroup of $(\VV^n, 0, +)$.

\begin{observation} \label{obser_ee_to_direkt_sum}
    The structure $\bigoplus\nolimits_{i=1}^q (\VV, \theta)$ is $C$-complete, and hence a model of $\TKvs\theta^C$.
    Notice that
    $$
    \iota_i \colon (\VV, \theta) \to \bigoplus\nolimits_{i=1}^q (\VV, \theta);\quad x \mapsto (0, \dots, x, \dots, 0)
    $$
    (where $x$ is in the $i$-th coordinate) is an embedding.
    Since $\TKvs\theta^C$ is model-complete, this embedding is elementary, and hence maps elements of $X$ to $X^{\bigoplus_{i=1}^q(\VV, \theta)}$ for any $\varnothing$-definable set $X$.
    In particular, this embedding maps elements of $G$ to $G^{\bigoplus_{i=1}^q(\VV, \theta)}$.
\end{observation}

\noindent Let $\uv_1, \dots, \uv_q \in G$ be elements of the given group.
We may sometimes write $(\uv_1, \dots, \uv_q)$ for the element $\iota_1(\uv_1) + \cdots + \iota_q(\uv_q) \in G^{\bigoplus_{i=1}^q(\VV, \theta)}$.
The elementary equivalence of $(\VV, \theta)$ and $\bigoplus\nolimits_{i=1}^q (\VV, \theta)$ is a well-known concept in the model theory of modules (see $T = T^{\aleph_0}$ in \cite{Pre88}) and implies the following two lemmas.

\begin{lemma} \label{lemma_remove_ineq}
    Let $\psi_0(\ux) := \bigwedge_{k=1}^{m_0} \sum_{l=1}^n r_{0, k, l}(x_l) \neq 0$ and $\psi_1(\ux) := \bigwedge_{k=1}^{m_1} \sum_{l=1}^n r_{1, k, l}(x_l) = 0$ be given such that $\psi_0(\ux) \wedge \psi_1(\ux)$ is consistent.
    Suppose $\TKvs\theta^C \models \forall \ux : (\psi_0(\ux) \wedge \psi_1(\ux)) \rightarrow \ux \in G$.
    Then $\TKvs\theta^C \models \forall \ux : \psi_1(\ux) \rightarrow \ux \in G$.
\begin{proof}
    Let $\uv$ be a realization of $\psi_0(\ux) \wedge \psi_1(\ux)$, and let $\uu$ be any realization of $\psi_1(\ux)$.
    Work in the model $(\VV', \theta') := \bigoplus_{i=1}^2 (\VV, \theta) \models \TKvs\theta^C$, and let $G' := G^{(\VV', \theta')}$.
    The element $(\uv, \uu)$ lies in $G'$ because it satisfies $\psi_0(\ux) \wedge \psi_1(\ux)$: the conjunction of inequations $\psi_0(\ux)$ holds in the first coordinate, and the conjunction of equations $\psi_1(\ux)$ holds in both coordinates.
    Since $\uv$ satisfies $\psi_0(\ux) \wedge \psi_1(\ux)$, we have $\uv \in G$, and hence $(\uv, \uzero) = \iota_1(\uv) \in G'$, since $\iota_1$ is elementary.
    Thus $\iota_2(\uu) = (\uzero, \uu) = (\uv, \uu) - (\uv, \uzero)$ lies in $G'$.
    Since $\iota_2$ is elementary, we obtain $\uu \in G$.
\end{proof}
\end{lemma}

\begin{lemma} \label{lemma_stack_inequal}
    Suppose there are elements $\uv_1, \dots, \uv_q \!\in\! G$ that individually satisfy some $\LRC$-inequations, i.e., we have
    $
    (\VV, \theta) \models \sum\nolimits_{l=1}^n r_{i, l}(v_{i,l}) \neq 0
    $
    for each $i \in \set{1, \dots, q}$.
    Then there is an element $\uv \in G$ that satisfies all these inequations at once, i.e.,
    $$
    (\VV, \theta) \models \bigwedge\nolimits_{i=1}^q \sum\nolimits_{l=1}^n r_{i, l}(v_l) \neq 0.
    $$
\begin{proof}
    The element $(\uv_1, \dots, \uv_q) = \sum_{i=1}^q \iota_i(\uv_i) \in G^{\bigoplus_{i=1}^q(\VV, \theta)}$ satisfies all these inequations.
    Since the structures $\bigoplus_{i=1}^q(\VV, \theta)$ and $(\VV, \theta)$ are elementarily equivalent, we conclude.
\end{proof}
\end{lemma}

\begin{theorem} \label{theorem_grp_fml}
    The group $G$ is defined by a formula of the form $\bigwedge_{k=1}^m \sum_{l=1}^n r_{k, l}(x_l) = 0$, where the coefficients $r_{k, l}$ are in $R_C$.
\begin{proof}
    By quantifier elimination (see Example \ref{example_kvs}), the group $G$ is defined by a formula of the form
    $$
    \bigvee\nolimits_{i=1}^q \bigwedge\nolimits_{k=1}^{m_i} \big(\sum\nolimits_{l=1}^{n} r_{i, k, l}(x_l) = 0\big)^{\epsilon_{i, k}},
    $$
    where each $\epsilon_{i, k} \in \set{0,1}$, and we set $\varphi(\ux)^0 := \neg \varphi(\ux)$ and $\varphi(\ux)^1 := \varphi(\ux)$ for any formula $\varphi(\ux)$.
    Let $\phi_i(\ux) := \bigwedge\nolimits_{k=1}^{m_i} \big(\sum\nolimits_{l=1}^{n} r_{i, k, l}(x_l) = 0\big)^{\epsilon_{i, k}}$.
    For each $i \in \set{1, \dots, q}$, define
    $$
    \phi^*_i(\ux) := \bigwedge\nolimits_{k \in \kk_i} \sum\nolimits_{l=1}^{n} r_{i, k, l}(x_l) = 0,
    $$
    where $\kk_i \subseteq \set{k \in \set{1, \dots, m_i} : \epsilon_{i, k} = 1}$ is maximal such that, for each $j$, the formula $\phi_j(\ux)$ implies $\bigwedge\nolimits_{k \in \kk_i} \sum\nolimits_{l=1}^{n} r_{i, k, l}(x_l) = 0$ modulo $\TKvs\theta^C$.
    Define $\phi(\ux) := \bigwedge_{i=1}^q \phi^*_i(\ux)$.
    By definition, $\phi(\ux)$ is implied by every $\phi_i(\ux)$, so each $\phi_i(\ux)$ is equivalent to a formula of the form
    $$
    \zeta_i(\ux) := \phi(\ux) \wedge \Big(\bigwedge\nolimits_{k=1}^{s_i} \sum\nolimits_{l=1}^n r'_{i, k, l}(x_l) = 0\Big) \wedge \Big( \bigwedge\nolimits_{k=1}^{t_i} \sum\nolimits_{l=1}^n r''_{i, k, l}(x_l) \neq 0 \Big),
    $$
    where each equation $\sum\nolimits_{l=1}^n r'_{i, k, l}(x_l) = 0$ is not implied by some $\phi_j(\ux)$.
    To see this, note that $\phi_i(\ux) \equiv \phi(\ux) \wedge \phi_i(\ux)$, then sort all conjuncts of $\phi_i(\ux)$ by $\epsilon_{i, k}$, and finally remove all equations that are already implied by $\phi(\ux)$.
    By Lemma \ref{lemma_stack_inequal}, there is an element $\uv \in G$ that satisfies $\sum_{l=1}^n r'_{i, k, l}(v_l) \neq 0$ for all $i \in \set{1, \dots, q}$ and $k \in \set{1, \dots, s_i}$.
    Since $(\VV, \theta) \models \zeta_i(\uv)$ for some $i$, this $i$ must satisfy $s_i = 0$.
    Now apply Lemma \ref{lemma_remove_ineq} with $\psi_0(\ux) := \bigwedge\nolimits_{k=1}^{t_i} \sum\nolimits_{l=1}^n r''_{i, k, l}(x_l) \neq 0$ and $\psi_1(\ux) := \phi(\ux)$.
    Since $s_i = 0$, we have $\psi_0(\ux) \wedge \psi_1(\ux) = \zeta_i(\ux)$, and therefore any realization of $\phi(\ux)$ is already in $G$.
    By the definition of $\phi(\ux)$, any element of $G$ realizes $\phi(\ux)$.
    We conclude that the conjunction of $\LRC$-equations $\phi(\ux)$ defines $G$.
\end{proof}
\end{theorem}

\noindent Notice that Theorem \ref{theorem_grp_fml} is wrong if one replaces the assumption that $G$ is $\varnothing$-definable with the assumption that $G$ is definable with parameters.
For example, assume that $K$ has characteristic $p > 0$; take any $\varnothing$-definable subgroup $H \subsetneq \VV^n$, any $\uv \in \VV^n \setminus H$, and define $G' := \bigcup_{k = 0}^{p-1} H + k \cdot \uv$.
The group $G'$ will, in general, not be definable by a formula of the form $\bigwedge_{k=1}^m \sum_{l=1}^n r_{k, l}(x_l) = 0$.

\begin{corollary}
    The group $G$ is actually an $R_C$-submodule of $(\VV^n, 0, +, (r \cdot)_{r\in R_C})$.
\end{corollary}

\noindent In \cite{Chi25}, we claimed that any $\LKThe$-definable endomorphism of $\VV$ is an element of $R_C$.
We can now prove this:

\begin{theorem}
    \label{theorem_homomorphism_LKthe} Let $h \colon G \to H$ be a $\varnothing$-definable homomorphism between two $\varnothing$-definable subgroups $G \leq (\VV^m, 0, +)$ and $H \leq (\VV^n, 0, +)$.
    Then $h$ is given by a matrix of elements in $R_C$, i.e.,
    $$
    h(\uv) = \uv \cdot \begin{tikzpicture}[baseline=(Frame.base)]
    \drawText{0}{0}{r_{1,1}}
    \drawText{0}{2}{r_{m,1}}
    \drawText{2}{0}{r_{1,n}}
    \drawText{2}{2}{r_{m,n}}
    \drawHDots{1}{0}{1}
    \drawHDots{1}{2}{1}
    \drawVDots{0}{1}{1}
    \drawVDots{2}{1}{1}
    \drawBorder{0}{0}{3}{3}
    \end{tikzpicture}
    $$
    for any $\uv \in G$, where we treat elements of $\VV^n$ and $\VV^m$ as $1 \times n$ and $1 \times m$ matrices.
    In particular, any $\LKThe$-definable endomorphism of $\VV$ in $T\theta^C$ is already in $R_C$.
\begin{proof}
    Let $\pi_k \colon \VV^n \to \VV ; (x_1, \dots, x_n) \mapsto x_k$ be a coordinate projection.
    Clearly, $\pi_k \circ h \colon G \to \VV$ is also a group homomorphism, and $h$ is completely determined by $\pi_1 \circ h, \dots, \pi_n \circ h$.
    Thus we may assume $n = 1$.
    Since $\dcl_{\LKThe} = \cl_\theta$ (see Example \ref{example_kvs}), for any $\uv \in G$ there are elements $r_{\uv, 1}, \dots, r_{\uv, m} \in R_C$ such that $h(\uv) = \sum\nolimits_{k=1}^m r_{\uv, k}(v_k)$.
    By compactness, there are $q \in \NN$ and $r_{i, k} \in R_C$ for all $i \in \set{1, \dots, q}$ and $k \in \set{1, \dots, m}$ such that
    \begin{align}
        \TKvs\theta^C \models \forall \ux \in G : \bigvee\nolimits_{i=1}^q h(\ux) = \sum\nolimits_{k=1}^m r_{i, k}(x_k). \label{tag_disjunction_endo}
    \end{align}
    Now assume, toward a contradiction, that for every $i \in \set{1, \dots, q}$ there is some $\uv_i \in G$ with $h(\uv_i) \neq \sum\nolimits_{k=1}^m r_{i, k}(v_{i, k})$.
    Since the graph of $h$ is a subgroup of $(\VV^{m+1}, 0, +)$, Lemma \ref{lemma_stack_inequal} gives an element $\uv' \in G$ with $h(\uv') \neq \sum\nolimits_{k=1}^m r_{i, k}(v'_{k})$ for all $i$.
    This contradicts (\ref{tag_disjunction_endo}).
    Therefore there is some $i$ such that $h(\uv) = \sum\nolimits_{k=1}^m r_{i, k}(v_{k})$ holds for all $\uv \in G$.
\end{proof}
\end{theorem}

\noindent Notice that Theorem \ref{theorem_homomorphism_LKthe} can fail for reasons similar to those for Theorem \ref{theorem_grp_fml} if one allows $G$, $H$, and $h$ to be definable with parameters. However, in the case where $G = \VV^m$ and $H = \VV^n$, one can actually show that the above also holds if $h$ is definable with parameters.

Theorem \ref{theorem_homomorphism_LKthe} with $G = H = \VV$ implies that, for general theories $T$, the $\LKThe$-definable endomorphisms of $(\VV, 0, +, (\lambda\cdot)_{\lambda \in K})$ are precisly the elements of $R_C$.
This does not necessarily mean that all endomorphisms of $(\VV, 0, +, (\lambda\cdot)_{\lambda \in K})$ $L_\theta$-definable in $T\theta^C$ are also in $R_C$.
In fact, if $T$ defines any endomorphism $h$ of $(\VV, 0, +, (\lambda\cdot)_{\lambda \in K})$ that is not of the form $\lambda \cdot$ for some $\lambda \in K$ (see, e.g., (iii) of Example \ref{example_acl_cond}), $h$ is already a counterexample. In this case, any map that is $\set{+, \circ}$-generated by $h$ and $R_C$ is obviously also an $L_\theta$-definable endomorphism of $(\VV, 0, +, (\lambda\cdot)_{\lambda \in K})$. This leads to the following question:

\begin{question} \label{question_def_end}
    Suppose that $T$ satisfies \Hfour{}.
    Is it true that the $L_\theta$-definable endomorphisms of $(\VV, 0, +, (\lambda\cdot)_{\lambda \in K})$ are exactly those generated under $\set{+, \circ}$ by $R_C$ and all $L$-definable endomorphisms of $(\VV, 0, +, (\lambda\cdot)_{\lambda \in K})$?
\end{question}

\noindent The author believes that the answer to Question \ref{question_def_end} is positive.
Notice that if there are no $L$-definable endomorphisms apart from multiplication by elements of $K$, e.g., as in $\RCF{}$ with $(\VV, 0, +, (q\cdot)) = (\rr_{>0}, 1, \cdot, (x \mapsto x^q)_{q \in \QQ})$, then the question above asks whether $R_C$ consists of all $L_\theta$-definable endomorphisms.
The same question for $\ACF{}$ (with $C = C_\infty$) has already been asked in \cite{dEl25} (see Question 5.8 there).

\begin{remark} \label{rem_neo_stab_properties}
    Fix any kernel configuration $C$. Since $\TKvs\theta^C$ is a complete theory of $R_C$-modules (see Example \ref{example_kvs}), this theory is stable (see, e.g., (1) of Theorem 2.1 in \cite{Zie84}). Example \ref{example_kvs} also mentions that $\TKvs\theta^C$ is strongly minimal if $R_C$ is a field. It is easy to check that $\TKvs\theta^C$ cannot be strongly minimal otherwise, since $\Ker(f)$ will be neither finite nor cofinite for some $f \in \Kp{}$. It is also easy to check that there is an infinite descending chain of subgroups/submodules of $\VV$ if and only if $\Kp{0<C<\infty}$ is infinite (take a sequence $(f_i : i \in \omega)$ in $\Kp{0<C<\infty}$ and define $G_q$ with the formula $\bigwedge_{i=0}^{q-1} \pi_{\Ker(f_i^C)}(x) = 0$). Combining everything above with, e.g., Theorem 3.1 in \cite{Pre88}, one can obtain the following:
    \begin{enumerate}[(i)]
        \item $\TKvs\theta^C$ is strongly minimal if and only if $R_C$ is a field.
        \item $\TKvs\theta^C$ is totally transcendental and superstable if $\Kp{0<C<\infty}$ is finite.
        \item $\TKvs\theta^C$ is strictly stable if $\Kp{0<C<\infty}$ is infinite.
    \end{enumerate}
    An interesting observation is that the tameness of $\TKvs\theta^C$ seems to correspond somewhat to the notational complexity introduced when working with the given kernel configuration $C$.
\end{remark}

\section{Replacing $\theta$ with a predicate for $\Ker(\rho)$ or $\Image(\rho)$}

In this section, we axiomatize some reducts of an existentially closed model $(\mm, \theta) \models T^C_\theta$. 
Recall that, when $\theta$ is clear from the context, we write $\Ker(\rho)$ instead of $\Ker(\rho[\theta])$, and similarly for $\Image(\rho)$.
Given a model $(\mm, \theta) \models T^C_\theta$ and some $\rho \in K[X]$, the structures $(\mm, \Ker(\rho))$ and $(\mm, \Image(\rho))$ are reducts of $(\mm, \theta)$, where $\Ker(\rho)$ and $\Image(\rho)$ are predicates for the respective sets.
If $(\mm, \theta) \models T^C_\theta$ is existentially closed, we will show that these two reducts are either trivial, i.e., interdefinable with $\mm$, or expansions of $\mm$ by generic vector subspaces, as defined in Section \ref{sec_gen_subvector}.

\subsection{Generic expansions by a subvector space} \label{sec_gen_subvector}

Before studying the reducts obtained by removing $\theta$ and adding a predicate for $\Ker(\rho)$ or $\Image(\rho)$, we first study expansions by a predicate for a generic vector subspace.
Everything in this subsection is from Section 1 of \cite{dEl21b}, but specialized to the case where $T_0$ is the theory of $K$-vector spaces and generalized to the case where $T_0$ is not a reduct of $T$, but lives in an arbitrary definable subset of $\mm^n$. Since all proofs generalize easily, we omit them.

\begin{definition}
    Let $V$ be a predicate with the same arity as $\VV$ and set $L_V := L \cup \set{V}$.
    We define the $L_V$-theory
    $$
    T_V := T \cup \set{\text{``$V$ defines a vector subspace of $\VV$"}}.
    $$
\end{definition}

\noindent It is easy to see that $T_V$ is inductive, so the model companion of $T_V$ is exactly the first-order axiomatization of existentially closed models of $T_V$, if such an axiomatization exists (and otherwise no model companion exists). Throughout this section and Section \ref{sec_kernel_image}, we may write $\ulambda{} \cdot \uv := \sum\nolimits_{l=1}^n \lambda_l \cdot v_l$ for tuples $\ulambda{} = (\lambda_1, \dots, \lambda_n) \in K^n$ and $\uv = (v_1, \dots, v_n) \in \VV^n$.

\begin{theorem} \label{theorem_t_v_exist_closed}
    A model $(\mm, V) \models T_V$ is existentially closed if and only if for any $L(M)$-formula $\psi(\ux_0\ux_1)$ that implies no finite disjunction of non-trivial linear dependencies in $\ux_0\ux_1$ over $\VV$, any tuples $\ulambda{}_{1}, \dots, \ulambda{}_{m} \in K^{|\ux_1|} \setminus \set{\uzero}$, and any $u_1, \dots, u_m \in \VV$, we have
    $$
    (\mm, V) \models \exists \ux_0\ux_1 \in \VV : \psi(\ux_0\ux_1) \wedge \ux_0 \in V \wedge \bigwedge\nolimits_{k=1}^m \ulambda{}_{k} \cdot \ux_1 + u_k \not\in V.
    $$
\end{theorem}

\begin{corollary} \label{corollary_t_v_companion}
    If $T$ satisfies \Hfour{}, then the theory $T_V$ has a model companion, which we denote with $TV$.
    It is axiomatized by $T_V$ and, for every $L$-formula $\psi(\ux_0\ux_1; \uw)$, every $m \geq 0$, and every choice of tuples $\ulambda{}_1, \dots, \ulambda{}_m \in K^{|\ux_1|} \setminus \set{\uzero}$, the sentence
    $$
    \forall \uw : \forall \uy \in \VV : \Big(\sigma_\psi(\uw) \rightarrow \Big(\exists \ux_0\ux_1 \in \VV : \psi(\ux_0\ux_1; \uw) \wedge \ux_0 \in V \wedge \bigwedge\nolimits_{k=1}^m \ulambda{}_k \cdot \ux_1 + y_k \not\in V\Big)\Big).
    $$
    Here $\uy = (y_1, \dots, y_m)$, and $\sigma_\psi(\uw)$ is the formula obtained by applying \Hfour{} to $\psi(\ux_0\ux_1; \uw)$.
\end{corollary}

\begin{lemma} \label{lemma_completions_TV}
Assume that $T$ satisfies the condition \Hfour{}.
Two models $(\mm_1, V_1), (\mm_2, V_2) \models TV$ are elementarily equivalent if and only if there is an $L_V$-isomorphism
$$
\iota \colon (\acl^{\mm_1}(\varnothing), V_1 \cap \acl^{\mm_1}(\varnothing)) \to (\acl^{\mm_2}(\varnothing), V_2 \cap \acl^{\mm_2}(\varnothing))
$$
that is also $L$-elementary with respect to $\mm_1$ and $\mm_2$.
\end{lemma}

\subsection{Kernels and Images} \label{sec_kernel_image}

Assume that $(\mm, \theta) \models T_\theta^C$ is an existentially closed model.
Using the axiomatization of existentially closed models of $T_V$ from Section \ref{sec_gen_subvector} above, we can now axiomatize the reducts of the form $(\mm, \Ker(\rho))$.

\begin{theorem} \label{theorem_kernel_generic}
    If $\rho$ is a polynomial and $(\mm, \theta) \models T_\theta^C$ is an existentially closed model, then one of the following holds:
    \begin{enumerate}[(i)]
        \item $(\mm, \Ker(\rho)) = (\mm, \set{0})$;
        \item $(\mm, \Ker(\rho)) = (\mm, \VV)$;
        \item $(\mm, \Ker(\rho))$ is an existentially closed model of $T_V$.
    \end{enumerate}
\begin{proof}
    Case (i) occurs if and only if $\Fac(\rho) \subseteq \Kp{C=0}$ and $\rho \neq 0$.
    Case (ii) occurs if and only if $\rho = 0$ or if $C$ is algebraic and $\mipo(C) \mid \rho$.
    Assume that we are in neither of these cases, and let $\psi(\ux_0\ux_1)$ be an $L(M)$-formula that implies no finite disjunction of non-trivial linear dependencies in $\ux_0\ux_1$ over $\VV$.
    In order to prove that $(\mm, \Ker(\rho))$ is an existentially closed model of $T_V$, we need to show that
    $$
    (\mm, \theta) \models \exists \ux_0\ux_1 \in \VV : \psi(\ux_0\ux_1) \wedge \ux_0 \in \Ker(\rho) \wedge \bigwedge\nolimits_{k=1}^m \ulambda{}_k \cdot \ux_1 + u_k \not\in \Ker(\rho)
    $$
    holds for any $\ulambda{}_1, \dots, \ulambda{}_m \in K^{|\ux_1|} \setminus \set{\uzero}$ and $u_1, \dots, u_m \in \VV$.
    For this, we define an $L(M)$-formula $\psi'(\uxvec{}_0\uxvec{}_1)$ and another $\LKThe$-formula $S(\ux_0\ux_1)$ as follows:
    \begin{enumerate}[(a)]
        \item If $C$ is transcendental and there is some $f \in \Kp{0<C}$ with $f \mid \rho$, we set
        \begin{align*}
            &\psi'(\uxvec{}_0\uxvec{}_1) := \psi(\ux^0_0\ux^0_1) \wedge \bigwedge\nolimits_{k=1}^m \ulambda{}_k \cdot \rho[\ux_1] + \rho[\theta](u_k) \neq 0 \\
            &\hspace{200pt}\text{and} \quad S(\ux_0\ux_1) := \bigwedge\nolimits_{l=1}^{|\ux_0|} f[\theta](x_{0, l}) = 0.
        \end{align*}
        Here $\rho[\ux_1]$ is the $|\ux_1|$-tuple of $\LK$-terms consisting of the entries $\rho[x_{1, l}] := \sum\nolimits_{i=0}^{\deg(\rho)} (\rho)_i \cdot x_{1, l}^i$.
        \item If $C$ is algebraic and there are $f, g \in \Kp{0<C}$ with $f \mid \rho$ and $g \nmid \rho$, we let $r$ be the remainder of $\rho$ divided by $g$ and set
        \begin{align*}
            &\psi'(\uxvec{}_0\uxvec{}_1) := \psi(\ux^0_0\ux^0_1) \wedge \bigwedge\nolimits_{k=1}^m \ulambda{}_k \cdot r[\ux_1] + \rho[\theta](u_k) \neq 0 \\
            &\hspace{100pt}\text{and} \quad S(\ux_0\ux_1) := \bigwedge\nolimits_{l=1}^{|\ux_0|} f[\theta](x_{0, l}) = 0 \wedge \bigwedge\nolimits_{l=1}^{|\ux_1|} g[\theta](x_{1, l}) = 0.
        \end{align*}
        \item If $C$ is algebraic and there is some $f \in \Kp{0<C}$ with $f \mid \rho$ and $f^C \nmid \rho$, we set $g := f^C$, let $r$ be the remainder of $\rho$ divided by $f^C$, and define $\psi'(\uxvec{}_0\uxvec{}_1)$ and $S(\ux_0\ux_1)$ as in case (b).
    \end{enumerate}
    In all cases, it is easy to check that $S$ is a $C$-sequence-system over $(\VV, \theta)$ as in Definition \ref{def_c_sequence_system}, setting $\ux_\li := \ux_1$ in case (a), and that $S$ bounds $\psi'(\uxvec{}_0\uxvec{}_1)$.
    Since $\psi'(\uxvec{}_0\uxvec{}_1)$ is obtained by adding some linear inequations to $\psi(\ux^0_0\ux^0_1)$, we also see that this formula implies no finite disjunction of non-trivial linear dependencies in $\uxvec{}_0\uxvec{}_1$ over $\VV$.
    By Theorem \ref{theorem_big_characterization}, there is $\uv_0\uv_1 \in \VV$ with
    $$
    (\mm, \theta) \models \psi'_\theta(\uv_0\uv_1) \wedge S(\uv_0\uv_1).
    $$
    Notice that in cases (b) and (c), $\psi'_\theta(\uv_0\uv_1)$ implies $\bigwedge\nolimits_{k=1}^m \ulambda{}_k \cdot r[\theta](\uv_1) + \rho[\theta](u_k) \neq 0$, which, together with the $\bigwedge\nolimits_{l=1}^{|\ux_1|} g[\theta](v_{1, l}) = 0$ part of $S$ and the fact that $r$ is the remainder of $\rho$ divided by $g$, implies $\bigwedge\nolimits_{k=1}^m \rho[\theta]( \ulambda{}_k \cdot \uv_1 + u_k) \neq 0$.
    The same conclusion follows trivially in case (a).
    In all cases, we have $\uv_0 \in \Ker(\rho)$, since $f \mid \rho$ in any case.
    Hence
    $$
    (\mm, \theta) \models \psi(\uv_0\uv_1) \wedge \uv_0 \in \Ker(\rho) \wedge \bigwedge\nolimits_{k=1}^m \ulambda{}_k \cdot \uv_1 + u_k \not\in \Ker(\rho).
    $$
    This finishes the proof, since, under our assumption that we are not in case (i) or (ii), it is clear that we must be in at least one of cases (a), (b), and (c).
\end{proof}
\end{theorem}

\noindent We obtain exactly the same theorem for the images:

\begin{theorem} \label{theorem_image_gen}
    If $\rho$ is a polynomial and $(\mm, \theta) \models T_\theta^C$ is existentially closed, then one of the following holds:
    \begin{enumerate}[(i)]
        \item $(\mm, \Image(\rho)) = (\mm, \set{0})$;
        \item $(\mm, \Image(\rho)) = (\mm, \VV)$;
        \item $(\mm, \Image(\rho))$ is an existentially closed model of $T_V$.
    \end{enumerate}
\begin{proof}
    In the algebraic case, this follows directly since $\Image(\rho) = \Ker(\rho')$ for some $\rho' \in K[X]$.
    From now on, we assume that $C$ is transcendental.
    We are in case (i) if and only if $\rho = 0$, and we are in case (ii) if and only if $\Fac(\rho) \subseteq \Kp{C=0} \cup \Kp{C=\infty}$ and $\rho \neq 0$.
    This means that we can assume that there is $f \in \Kp{0<C<\infty}$ with $f \mid \rho$.
    As in the proof of Theorem \ref{theorem_kernel_generic} above, let $\psi(\ux_0\ux_1)$ be an $L(M)$-formula that implies no finite disjunction of non-trivial linear dependencies in $\ux_0\ux_1$ over $\VV$.
    For any $\ulambda{}_1, \dots, \ulambda{}_m \in K^{|\ux_1|} \setminus \set{\uzero}$ and $u_1, \dots, u_m \in \VV$, we need to show that
    $$
    (\mm, \theta) \models \exists \ux_0\ux_1 \in \VV : \psi(\ux_0\ux_1) \wedge \ux_0 \in \Image(\rho) \wedge \bigwedge\nolimits_{k=1}^m \ulambda{}_k \cdot \ux_1 + u_k \not\in \Image(\rho).
    $$
    For this, define
    $$
        \psi'(\uxvec{}_0\uxvec{}_1) := \psi\big(\rho[\ux_0]\ux^0_1\big) \wedge \bigwedge\nolimits_{k=1}^m \ulambda{}_k \cdot f^{C-1}[\ux_1] + f^{C-1}[\theta] \circ \pi_{\Ker(f^C)}(u_k) \neq 0
    $$
    and $S(\ux_0\ux_1) := \bigwedge\nolimits_{l=1}^{|\ux_1|} f^C[\theta](x_{1, l}) = 0$.
    Using Theorem \ref{theorem_big_characterization}, we can now find tuples $\uv_0\uv_1 \in \VV$ with $(\mm, \theta) \models \psi'_\theta(\uv_0\uv_1) \wedge S(\uv_0\uv_1)$.
    With our \placeholderNotation{}, this is
    \begin{align}
        (\mm, \theta) \models &\ \psi(\rho[\theta](\uv_0)\uv_1) \notag \\
        & \; \wedge \bigwedge\nolimits_{k=1}^m \ulambda{}_k \cdot f^{C-1}[\theta](\uv_1) + f^{C-1}[\theta] \circ \pi_{\Ker(f^C)}(u_k) \neq 0 \wedge \bigwedge\nolimits_{l=1}^{|\ux_1|} f^C[\theta](v_{1, l}) = 0.\!\! \label{tag_second_line_in_here}
    \end{align}
    Since clearly $\rho[\theta](\uv_0) \in \Image(\rho)$, we just need to show $\ulambda{}_k \cdot \uv_1 + u_k \not\in \Image(\rho)$ for each $k$.
    First, note that the second line of (\ref{tag_second_line_in_here}) implies $f^{C-1}[\theta] \circ \pi_{\Ker(f^C)}(\ulambda{}_k \cdot \uv_1 + u_k) \neq 0$.
    Suppose we had $\ulambda{}_k \cdot \uv_1 + u_k \in \Image(\rho)$.
    Then, as $f \mid \rho$, also $\ulambda{}_k \cdot \uv_1 + u_k \in \Image(f)$, i.e., $\ulambda{}_k \cdot \uv_1 + u_k = f[\theta](v'_{k})$ for some $v'_k \in \VV$.
    Now
    $$
    f^{C-1}[\theta] \circ \pi_{\Ker(f^C)}(\ulambda{}_k \cdot \uv_1 + u_k) = f^{C-1}[\theta] \circ \pi_{\Ker(f^C)}(f[\theta](v'_k)) = f^C[\theta] \circ \pi_{\Ker(f^C)}(v'_k) = 0
    $$
    follows, contradicting $f^{C-1}[\theta] \circ \pi_{\Ker(f^C)}(\ulambda{}_k \cdot \uv_1 + u_k) \neq 0$.
    Hence, we obtain $\ulambda{}_k \cdot \uv_1 + u_k \not\in \Image(\rho)$ for all $k$, completing the proof.
\end{proof}
\end{theorem}

\noindent The following likely also holds with images instead of kernels:

\begin{remark} \label{remark_all_completios_TV}
    Suppose that $T$ satisfies \Hfour{}. There is a kernel configuration $C$ and a polynomial $\rho$ such that
    any completion of the theory $TV$ is of the form $\Th(\mm, \Ker(\rho))$ for some $(\mm, \theta) \models T\theta^C$.
\begin{proof}
    Let $(\mm, V)$ be a sufficiently saturated model of a completion of $TV$.
    Choose a basis $\bb_V$ of $V$ and a basis $\set{v^i_k : i \in \omega, k \in \kk}$ of $\VV$ over $V$.
    Define an endomorphism $\theta$ of $\VV$ by setting $\theta(v) := 0$ for all $v \in \bb_V$ and $\theta(v^i_k) := v^{i+1}_k$ for all $i \in \omega$ and $k \in \kk$.
    Then $\Ker(\rho[\theta]) = V$ for $\rho = X$.

    Recall that every endomorphism is a $C_\infty$-endomorphism, where $C_\infty$ is the unique transcendental kernel configuration with $C_\infty(f) = \infty$ for all $f \in \Kp{}$.
    Hence $(\mm, \theta) \models T^{C_\infty}_\theta$.
    Choose some $(\mm', \theta') \models T\theta^{C_\infty}$ with $(\mm, \theta) \subseteq (\mm', \theta')$.
    Clearly,
    $$
    \acl^{\mm'}(\varnothing) \cap \Ker(\rho[\theta'])
    =
    \acl^{\mm}(\varnothing) \cap \Ker(\rho[\theta])
    =
    \acl^{\mm}(\varnothing) \cap V.
    $$
    Thus Lemma \ref{lemma_completions_TV} yields
    $
    (\mm', \Ker(\rho[\theta'])) \equiv (\mm, V).
    $
\end{proof}
\end{remark}

\noindent In a future paper, we will use Theorem \ref{theorem_kernel_generic} to show that the model companions obtained in \cite{Blo23} are reducts of $T\theta^C$ (for some $C$) and the same base theory $T$. We can then combine Remark \ref{remark_all_completios_TV} with our preservation of \NATP{} result, which will also appear in a future paper, to show that the model companions obtained in \cite{Blo23} are \NATP{}.

\section{Replacing $\theta$ with $\zeta[\theta]$}

The next natural question one could ask is whether, given $(\mm, \theta) \models T\theta^C$ and $\zeta \in K[X]$, the reduct $(\mm, \zeta[\theta])$ is again a model of $T\theta^C$.
Except for a few special cases, which we will highlight later, the general answer is no.

\begin{example}
    Let $K = \QQ$. Take $\zeta = X^2$ and let $C$ be transcendental with $C(X^2 + 1) = 1$ and $C(f) = 0$ for all $f \in \QQp{} \setminus \set{X^2 +1}$.
    Given $(\mm, \theta) \models T\theta^C$, the structure $(\mm, \zeta[\theta])$ is not a model of $T\theta^C$.
    This happens because
    $$
    \Ker((X+1)^0[\zeta[\theta]]) = \set{0} \neq \Ker((X^2 +1)[\theta]) = \Ker((X+1)^1[\zeta[\theta]])
    $$
    and $C(X+1) = 0$.
\end{example}

\noindent However, in the example above, one can verify $(\mm, \zeta[\theta]) \models T\theta^{C'}$ for the unique transcendental kernel configuration $C'$ with $C'(X+1) = 1$ and $C'(f) = 0$ for all $f \in \Kp{} \setminus \set{X +1}$.
This raises the question whether we can find such a $C'$ for arbitrary $C$ and $\zeta$.
The general answer will again be no, but we will be able to specify exactly for which $C$ and $\zeta$ we have such a $C'$, and we will be able to describe $C'$ in those cases.

\subsection{The Statement}

Before we state and prove when $(\mm, \zeta[\theta])$ is an existentially closed model of $T^{C'}_\theta$ for some kernel configuration $C'$, we establish some rules for the composition of polynomials:

\begin{lemma} \label{lemma_composition_rules}
    The following holds for any $\zeta \in K[X]$:
    \begin{enumerate}[(i)]
        \item $\rho[\zeta[\theta]] = (\rho \circ \zeta)[\theta]$ for any endomorphism $\theta$;
        \item $(\rho \cdot \eta) \circ \zeta = (\rho \circ \zeta) \cdot (\eta \circ \zeta)$;
        \item if $\rho$ and $\eta$ are relatively prime, then so are $\rho \circ \zeta$ and $\eta \circ \zeta$;
        \item if $\zeta \not\in K$ and $g \in \Kp{}$ are given, then there is a unique $f \in \Kp{}$ with $g \mid (f \circ \zeta)$;
        \item if $\zeta \not\in K$ and $f \in \Kp{}$ are given, then $\deg(f) \mid \deg(g)$ for any irreducible factor $g$ of $f \circ \zeta$.
    \end{enumerate}
\begin{proof}
    Recall that $\zeta[\theta]$ is itself an endomorphism, so all the usual rules from Corollary \ref{corollary_polynomials_kernel_facts} and the paragraph after it hold with $\zeta[\theta]$ instead of $\theta$.
    \begin{enumerate}[(i)]
        \item This follows from the equality
        $
            \rho[\zeta[\theta]] := \sum\nolimits_{i=0}^{\deg(\rho)} (\rho)_i \cdot \zeta[\theta]^i = \sum\nolimits_{i=0}^{\deg(\rho)} (\rho)_i \cdot \zeta^i[\theta] = (\rho \circ \zeta)[\theta].
        $
        \item This is clear.
        \item By Bézout's Identity for polynomials, we have $\chi_1 \cdot \rho + \chi_2 \cdot \eta = 1$. Together with (ii) this implies 
        $$
        (\chi_1 \circ \zeta) \cdot (\rho\circ \zeta) + (\chi_2\circ \zeta) \cdot (\eta\circ \zeta) = (1 \circ \zeta) = 1,
        $$
        and therefore also $\gcd(\rho\circ \zeta, \eta\circ \zeta) = 1$.
        \item Take some non-trivial $\theta$ with $\Ker(g[\theta]) = \VV$.
        Notice that $\spanA{\zeta[\theta]^i : i \in \omega}{K}$ is a vector subspace of the $\deg(g)$-dimensional vector space $\spanA{\theta^i : i \in \omega}{K} = \spanA{\theta^i : 0 \leq i < \deg(g)}{K}$, so there is some $\rho \in K[X] \setminus \set{0}$ with $0 < \deg(\rho) \leq \deg(g)$ and $\rho[\zeta[\theta]] = 0$.
        Since $\zeta \not\in K$, we have $\deg(\rho \circ \zeta) > 0$ and can therefore write $\rho \circ \zeta$ in the form $q \cdot \prod\nolimits_{k=1}^m (f_k \circ \zeta)^{n_k}$ with $m > 0$, all $f_k$ irreducible, and all $n_k > 0$.
        By (iii), we now have
        $$
        \Ker(g[\theta]) = \Ker(\rho[\zeta[\theta]]) = \bigoplus\nolimits_{k=1}^m \Ker((f_k \circ \zeta)^{n_k}[\theta])
        $$
        and hence $g \mid (f_k \circ \zeta)$ for some $k$, since otherwise $\Ker(g[\theta]) \cap \Ker((f_k \circ \zeta)^{n_k}[\theta]) = \set{0}$ for all $k$.
        The uniqueness of $f$ follows from (iii).
        \item Let $g \mid (f \circ \zeta)$ be irreducible, and let $a$ be a root of $g$ in the algebraic closure of $K$.
        This implies that $\zeta(a)$ is a root of the irreducible polynomial $f$.
        We see that both $K(a)$ and $K(\zeta(a))$ are fields.
        Hence
        $$
        \deg(g) = [K(a) : K] = [K(a) : K(\zeta(a))] \cdot [K(\zeta(a)) : K] = [K(a) : K(\zeta(a))] \cdot \deg(f)
        $$
        as $K(a) \supseteq K(\zeta(a)) \supseteq K$.
        \qedhere
    \end{enumerate}
\end{proof}
\end{lemma}

\begin{remark}
    One could hope that for $f$ irreducible, the composition $f \circ \zeta$ stays irreducible or at least has only a single irreducible factor.
    However, this is false; take $f = X$ and $\zeta$ as complicated as desired.
\end{remark}

\begin{definition} \label{def_m_g_s_new_i} \label{def_C_prime_new_ii}
Let $\zeta \in K[X] \setminus K$ be given.
\begin{enumerate}[(i)]
    \item For $g \in \Kp{}$, we define a positive integer $\mf_g(\zeta)$ by setting
    $$
    \mf_g(\zeta) := \max\set{m \in \NN_{>0} : g^m \mid (f \circ \zeta)}
    $$
    for the unique $f \in \Kp{}$ with $g \mid (f \circ \zeta)$, as stated by (iv) of Lemma \ref{lemma_composition_rules}.
    \item Given a kernel configuration $C \in \Cc$, we let $C_{\zeta} \in \Cc$ denote the unique kernel configuration that is algebraic if and only if $C$ is algebraic and that satisfies
    $$
    C_\zeta(f) := \min\set{n \in \NN \cup \set{\infty} : \forall g \in \Fac(f \circ \zeta) : \mf_g(\zeta) \cdot n \geq C(g)}
    $$
    for every $f \in \Kp{}$.
\end{enumerate}
\end{definition}

\begin{theorem} \label{theorem_iterations}
    Let $\zeta \in K[X] \setminus K$ be given, and let $C, C' \in \Cc$ be two kernel configurations.
    For an existentially closed $(\mm, \theta) \models T_\theta^C$, the structure $(\mm, \zeta[\theta])$ is an existentially closed model of $T_\theta^{C'}$ if and only if $C' = C_\zeta$ and
    \begin{itemize}
        \item[$(\triangle)$] $\mf_g(\zeta) \cdot C_\zeta(f) = C(g)$ holds for all irreducible polynomials $f \in \Kp{}$ with $C_\zeta(f) > 1$ and all $g \in \Fac(f \circ \zeta) \cap \Kp{0<C}$.
    \end{itemize}
\begin{proof}
    We fix $\zeta \in K[X] \setminus K$ and an existentially closed model $(\mm, \theta) \models T^C_\theta$ throughout the proof.
    To make our notation less complicated, we set $\mf_g := \mf_g(\zeta)$ for any $g \in \Kp{}$.

    We start by showing that we must have $C' = C_\zeta$ if $(\mm, \zeta[\theta]) \models T^{C'}_\theta$ is existentially closed.
    For this, we first study how the sets $\Ker(\rho[\zeta[\theta]])$ and $\Image(\rho[\zeta[\theta]])$ can be written in terms of our original endomorphism $\theta$, for a polynomial $\rho \in K[X]$.
    Since we have $\Ker(\rho \cdot \eta) = \Ker(\rho) \oplus \Ker(\eta)$ if $\rho$ and $\eta$ are relatively prime (Corollary \ref{corollary_polynomials_kernel_facts}), it suffices to do this in the case where $\rho$ is a power of some $f \in \Kp{}$:
    \begin{subclaim}\label{claim_irred_composition_g_s}
        The following holds:
    \begin{enumerate}[(i)]
        \item We have
        $
        \Ker(f^n[\zeta[\theta]]) = \bigoplus\nolimits_{g \in \Fac(f \circ \zeta)} \Ker(g^{\min(\mf_g \cdot n, C(g))}[\theta])
        $.
        If $C_\zeta(f) < \infty$, then the following holds for any $k \geq 0$:
        $$
        \Ker(f^{C_\zeta(f)+k}[\zeta[\theta]]) = \bigoplus\nolimits_{g \in \Fac(f \circ \zeta)} \Ker(g^{C(g)}[\theta]).
        $$
        \item We have
        $
        \Image(f^n[\zeta[\theta]]) = \Image\big(\big(\prod\nolimits_{g \in \Fac(f \circ \zeta)}g^{\min(\mf_g \cdot n, C(g))}\big)[\theta]\big).
        $
        If $C_\zeta(f) < \infty$, then the following holds for any $k \geq 0$:
        $$
        \Image(f^{C_\zeta(f)+k}[\zeta[\theta]]) = \Image\Big(\Big(\prod\nolimits_{g \in \Fac(f \circ \zeta)}g^{C(g)}\Big)[\theta]\Big).
        $$
    \end{enumerate}
    \begin{innerproof}
        For (i), write $f^n \circ \zeta = (f \circ \zeta)^n = \big(\prod_{g \in \Fac(f \circ \zeta)} g^{\mf_g}\big)^n$ as the product of its irreducible factors and use the fact that $\Ker(g^C[\theta]) = \Ker(g^{C+1}[\theta])$ holds for all $g \in \Kp{C<\infty}$ (as $\theta$ is a $C$-endomorphism; see Fact \ref{remark_alg_kc} for the algebraic case).
        The equation for $\Ker(f^{C_\zeta(f)+k}[\zeta[\theta]])$ follows since $\mf_g \cdot C_\zeta(f) \geq C(g)$ holds by definition of $C_\zeta$.

        Point (ii) follows in the same way, using Fact \ref{corollary_polynomials_kernel_facts2} and $\Image(g^C) = \Image(g^{C+1})$, which holds by the $C$-image-completeness of $\theta$ (see Definition \ref{def_C_image_comple} and (ii) of Fact \ref{fact_c_image_comple}) for all $g \in \Kp{C<\infty}$.
    \end{innerproof}
    \end{subclaim}

    \begin{subclaim} \label{lemma_restrict_possible_C_prime_new}
    The following holds:
    \begin{enumerate}[(i)]
        \item $(\mm, \zeta[\theta]) \models T_\theta^{C_\zeta}$.
        \item If $(\mm, \zeta[\theta]) \models T_\theta^{C'}$ is also existentially closed, then $C' = C_\zeta$.
    \end{enumerate}
\begin{innerproof}
    We start with (i).
    First assume that $C$ is algebraic with $\mipo(C) = \prod_{k=1}^m g_k^{C(g_k)}$ (with all $g_k$ distinct and $\Kp{0<C} = \set{g_1, \dots, g_m}$).
    For each $k$, there is a unique $f_k \in \Kp{}$ with \hbox{$g_k \in \Fac(f_k \circ \zeta)$} by (iv) of Lemma \ref{lemma_composition_rules}.
    Define $F := \set{f_k : 1 \leq k \leq m}$.
    By definition, $\mf_{g_k}$ is the maximal $q > 0$ with $g_k^q \mid (f_k \circ \zeta)$, and $\mf_{g_k} \cdot C_\zeta(f_k) \geq C(g_k)$.
    Hence, we obtain
    $$
    g_k^{C(g_k)} \mid \Big(\prod\nolimits_{f \in F} f^{C_\zeta(f)}\Big) \circ \zeta
    $$
    for every $k \in \set{1, \dots, m}$.
    Now for every $f \in \Kp{} \setminus F$ we have $\Fac(f \circ \zeta) \subseteq \Kp{C=0}$ by the definition of $F$ and (iii) of Lemma \ref{lemma_composition_rules}.
    This implies $0 \cdot \mf_g \geq C(g)$ for any $g \in \Fac(f \circ \zeta)$, and therefore also $C_\zeta(f) = 0$.
    With this, we obtain $\mipo(C_\zeta) = \prod\nolimits_{f \in F} f^{C_\zeta(f)}$.
    We conclude
    $$
    \mipo(C) \mid \mipo(C_\zeta) \circ \zeta
    $$
    and hence $\mipo(C_\zeta)[\zeta[\theta]] = 0$.
    In the transcendental case, this follows directly from (i) of Claim \ref{claim_irred_composition_g_s}.

    For (ii), recall the definition of $C$-completeness (Definition \ref{def_c_complete_new}) and that every existentially closed model of $T^C_\theta$ is $C$-complete (Theorem \ref{corollary_c_comp}).
    Assume that $(\mm, \zeta[\theta]) \models T_\theta^{C'}$ is existentially closed.
    Notice that if $C$ is algebraic, then $\zeta[\theta]$ cannot be $C'$-complete if $C'$ is transcendental, as the type $\set{\text{``$(\zeta[\theta]^i(x) : i \in \omega)$ is linearly independent over $\VV$"}}$ cannot be realized.
    On the other hand, if $C$ is transcendental, then $\zeta[\theta]$ cannot be a $C'$-endomorphism if $C'$ is algebraic, as we can easily find some $v \in \VV$ with $\mipo(C')[\zeta[\theta]](v) \neq 0$.
    So far, we have shown that $C'$ is algebraic/transcendental if and only if $C$ is algebraic/transcendental.
    Assume toward a contradiction that we have $f \in \Kp{}$ with $C'(f) \neq C_\zeta(f)$.
    There are two cases:
    \begin{enumerate}[(a)]
        \item $C_\zeta(f) < C'(f)$.
        In this case, we have $\Ker(f^{C_\zeta(f)}[\zeta[\theta]]) = \Ker(f^{C_\zeta(f)+1}[\zeta[\theta]])$ by (i) (recall Fact \ref{remark_alg_kc} for $C$ algebraic), which contradicts that $\zeta[\theta]$ is $C'$-complete (see, e.g., Lemma \ref{corollary_kernels_growing}).
        \item $C_\zeta(f) > C'(f)$.
        In this case, we can use $\Ker(f^{C'(f)}[\zeta[\theta]]) = \Ker(f^{C'(f)+1}[\zeta[\theta]])$ ($\zeta[\theta]$ is a $C'$-endomorphism) and (i) of Claim \ref{claim_irred_composition_g_s} to obtain
        $$
        \bigoplus\nolimits_{g \in \Fac(f \circ \zeta)} \Ker(g^{\mf_g \cdot C'(f)}[\theta]) = \bigoplus\nolimits_{g \in \Fac(f \circ \zeta)} \Ker(g^{\mf_g \cdot (C'(f)+1)}[\theta]).
        $$
        As $C'(f) < C_\zeta(f)$, we know, by definition of $C_\zeta(f)$, that there is at least one $g \in \Fac(f \circ \zeta)$ with $\mf_g \cdot C'(f) < C(g)$.
        This implies $\Ker(g^n[\theta]) = \Ker(g^{n+1}[\theta]) = \Ker(g^{n+\mf_g}[\theta])$ for some $n < C(g)$.
        This contradicts that $\theta$ is a $C$-complete endomorphism (see again Lemma \ref{corollary_kernels_growing}).
    \end{enumerate}
    We obtain $C'(f) = C_\zeta(f)$ for all $f \in \Kp{}$.
    Since kernel configurations are completely determined by being algebraic/transcendental and the function $f \mapsto C(f)$ (see Definition \ref{def_kernel_conf}), we conclude $C' = C_\zeta$.
\end{innerproof}
\end{subclaim}

\noindent Next, we show the implication from left to right.
More precisely, we show that $(\mm, \zeta[\theta])$ cannot be an existentially closed model of $T_\theta{\!\!}^{C_\zeta}$ if \iterCond{} from Theorem \ref{theorem_iterations} fails.

\begin{subclaim}
    Suppose that there is some $f \in \Kp{}$ such that the following holds:
    \begin{enumerate}[(i)]
        \item There is some $g \in \Fac(f \circ \zeta)$ with $C(g) > 0$ such that $\mf_g \cdot C_\zeta(f) > C(g)$;
        \item $C_\zeta(f) \geq 2$.
    \end{enumerate}
    Then $(\mm, \zeta[\theta])$ is not an existentially closed model of $T_\theta{\!\!}^{C_\zeta}$.
\begin{innerproof}
    We show that $\zeta[\theta]$ cannot be $C_\zeta$-complete (see Definition \ref{def_c_complete_new}), so $(\mm, \zeta[\theta]) \models T_\theta{\!\!}^{C_\zeta}$ cannot be existentially closed by Theorem \ref{corollary_c_comp}.
    Fix $f \in \Kp{}$, as in the statement, and $g \in \Fac(f \circ \zeta)$, as described in (i).

    If we have $C_\zeta(f) = \infty$, then (i) implies $0 < C(g) < \infty$.
    Since $\theta$ is $C$-complete, we obtain $\Ker(g^C[\theta]) = \Ker(g^{C+1}[\theta])$ and some $v \in \Ker(g^C[\theta]) \setminus \Ker(g^{C-1}[\theta])$.
    If $\zeta[\theta]$ were $C_\zeta$-complete, we would obtain $v \in \Image(f[\zeta[\theta]]) \subseteq \Image(g^{\mf_g}[\theta])$, as $g^{\mf_g} \mid f \circ \zeta$ holds by definition of $\mf_g$.
    However, this implies that $\Ker(g^{C + \mf_g}[\theta]) \setminus \Ker(g^C[\theta])$ is non-empty, contradicting $\Ker(g^C[\theta]) = \Ker(g^{C+1}[\theta])$ (recall $\mf_g > 0$ and note that $\Ker(g^C[\theta]) = \Ker(g^{C+1}[\theta])$ implies $\Ker(g^C[\theta]) = \Ker(g^{C+s}[\theta])$ for any $s \geq 0$).

    Now assume $C_\zeta(f) < \infty$.
    This time, (i) implies that we have $0 < C(g) < \infty$ and \hbox{$\mf_g \cdot n + r = C(g)$} for some $n, r \in \NN$ with $n < C_\zeta(f)$ and $0 \leq r < \mf_g$.
    First, assume $n > 0$.
    Use $C$-completeness, or rather Lemma \ref{corollary_kernels_growing}, to find
    $$
    v \in \Ker(g^{n \cdot \mf_g}[\theta]) \setminus \Ker(g^{(n-1) \cdot \mf_g + r}[\theta]) \subseteq \Ker(f^{n}[\zeta[\theta]]).
    $$
    Assuming that $\zeta[\theta]$ is $C_\zeta$-complete, we can find some $v' \in \VV$ with $f[\zeta[\theta]](v') = v$, as $n < C_\zeta(f)$.
    Because $g^{\mf_g} \mid f \circ \zeta$ holds by definition of $\mf_g$, we obtain
    $$
    (f \circ \zeta/g^{\mf_g})[\theta](v') \in \Ker(g^{(n + 1) \cdot \mf_g}[\theta]) \setminus \Ker(g^{n \cdot \mf_g + r}[\theta]).
    $$
    However, by definition, $(n + 1) \cdot \mf_g > C(g) = n \cdot \mf_g + r$.
    Since $\theta$ is a $C$-endomorphism, we have $\Ker(g^{C + k}[\theta]) = \Ker(g^C[\theta])$ for any $k \geq 0$.
    Hence, no such $v'$ can exist, so $\zeta[\theta]$ cannot be $C_\zeta$-complete.
    In the case where $n = 0$, we similarly use the $C$-completeness of $\theta$ to find an element $v \in \Ker(g^r[\theta]) \setminus \set{0} \subseteq \Ker(f[\zeta[\theta]])$.
    We then use the $C_\zeta$-completeness of $\zeta[\theta]$ together with (ii), i.e., $C_\zeta(f) \geq 2$, to construct an element in
    $$
    \Ker(g^{\mf_g + r}[\theta]) \setminus \Ker(g^{\mf_g}[\theta]) = \Ker(g^C[\theta]) \setminus \Ker(g^C[\theta]) = \varnothing,
    $$
    leading to the same contradiction.
\end{innerproof}
\end{subclaim}

\noindent It remains to show the implication from right to left in Theorem \ref{theorem_iterations}.
For this, we use our axiomatization of existentially closed models of $T^C_\theta$ (and $T_\theta{\!\!}^{C_\zeta}$), i.e., Theorem \ref{theorem_big_characterization}.
The $C_\zeta$-image-completeness of $\zeta[\theta]$ follows directly from (ii) of Claim \ref{claim_irred_composition_g_s}.
It remains to show that we have
$$
(\mm, \zeta[\theta]) \models \exists \ux' \in \VV : \psi'_\theta(\ux') \wedge S'(\ux')
$$
for any $C_\zeta$-sequence-system $S'(\ux')$ over $(\VV, \zeta[\theta])$ and $L(M)$-formula $\psi'(\uxvec')$ that is bounded by $S'$ and implies no finite disjunction of non-trivial linear dependencies in $\uxvec'$ over $\VV$.

\begin{remark*}
        Notice that so far we have never really distinguished between $\theta$ as a function symbol and $\theta$ as an actual function in some structure.
        In our \placeholderNotation{} (Definition \ref{def_placeholder_notation}), we define $\psi'_\theta(\ux')$ to be the $L_\theta$-formula obtained by replacing every instance of $x_{\li, k}'^i$ and $x'^i_{\ld, k}$ with the $L_\theta$-terms $\theta^i(x'_{\li, k})$ and $\theta^i(x'_{\ld, k})$.
        To verify $(\mm, \zeta[\theta]) \models \psi'_\theta(\uv')$, we actually need to verify that $\psi'(\uxvec{}')$, with every instance of $x'^i_{\li, k}$ and $x'^i_{\ld, k}$ replaced by $\zeta[\theta]^i(v'_{\li, k})$ and $\zeta[\theta]^i(v'_{\ld, k})$, holds in $\mm$.
\end{remark*}

\noindent Our strategy is to find an $L(M)$-formula $\psi(\uxvec)$ and a $C$-sequence-system $S(\ux)$ that satisfy
$$
(\mm, \theta) \models \exists \ux\in \VV : \psi_\theta(\ux) \wedge S(\ux) \quad \Rightarrow \quad (\mm, \zeta[\theta]) \models \exists \ux'\in \VV : \psi'_\theta(\ux') \wedge S'(\ux')
$$
and use Theorem \ref{theorem_big_characterization} to conclude.
The first step is to define $S(\ux)$ from $S'(\ux')$ and prove some important properties.
We write
$$
    S'(\ux') = \bigwedge\nolimits_{k=1}^{n} f_k^{r_k}[\theta](x'_{\ldd, k}) = u'_k
$$
with $\ux' := \ux'_\li\ux'_\ld := (x'_{\lii, k} : 1 \leq k \leq m)(x'_{\ldd, k} : 1 \leq k \leq n)$ and $\uu' = (u'_1, \dots, u'_{n})$ satisfying the following conditions:
\begin{itemize}
    \item If $C_\zeta$ is algebraic, then $m = 0$;
    \item $f_k \in \Kp{0<C_\zeta}$ and $r_k \in \set{q \in \NN : 0 < q \leq C_\zeta(f_k)}$ hold for all $k \in \set{1, \dots, n}$;
    \item If $C_\zeta(f_k) < \infty$, then $u'_k \in \Ker(f_k{\!\!}^{C_\zeta-r_k}[\zeta[\theta]])$.
\end{itemize}
These conditions hold by the definition of a $C_\zeta$-sequence-system over $(\VV, \zeta[\theta])$; see Definition \ref{def_param_c_sequence_system} and Definition \ref{def_c_sequence_system}.
Throughout the rest of this proof, we also assume that \iterCond{} of Theorem \ref{theorem_iterations} holds, i.e., that $\mf_g \cdot C_\zeta(f) = C(g)$ holds for all $f \in \Kp{}$ with $C_\zeta(f) > 1$ and all $g \in \Fac(f \circ \zeta) \cap \Kp{0<C}$.

\begin{subdefinition} \label{def_iter_c_ss}
    For each $k \in \set{1, \dots, n}$, define $G_k := \Fac(f_k \circ \zeta) \cap \Kp{0<C}$ and write $G_k = \set{g_1, \dots, g_{s_k}}$.
    For $k \in \set{1, \dots, n}$ and $g \in G_k$, define $\eta_{k, g} := (f_k \circ \zeta)^{r_k} / g^{\mf_g \cdot r_k}$ and
    \begin{enumerate}[(i)]
        \item if $C_\zeta(f_k) = \infty$, set $q_{k, g} := \mf_g \cdot r_k$ and define $u_{k, g}$ such that $\eta_{k, g}[\theta](u_{k, g}) = u'_k$ if $g = g_1$ and $u_{k, g} = 0$ otherwise.
        Note that if $C_\zeta(f_k) = \infty$, then we must have $C(g) \in \set{0, \infty}$ for each $g \mid (f_k \circ \zeta)^{r_k}$ by \iterCond{}.
        By (i) and (iii) of Definition \ref{def_c_complete_new}, this implies that $(f_k \circ \zeta)^{r_k}[\theta]$, and hence also $\eta_{k, g}[\theta]$, is surjective.
        \item if $1 < C_\zeta(f_k) < \infty$, set $q_{k, g} := \mf_g \cdot r_k$ and $u_{k, g} := \eta_{k, g}[\theta]^{-1}_{\restriction \Ker(g^C[\theta])} \circ \pi_{\Ker(g^C[\theta])}(u'_{k})$.
        By definition, $\eta_{k, g}$ and $g^C$ are relatively prime, so $\eta_{k, g}[\theta]^{-1}_{\restriction \Ker(g^C[\theta])}$ is given by $\chi[\theta]$ for some $\chi \in K[X]$ (see Fact \ref{fact_inv_poly_on_ker}).
        \item if $C_\zeta(f_k) = 1$, set $q_{k, g} := C(g)$ and $u_{k, g} := 0$.
    \end{enumerate}
    Now we define $\ux_\li := (x_{\li, k} : 1 \leq k \leq m)$, $\ux_\ld := (x_{\ld, k, g} : 1 \leq k \leq n, g \in G_k)$, $\ux := \ux_\li\ux_\ld$ and
    $$
    S(\ux) := \bigwedge\nolimits_{k=1}^{n} \bigwedge\nolimits_{g \in G_k} g^{q_{k, g}}[\theta](x_{\ld, k, g}) = u_{k, g}.
    $$
\end{subdefinition}

\begin{subclaim} \label{lemma_iter_c_ss_holds}
    The formula $S(\ux)$ is a $C$-sequence-system over $(\VV, \theta)$, and given a tuple $\uv \in \VV$ with $(\mm, \theta) \models S(\uv)$, we also have $(\mm, \zeta[\theta]) \models S'(\uv')$ for the tuple $\uv' \in \VV$ defined by
    $$\uv'_\li := \uv_\li \quad  \text{and} \quad v'_{\ld, k} := \sum\nolimits_{g \in G_k} v_{\ld, k, g}.$$
\begin{innerproof}
    We first verify that $S(\ux)$ is indeed a $C$-sequence-system over $(\VV, \theta)$.
    First, it is clear that $|\ux_\li| = 0$ if $C$ is algebraic, since $|\ux_\li| = |\ux'_\li|$ and $C$ is algebraic if and only if $C_\zeta$ is algebraic.
    Every $G_k$ is a subset of $\Kp{0<C}$ by definition.
    Fix $k \in \set{1, \dots, n}$ and $g \in G_k$, and consider the following cases:
    \begin{enumerate}[(i)]
        \item $C_\zeta(f_k) = \infty$.
        This implies $C(g) = \infty$ by \iterCond{} of Theorem \ref{theorem_iterations}.
        By definition, we have $q_{k, g} = \mf_g \cdot r_k$.
        Since $r_k > 0$ and $\mf_g > 0$, we obtain $0 < q_{k, g} < C(g)$.
        \item $1 < C_\zeta(f_k) < \infty$.
        Again, $q_{k, g} = \mf_g \cdot r_k$ holds by definition.
        Since $0 < r_k \leq C_\zeta(f_k)$ holds by definition of $r_k$ and $\mf_g \cdot C_\zeta(f_k) = C(g)$ holds by \iterCond{}, we obtain $0 < q_{k, g} \leq C(g)$.
        We also obtain
        $$
        u'_k \in \Ker(f_k{\!\!}^{C_\zeta-r_k}[\zeta[\theta]]) = \bigoplus\nolimits_{g \in G_k}\!\!\! \Ker(g^{\min(\mf_g \cdot (C_\zeta(f_k)-r_k), C(g))}[\theta]) = \bigoplus\nolimits_{g \in G_k}\!\!\! \Ker(g^{C - q_{k, g}}[\theta])
        $$
        using our assumption on $u'_k$ and (i) of Claim \ref{claim_irred_composition_g_s}.
        As the map $\eta_{k, g}[\theta]^{-1}_{\restriction \Ker(g^C[\theta])}$ is an automorphism of $\Ker(g^C)$, we obtain that $u_{k, g} := \eta_{k, g}[\theta]^{-1}_{\restriction \Ker(g^C[\theta])} \circ \pi_{\Ker(g^C[\theta])}(u'_{k})$ lies in $\Ker(g^{C - q_{k, g}}[\theta])$.
        \item $C_\zeta(f_k) = 1$.
        In this case, $0 < q_{k, g} = C(g) \leq \mf_g \cdot C_\zeta(f_k) = \mf_g = \mf_g \cdot r_k$ and $u_{k, g} = 0$ hold trivially by definition.
        In particular, we obtain $u_{k, g} \in \Ker(g^{C - q_{k, g}}[\theta])$.
    \end{enumerate}
    This completes the proof that $S(\ux)$ is a $C$-sequence-system.

    Now let $\uv \in \VV$ with $(\mm, \theta) \models S(\uv)$ be given and let $\uv'$ be defined as in the statement.
    The goal is to show that $(\mm, \zeta[\theta]) \models S'(\uv')$.
    Notice that $f_k^{r_k}[\zeta[\theta]] = \big((f_k \circ \zeta)^{r_k} / g^{\mf_g \cdot r_k}\big)[\theta] \circ g^{\mf_g \cdot r_k}[\theta]$ and recall that $\eta_{k, g} = (f_k \circ \zeta)^{r_k} / g^{\mf_g \cdot r_k}$ holds by definition.
    Using $q_{k, g} = \mf_g \cdot r_k$, we obtain
    \begin{align*}
        f_k^{r_k}[\zeta[\theta]](v'_{\ld, k}) &= \sum\nolimits_{g \in G_k} \eta_{k, g}[\theta] \circ g^{\mf_g \cdot r_k}[\theta](v_{\ld, k, g}) \\
        &= \sum\nolimits_{g \in G_k} \eta_{k, g}[\theta] \circ g^{q_{k, g}}[\theta](v_{\ld, k, g}) \\&= \eta_{k, g_1}[\theta](u_{k, g_1}) + \sum\nolimits_{l=2}^{s_k} \eta_{k, g_l}[\theta](0) = \eta_{k, g_1}[\theta](u_{k, g_1}) = u'_k
    \end{align*}
    in case (i), i.e., $C_\zeta(f_k) = \infty$.
    Similarly, and with $u'_k \in \bigoplus_{g \in G_k} \Ker(g^C[\theta])$, we obtain
    \begin{align*}
    f_k^{r_k}[\zeta[\theta]](v'_{\ld, k}) &= \sum\nolimits_{g \in G_k} \eta_{k, g}[\theta] \circ g^{q_{k, g}}[\theta](v_{\ld, k, g})\\ &= \sum\nolimits_{g \in G_k} \eta_{k, g}[\theta] \circ \eta_{k, g}[\theta]^{-1}_{\restriction \Ker(g^C[\theta])} \circ \pi_{\Ker(g^C[\theta])}(u'_{k}) \\& = \sum\nolimits_{g \in G_k} \pi_{\Ker(g^C[\theta])}(u'_k) = u'_k
    \end{align*}
    in case (ii), i.e., $1 < C_\zeta(f_k) < \infty$.
    In case (iii), i.e., $C_\zeta(f_k) = 1$, we easily obtain
    \begin{align*}
        f_k^{r_k}[\zeta[\theta]](v'_{\ld, k}) &= \sum\nolimits_{g \in G_k} \eta_{k, g}[\theta] \circ g^{\mf_g \cdot r_k}[\theta](v_{\ld, k, g}) \\&= \sum\nolimits_{g \in G_k} \eta_{k, g}[\theta] \circ g^{\mf_g \cdot r_k - q_{k, g}}[\theta](g^{q_{k, g}}[\theta](v_{\ld, k, g})) \\&= \sum\nolimits_{g \in G_k} \eta_{k, g}[\theta] \circ g^{\mf_g \cdot r_k - q_{k, g}}[\theta](0) = 0 = u'_k
    \end{align*}
    as $\mf_g \cdot r_k \geq q_{k, g}$ and $u_{k, g} = 0$ hold for all $g \in G_k$, and as the element $u'_k$ lies in $\Ker(f_k{\!\!}^{C_\zeta- r_k}[\zeta[\theta]]) = \Ker(f_k^0[\zeta[\theta]]) = \set{0}$.
    We conclude that $(\mm, \zeta[\theta]) \models S'(\uv')$.
\end{innerproof}
\end{subclaim}

\noindent Recall that our strategy is to find an $L(M)$-formula $\psi(\uxvec)$ and a $C$-sequence-system $S(\ux)$ that satisfy
$$
(\mm, \theta) \models \exists \ux\in \VV : \psi_\theta(\ux) \wedge S(\ux) \quad \Rightarrow \quad (\mm, \zeta[\theta]) \models \exists \ux'\in \VV : \psi'_\theta(\ux') \wedge S'(\ux')
$$
and use Theorem \ref{theorem_big_characterization} to conclude.
We already have the $C$-sequence-system.
Before we define the $L(M)$-formula $\psi(\uxvec{})$, note that if we want to apply Theorem \ref{theorem_big_characterization}, we must also ensure that $\psi(\uxvec)$ is bounded by $S(\ux)$ and implies no finite disjunction of non-trivial linear dependencies in $\uxvec$ over $\VV$.
The following claim will help with the latter:

\begin{subclaim} \label{lemma_independent_2}
    The tuple $((X^i \circ \zeta : g \in G_k) : 0 \leq i < \deg(f_k^{r_k}))$ is $K$-linearly independent in the vector space $\bigoplus_{g\in G_k} K[X]/(g^{q_{k, g}})$.
\begin{innerproof}
    For easier notation, define $f = f_k$, $r := r_k$, $G := G_k$, and $q_g := q_{k, g}$ for all $g \in G$.
    By the Chinese remainder theorem, we only need to show that the tuple $(X^i \circ \zeta : 0 \leq i < \deg(f^r))$ is linearly independent in $K[X]/(\prod_{g\in G}g^{q_g})$.
    Suppose that the opposite is true.
    Then there are $\xi \in K[X] \setminus \set{0}$ with $\deg(\xi) < \deg(f^r)$ and $\chi \in K[X]$ such that
    $$
    \xi \circ \zeta = \chi \cdot \prod\nolimits_{g\in G} g^{q_g}
    $$
    holds.
    Without loss, we can assume that $\xi$ is monic, so we can write $\xi = \prod_{i=1}^q h_i^{l_i}$ with $h_i \in \Kp{}$ and $l_i > 0$ for all $i \in \set{1, \dots, q}$.
    With (ii) of Lemma \ref{lemma_composition_rules}, and the equation above, we obtain
    $$
    \prod\nolimits_{g\in G} g^{q_g} \mid \prod\nolimits_{i=1}^q (h_i \circ \zeta)^{l_i}.
    $$
    Take any $g \in G$.
    By definition of $G$ (see Definition \ref{def_iter_c_ss} and recall $G = G_k$) and (iv) of Lemma \ref{lemma_composition_rules}, $f$ is the only monic irreducible polynomial $f' \in \Kp{}$ with $g \mid f' \circ \zeta$.
    Hence, there must be some $i_0 \in \set{1, \dots, q}$ with $h_{i_0} = f$.
    By the definition of $\mf_g$, we have
    $$
    g^{\mf_g} \mid f \circ \zeta \quad \text{and} \quad g^{\mf_g +1} \nmid f \circ \zeta,
    $$
    so we must have $l_{i_0} \geq \lceil q_g/ \mf_g\rceil$.
    In the case $C_\zeta(f) > 1$, we have $q_g = \mf_g \cdot r$ by Definition \ref{def_iter_c_ss}, so $f^r \mid \xi$.
    In the case $C_\zeta(f) = 1$, we still have $\lceil q_g/ \mf_g\rceil \geq 1$, as $q_g \neq 0$.
    Since $0 < r \leq C_\zeta(f) = 1$, we also obtain $f^r \mid \xi$ in this case.
    However, in both cases $f^r \mid \xi$ contradicts that $\xi \in K[X] \setminus \set{0}$ satisfies $\deg(\xi) < \deg(f^r)$.
\end{innerproof}
\end{subclaim}

\noindent We will now define our $L(M)$-formula $\psi(\uxvec{})$:

\begin{subdefinition} \label{def_iter_fml}
    For all $k, g, i$ with $1 \leq k \leq n$, $g \in G_k$, and $0 \leq i < \deg(f_k^{r_k})$, we define the polynomials $\rf_{k, g, i}, \chi_{k, g, i} \in K[X]$ such that
    $$
    \deg(\rf_{k, g, i}) < \deg(g^{q_{k, g}}) \quad \text{and} \quad \chi_{k, g, i} \cdot g^{q_{k, g}} + \rf_{k, g, i} = X^i \circ \zeta.
    $$
    We define $\psi(\uxvec)$ to be the $L(M)$-formula obtained by taking $\psi'(\uxvec{}')$ and replacing
    \begin{enumerate}[(i)]
        \item each occurrence of $x'^i_{\li, k}$ with $(X^i \circ \zeta)[x_{\li, k}]$;
        \item each occurrence of $x'^i_{\ld, k}$ with $\sum_{g \in G_k} \big(\rf_{k, g, i}[x_{\ld, k, g}] + \chi_{k, g, i}[\theta](u_{k, g})\big)$.
    \end{enumerate}
    Here we set $\rho[x_{\li, k}] := \sum_{i=0}^{\deg(\rho)} (\rho)_i \cdot x^i_{\li, k}$, and so on.
    Also, recall that each $x^i_{\li, k}$ is the placeholder variable for $\theta^i(x_{\li, k})$ from our \placeholderNotation{}.
\end{subdefinition}

\begin{subclaim}
    The following holds:
    \begin{enumerate}[(i)]
        \item The formula $\psi(\uxvec)$ is bounded by the $C$-sequence-system $S(\ux)$ from Definition \ref{def_iter_c_ss}.
        \item The formula $\psi(\uxvec)$ implies no finite disjunction of non-trivial linear dependencies in $\uxvec$ over $\VV$.
    \end{enumerate}
\begin{innerproof}
    Point (i) is clear, as $\deg(\rf_{k, g, i}) < \deg(g^{q_{k, g}})$, so each $x^j_{\ld, k, g}$ can only appear in the formula $\psi(\uxvec)$ for $j < \deg(g^{q_{k, g}})$.
    For point (ii), note that
    \begin{enumerate}[(a)]
        \item for each $k \in \set{1, \dots, m}$, the sequence of $\LK(\VV)$-terms
        $$
        B_{(a), k} := \big((X^i \circ \zeta)[x_{\li, k}] : i \in \omega\big)
        $$
        is linearly independent over $\VV$.
        To see this, recall that $\zeta \in K[X] \setminus K$, so we obtain $\deg(X^i \circ \zeta) < \deg(X^j \circ \zeta)$ if and only if $i < j$.
        \item for each $k \in \set{1, \dots, n}$, the sequence of $\LK(\VV)$-terms
        $$
        B_{(b), k} :=\Big(\sum\nolimits_{g \in G_k} \big(\rf_{k, g, i}[x_{\ld, k, g}] + \chi_{k, g, i}[\theta](u_{k, g})\big) : 0 \leq i < \deg(f_k^{r_k})\Big)
        $$
        is linearly independent over $\VV$.
        Without loss, we can replace all $\chi_{k, g, i}[\theta](u_{k, g})$ subterms with $0$ and check linear independence over the empty set.
        As each $\rf_{k, g, i}$ is the remainder of $(X^i \circ \zeta)$ divided by $g^{q_{k, g}}$, this follows directly from the linear independence of the tuple $((X^i \circ \zeta : g \in G_k) : 0 \leq i < \deg(f_k^{r_k}))$ in $\bigoplus_{g\in G_k} K[X]/(g^{q_{k, g}})$ (see Claim \ref{lemma_independent_2}).
    \end{enumerate}
    Now the concatenation $B$ of all these $B_{(a), k}$'s and $B_{(b), k}$'s is linearly independent over $\VV$, as the sets of used variables are disjoint for different sequences.
    Since $\psi'(\uxvec{}')$ is bounded by $S'$, the variable $x'^i_{\ld, k}$ can only appear in $\psi'(\uxvec{}')$ for $i < \deg(f_k^{r_k})$, so $B$ contains all the terms with which we replace in the definition of $\psi(\uxvec)$ (see Definition \ref{def_iter_fml} above).
    Also, note that every variable in $\uxvec{}'$ gets replaced in that definition.
    Since $\psi'(\uxvec{}')$ implies no finite disjunction of non-trivial linear dependencies in $\uxvec{}'$ over $\VV$, Corollary \ref{corollary_iter_lin_indp} yields that $\psi(\uxvec{})$ also implies no finite disjunction of non-trivial linear dependencies in $\uxvec{}$ over $\VV$.
\end{innerproof}
\end{subclaim}

\noindent By our axiomatization of existentially closed models (Theorem \ref{theorem_big_characterization}), there is a tuple $\uv \in \VV$ such that $(\mm, \theta) \models \psi_\theta(\uv) \wedge S(\uv)$.
Fix such a tuple $\uv = \uv_\li \uv_\ld$ and recall that $\uv_\li = (v_{\li, k} : 1 \leq k \leq m)$ and $\uv_\ld = (v_{\ld, k, g} : 1 \leq k \leq n, g \in G_k)$ by the definition of $\ux$ in Definition \ref{def_iter_c_ss}.

\begin{subclaim}
    Given $\uv \in \VV$ with $(\mm, \theta) \models \psi_\theta(\uv) \wedge S(\uv)$, we have $(\mm, \zeta[\theta]) \models \psi'_\theta(\uv') \wedge S'(\uv')$ for the tuple $\uv'$ defined by
    $$\uv'_\li := \uv_\li \quad  \text{and} \quad v'_{\ld, k} := \sum\nolimits_{g \in G_k} v_{\ld, k, g}.$$
\begin{innerproof}
    That $(\mm, \zeta[\theta]) \models S'(\uv')$ holds was shown in Claim \ref{lemma_iter_c_ss_holds}.
    We know that $\psi_\theta(\uv)$ holds in $(\mm, \theta)$ and, by the definition of $\psi(\uxvec)$ and our \placeholderNotation{}, we see that $\psi_\theta(\uv)$ is $\psi'(\uxvec{}')$, but with
    \begin{enumerate}[(i)]
        \item each occurrence of $x'^i_{\li, k}$ replaced by $(X^i \circ \zeta)[\theta](v_{\li, k}) = \zeta[\theta]^i(v'_{\li, k})$ (see Lemma \ref{lemma_composition_rules} and recall $v'_{\li, k} = v_{\li, k}$).
        \item each occurrence of $x'^i_{\ld, k}$ replaced by $\sum_{g \in G_k} \big(\rf_{k, g, i}[\theta](v_{\ld, k, g}) + \chi_{k, g, i}[\theta](u_{k, g})\big)$.
        Since we have $(\mm, \theta) \models S(\uv)$, we also obtain $u_{k, g} = g^{q_{k, g}}[\theta](v_{\ld, k, g})$.
        Recall $\rf_{k, g, i} + \chi_{k, g, i} \cdot g^{q_{k, g}} = (X^i \circ \zeta)$, by the definition of $\rf_{k, g, i}$ and $\chi_{k, g, i}$, so we obtain
        \begin{align*}
            &\sum\nolimits_{g \in G_k} \big(\rf_{k, g, i}[\theta](v_{\ld, k, g}) + \chi_{k, g, i}[\theta](u_{k, g})\big)\\ =& \sum\nolimits_{g \in G_k} \big(\rf_{k, g, i}[\theta](v_{\ld, k, g}) + \chi_{k, g, i}[\theta](g^{q_{k, g}}[\theta](v_{\ld, k, g}))\big) \\
            =& \sum\nolimits_{g \in G_k} (\rf_{k, g, i} + \chi_{k, g, i} \cdot g^{q_{k, g}})[\theta](v_{\ld, k, g}) \\
            =& \sum\nolimits_{g \in G_k} (X^i \circ \zeta)[\theta](v_{\ld, k, g}) \\
            =&\; (X^i \circ \zeta)[\theta]\Big(\sum\nolimits_{g \in G_k} v_{\ld, k, g}\Big) \\
            =&\; \zeta[\theta]^i(v'_{\ld, k}),
        \end{align*}
        and hence each occurrence of $x'^i_{\ld, k}$ is actually replaced by $\zeta[\theta]^i(v'_{\ld, k})$.
    \end{enumerate}
    By our \placeholderNotation{} but with $\zeta[\theta]$ instead of $\theta$, we now see that
    $(\mm, \zeta[\theta]) \models \psi'_\theta(\uv')$, which completes the proof.
\end{innerproof}
\end{subclaim}

\noindent Recall that $S'(\ux')$ was an arbitrary $C_\zeta$-sequence-system over $(\VV, \zeta[\theta])$ and $\psi'(\uxvec')$ an arbitrary $L(M)$-formula that is bounded by $S'$ and implies no finite disjunction of non-trivial linear dependencies in $\uxvec'$ over $\VV$.
We have shown that $(\mm, \zeta[\theta]) \models \psi'_\theta(\uv') \wedge S'(\uv')$.
As noted earlier, $\zeta[\theta]$ is $C_\zeta$-image-complete and hence $(\mm, \zeta[\theta])$ is an existentially closed model of $T_\theta^{C_\zeta}$ by Theorem \ref{theorem_big_characterization}.
This finishes the proof of the direction from right to left of Theorem \ref{theorem_iterations}.
As we already proved the other direction, this finishes the entire proof of Theorem \ref{theorem_iterations}.
\end{proof}
\end{theorem}

\noindent Note that in the statement (and in the entire proof) of Theorem \ref{theorem_iterations} we do not assume that $T$ satisfies \Hfour{} or that the model companion of $T^C_\theta$ exists.

\subsection{Consequences}

We now present a few simpler criteria for when $(\mm, \zeta[\theta])$ is again an existentially closed model of $T_\theta{\!\!}^{C_\zeta}$.
Here $C_\zeta$ is the kernel configuration from Definition \ref{def_C_prime_new_ii}, which, by Theorem \ref{theorem_iterations}, is the only kernel configuration $C' \in \Cc$ for which $(\mm, \zeta[\theta])$ can be an existentially closed model of $T^{C'}_\theta$.

\begin{corollary} \label{corollary_iteration_crit}
    Let $(\mm, \theta) \models T_\theta^C$ be existentially closed.
    The following holds:
    \begin{enumerate}[(i)]
        \item The following conditions imply that $(\mm, \zeta[\theta]) \models T_\theta{}^{\!\!C_\zeta}$ is existentially closed for any polynomial $\zeta \in K[X] \setminus K$:
        \begin{enumerate}[(a)]
            \item $C(g) \in \set{0, 1}$ for all $g \in \Kp{}$;
            \item $C(g) \in \set{0, \infty}$ for all $g \in \Kp{}$;
            \item there is a unique $g \in \Kp{}$ with $C(g) = 2$, and $C(g') = 0$ for all $g' \in \Kp{}$ with $g' \neq g$.
        \end{enumerate}
        \item Let $c_0 \in \set{0, 1, \infty}$ be given.
        If $C$ is transcendental with $C(f) = c_0$ for all $f \in \Kp{}$, then $(\mm, \zeta[\theta]) \models T_\theta^C$ is existentially closed.
        This also implies $C_\zeta = C$ for any polynomial $\zeta \in K[X] \setminus K$.
        \item If $C$ is algebraic with $\mipo(C)$ irreducible, then $(\mm, \zeta[\theta]) \models T_\theta{}^{\!\!C_\zeta}$ is existentially closed for any $\zeta \in K[X] \setminus K$.
        Moreover, $\mipo(C_\zeta)$ is irreducible, $\deg(\mipo(C_\zeta)) \mid \deg(\mipo(C))$, and $\mipo(C_\zeta)$ is the unique $f \in \Kp{}$ with $\mipo(C) \mid (f \circ \zeta)$.
    \end{enumerate}
    \begin{proof}
    As in the proof of Theorem \ref{theorem_iterations}, we fix $\zeta \in K[X] \setminus K$ and write $\mf_g$ instead of $\mf_g(\zeta)$.
    For (i), we only need to verify that \iterCond{} from Theorem \ref{theorem_iterations} holds, i.e., that $\mf_g \cdot C_\zeta(f) = C(g)$ holds for all $f \in \Kp{}$ with $C_\zeta(f) > 1$ and all $g \in \Fac(f \circ \zeta) \cap \Kp{0<C}$:
        \begin{enumerate}[(a)]
            \item Notice that $\mf_g \cdot 1 \geq 1 \geq C(g)$ holds for all $f \in \Kp{}$ and $g \in \Fac(f \circ \zeta)$ (recall $\mf_g \geq 1$).
            By the definition of $C_\zeta$, this implies $C_\zeta(f) \leq 1$ for all $f \in \Kp{}$.
            This shows that \iterCond{} holds.
            \item We show that $\mf_g \cdot C_\zeta(f) = C(g)$ holds for every $f \in \Kp{}$ and $g \in \Fac(f \circ \zeta) \cap \Kp{0<C}$.
            By our assumption, we have $C(g) = \infty$, which implies $C_\zeta(f) = \infty$.
            Thus the equation $\mf_g \cdot C_\zeta(f) = \mf_g \cdot \infty = \infty = C(g)$ holds.
            \item Let $g_0$ be the unique element of $\Kp{}$ with $C(g_0) = 2$.
            There is a unique $f_0 \in \Kp{}$ with $g_0 \mid (f_0 \circ \zeta)$ by (iv) of Lemma \ref{lemma_composition_rules}.
            Since $\Fac(f_0 \circ \zeta) \setminus \set{g_0}$ is a subset of $\Kp{C=0}$, we have $C_\zeta(f_0) = 2$ if $\mf_{g_0} = 1$ and $C_\zeta(f_0) = 1$ if $\mf_{g_0} > 1$.
            In the case $\mf_{g_0} = 1$, we have $\mf_{g_0} \cdot C_\zeta(f_0) = C(g_0)$.
            For any $f' \in \Kp{} \setminus \set{f_0}$, we have $\Fac(f' \circ \zeta) \subseteq \Kp{C=0}$, so we obtain $C_\zeta(f') = 0$.
            Combining everything, we see that \iterCond{} holds.
        \end{enumerate}
    For (ii), notice that (i) yields that $(\mm, \zeta[\theta]) \models T_\theta{}^{\!\!C_\zeta}$ is existentially closed, so we only need to verify $C = C_\zeta$.
    For this, we work directly with the definition of $C_\zeta$ (see (ii) of Definition \ref{def_C_prime_new_ii}):
    \begin{enumerate}[(a)]
        \item If $c_0 = 0$, then $\mf_g \cdot 0 = c_0 = C(g)$ holds trivially for any irreducible polynomial $g \in \Kp{}$.
        This implies $C_\zeta(f) := \min\set{n \in \NN \cup \set{\infty} : \forall g \in \Fac(f \circ \zeta) : \mf_g \cdot n \geq C(g)} = 0 = c_0$ for all $f \in \Kp{}$.
        \item If $c_0 = 1$, then $\mf_g \cdot 1 \geq c_0 = 1 > \mf_g \cdot 0$ for any $g \in \Kp{}$.
        This similarly implies $C_\zeta(f) = 1 = c_0$ for all $f \in \Kp{}$.
        \item If $c_0 = \infty$, then $\mf_g \cdot \infty = c_0 > \mf_g \cdot n$ for any $g \in \Kp{}$ and $n \in \NN$.
        Again, this implies $C_\zeta(f) = \infty = c_0$ for all $f \in \Kp{}$.
    \end{enumerate}
    We now show (iii).
    For this, notice that for any $g \in \Kp{}$ we have $C(g) = 1$ if and only if $g = \mipo(C)$, and $C(g) = 0$ otherwise.
    By (i), this tells us that $(\mm, \zeta[\theta]) \models T_\theta{}^{\!\!C_\zeta}$ is existentially closed.
    We will now compute $C_\zeta$ for the moreover part.
    Notice that there is a unique $f \in \Kp{}$ with $\mipo(C) \mid (f \circ \zeta)$.
    For this $f$, we have $C_\zeta(f) = 1$, since
    $$\mf_{\mipo(C)} \cdot 1 \geq C(\mipo(C)) > \mf_{\mipo(C)} \cdot 0$$
    and $\mf_g \cdot 0 = C(g)$ for any $g \in \Fac(f \circ \zeta) \setminus \set{\mipo(C)}$.
    For any $f' \in \Kp{}$ with $f' \neq f$, we have $\Fac(f' \circ \zeta) \subseteq \Kp{C=0}$ and therefore $C_\zeta(f') = 0$.
    This tells us that $\mipo(C_\zeta) = f$, and by (v) of Lemma \ref{lemma_composition_rules}, we obtain $n \cdot \deg(f) = \deg(\mipo(C))$ for some $n \geq 1$.
    \end{proof}
\end{corollary}

\noindent Assume $K = \QQ$.
Given algebraic kernel configurations $C_1, \dots, C_n$ with each $\mipo(C_k)$ irreducible, (iii) of Corollary \ref{corollary_iteration_crit} allows us to find another algebraic kernel configuration $C$ with $\mipo(C)$ irreducible such that each existentially closed model of $T_\theta^{C}$ has an existentially closed model of each $T_\theta^{C_k}$ as a reduct.

\begin{remark}
    Let $f_1, \dots, f_n \in \QQp{}$ be given with $\deg(f_k) > 1$ for all $k \in \set{1, \dots, n}$.
    For each $k \in \set{1, \dots, n}$, define $C_k$ to be the unique algebraic kernel configuration that satisfies $\mipo(C_k) = f_k$.

    Then there is an algebraic kernel configuration $C$ given by $\mipo(C) = g$ for some irreducible polynomial $g \in \QQp{}$, and there are $\zeta_1, \dots, \zeta_n \in \QQ[X]$, such that for any existentially closed model $(\mm, \theta) \models T^C_\theta$, the structure $(\mm, \zeta_k[\theta])$ is an existentially closed model of $T^{C_k}_\theta$ for each $k \in \set{1, \dots, n}$.
\begin{proof}
    Choose $b_1, \dots, b_n$ such that each $b_k \in \QQ^{\operatorname{alg}}$ is a root of the respective $f_k$.
    By the primitive element theorem, there is $a \in \QQ^{\operatorname{alg}}$ such that $\QQ(b_1, \dots, b_n) = \QQ(a) \simeq \QQ[X]/(g)$, where $g \in \QQp{}$ is the minimal polynomial of $a$ over $\QQ$.
    This means that for each $k \in \set{1, \dots, n}$ there is $\zeta_k$ such that $b_k = \zeta_k(a)$.
    Now, as $0 = f_k(b_k) = f_k(\zeta_k(a))$, we see that $a$ is a root of $f_k \circ \zeta_k$; hence $g \mid (f_k \circ \zeta_k)$.
    Let $C$ be the unique algebraic kernel configuration with $\mipo(C) = g$, and let $(\mm, \theta) \models T^C_\theta$ be existentially closed.
    By (iii) of Corollary \ref{corollary_iteration_crit}, we obtain that $(\mm, \zeta_k[\theta]) \models T^{C_k}_\theta$ is existentially closed for every $k \in \set{1, \dots, n}$.
\end{proof}
\end{remark}

\noindent As noted in the introduction, d'Elb\'ee introduced the theory of an algebraically closed field with a generic multiplicative endomorphism, called $\operatorname{ACFH}$, in \cite{dEl25}.
As discussed in Remark 2.20 in \cite{Chi25}, the author believes that most results also apply to algebraically closed fields with a multiplicative endomorphism, even though multiplication is only a ($\QQ$-)vector space after dividing by torsion (which is only a $\bigvee$-definable group).
Note that the theory $\operatorname{ACFH}$ would correspond to the kernel configuration $C_\infty$.
Hence (ii) of Corollary \ref{corollary_iteration_crit} with $c_0 = \infty$ likely also applies to $\operatorname{ACFH}$.
If so, then we would have $(F, \zeta[\theta]) \models \operatorname{ACFH}$ for any $(F, \theta) \models \operatorname{ACFH}$ and $\zeta \in \ZZ[X] \setminus \ZZ$, answering Question 5.16 in \cite{dEl25} positively.

We will now show that there are no better criteria for $(\mm, \zeta[\theta]) \models T_\theta{}^{\!\!C_\zeta}$ being existentially closed than those in (i) of Corollary \ref{corollary_iteration_crit}, if we do not impose any additional conditions on $\zeta \in K[X]\setminus K$.
For this, we no longer fix $\zeta$:

\begin{corollary}
    If there is either
    \begin{enumerate}[(i)]
        \item some $g \in \Kp{}$ with $2 < C(g)<\infty$;
        \item some distinct $f, g \in \Kp{}$ with $1 < C(f) = C(g) < \infty$;
        \item some $f, g \in \Kp{}$ with $0 < C(f)<C(g)$;
    \end{enumerate}
    then we can choose $\zeta$ such that $(\mm, \zeta[\theta]) \models T_\theta{}^{\!\!C_\zeta}$ is not existentially closed for any existentially closed $(\mm, \theta) \models T_\theta^C$.
\begin{proof}
    We deal with each case individually:
    \begin{enumerate}[(i)]
        \item We set $\zeta = g^{C(g)-1}$.
        Clearly, we have $X \circ \zeta = g^{C(g)-1}$.
        This means $\mf_g = C(g)-1$, and because $C(g) > 2$, we obtain $(C(g)-1) \cdot 2 > C(g) > (C(g)-1) \cdot 1$.
        Thus $C_\zeta(X) = 2$ (notice that there is no other $g' \in \Kp{}$ with $g' \mid X \circ \zeta$) and $\mf_g \cdot C_\zeta(X) \neq C(g)$.
        This means that the irreducible polynomial $f := X$ witnesses that condition \iterCond{} of Theorem \ref{theorem_iterations} fails.
        \item We set $\zeta = f \cdot g^2$, so we have $X \circ \zeta = f \cdot g^2$, i.e., $\mf_f = 1$ and $\mf_g = 2$.
        One can check that we must have $C_\zeta(X) = C(f) > 1$, and therefore $\mf_g \cdot C_\zeta(X) = 2 \cdot C(f) = 2 \cdot C(g) \neq C(g)$.
        This again means that the irreducible polynomial $X$ witnesses that \iterCond{} fails.
        \item This time, we set $\zeta = f \cdot g$.
        Again, this implies $X \circ \zeta = f \cdot g$.
        We can now verify $C_\zeta(X) = C(g) \geq 2$; however, $\mf_f \cdot C_\zeta(X) = 1 \cdot C_\zeta(X) = C(g) > C(f)$.
        We conclude that \iterCond{} fails. \qedhere
    \end{enumerate}
\end{proof}
\end{corollary}

\noindent For the remainder of this section, we assume \Hfour{} for simplicity, so we can write ``$(\mm, \theta) \models T\theta^C$'' instead of ``let $(\mm, \theta) \models T^C_\theta$ be existentially closed''.

\begin{question}\label{quest_fail_iter_cond}
    Given $(\mm, \theta) \models T\theta^C$, what theory axiomatizes $\Th((\mm, \zeta[\theta]))$ if condition \iterCond{} from Theorem \ref{theorem_iterations} fails?
\end{question}

\noindent We take a quick look at one example with $(\mm, \zeta[\theta]) \not\models T\theta^{C_\zeta}$, which roughly represents how the author expects $(\mm, \zeta[\theta])$ to look: ``almost'' a model of $T\theta^{C_\zeta}$ with some kernels $\Ker(\rho[\theta])$ definable for some $\rho \in K[X]$ not of the form $\rho' \circ \zeta$.

\begin{example}
    Let $C$ be algebraic with $\mipo(C) = X^3$, and let $(\VV, \theta) \models \TKvs\theta^C$ be given.
    For $\zeta := X^2$, one can show $C_\zeta(X) = 2$ and that
    $$
    \Ker(X[\theta]) := \set{x \in \VV : x \in \Ker(X[\zeta[\theta]]) \wedge \exists y \in \Ker(X^2[\zeta[\theta]]) : X[\zeta[\theta]](y) = x}
    $$
    is definable in $(\VV, \zeta[\theta])$, although $X$ is not of the form $\rho \circ \zeta$.
    One can show that condition (ii) of $C_\zeta$-completeness (see Definition \ref{def_c_complete_new}) fails for $u \in \Ker(X^2[\theta]) \setminus \Ker(X[\theta]) \subseteq \Ker(X[\zeta[\theta]])$, as we have $\Ker(X^4[\theta]) = \Ker(X^3[\theta])$.
\end{example}

\noindent As proven in the previous section, these kernels are generic, so the theory of $(\mm, \zeta[\theta])$ might roughly look like a reasonable combination of parts of $T\theta^{C_\zeta}$ and, in general, multiple copies of $TV$ (see Corollary \ref{corollary_t_v_companion}).

One last question one might ask is the following: Given $(\mm, \theta) \models T\theta^C$ and $\zeta \in K[X] \setminus K$, is there some endomorphism $\theta'$ with $\zeta[\theta'] = \theta$?
We give a positive answer in one simple case, although the author does not expect this to be true in general.

\begin{remark}[$\zeta$-Roots of $\theta$]
    Let $C = C_0$, i.e., let $C$ be the unique transcendental kernel configuration with $C(f) = 0$ for all $f \in \Kp{}$.
    Also let $(\VV, \theta) \models \TKvs\theta^C$ be a model of cardinality at least $\max\set{\aleph_0, |K|}{}^+$.
    By (ii) of Corollary \ref{corollary_iteration_crit}, we know that $(\VV, \zeta[\theta])$ is also a model of $\TKvs\theta^C$.
    Since $\TKvs\theta^C$ is just the theory of $K(X)$-vector spaces (see Example \ref{example_kvs}), our cardinality assumption on $\VV$ implies that there is an $\LKThe$-isomorphism
    $$
    \sigma \colon (\VV, \theta) \to (\VV, \zeta[\theta]).
    $$
    Define $\theta' := \sigma^{-1} \circ \theta \circ \sigma$.
    Using
    $$\zeta[\theta'](v) = \sigma^{-1} \circ \zeta[\theta] \circ \sigma(v) = \sigma^{-1}(\zeta[\theta](\sigma(v))) = \sigma^{-1}(\sigma(\theta(v))) = \theta(v),$$
    we obtain
    $
    (\VV, \zeta[\theta']) = (\VV, \theta).
    $
    Hence, we have found another endomorphism $\theta'$ of the vector space $\VV$ with $(\VV, \theta') \models \TKvs\theta^C$ and $\zeta[\theta'] = \theta$.
    We call such a $\theta'$ a \textbf{$\bm{\zeta}$-root} of $\theta$.
\end{remark}

\begin{question}\label{quest_zeta_roots}
    For which fields $K$, ($K$-)kernel configurations $C$, theories $T$, models of $T\theta^C$, and polynomials $\zeta \in K[X] \setminus K$ can we find a $\zeta$-root of $\theta$?
\end{question}

\noindent Recall that $R_C$ is a ring of endomorphisms definable in $T\theta^C$.
Another question is what happens if we replace $\theta$ not by $\zeta[\theta]$, but by some $r \in R_C$:

\begin{question}\label{quest_what_happens_with_r_instead_theta}
    What theory axiomatizes $\Th((\mm, r))$ given $(\mm, \theta) \models T\theta^C$ and $r \in R_C$?
\end{question}

\bibliographystyle{alpha}
\bibliography{sample}

\Addresses

\end{document}